\documentclass[a4paper,10pt]{article}
\usepackage{eurosym}
\usepackage{makeidx}
\usepackage{amsfonts}
\usepackage{amsmath}
\usepackage[UKenglish]{babel}
\usepackage{graphicx}
\usepackage{amssymb}
\usepackage{mathrsfs}
\usepackage{graphicx}
\usepackage[all,cmtip]{xy}
\usepackage{synttree}
\usepackage{tikz}
\usepackage{color}

\providecommand{\U}[1]{\protect\rule{.1in}{.1in}}
\newtheorem{theorem}{Theorem}[section]

\newtheorem{corollary}[theorem]{Corollary}

\newtheorem{definition}[theorem]{Definition}
\newtheorem{example}[theorem]{Example}
\newtheorem{lemma}[theorem]{Lemma}

\newtheorem{proposition}[theorem]{Proposition}

\newtheorem{remark}[theorem]{Remark}

\newenvironment{proof}[1][Proof]{\noindent \emph{#1.} }{\hfill \ \rule{0.5em}{0.5em}}
\newcommand{\BIGOP}[1]{\mathop{\mathchoice{\raise-0.22em\hbox{\huge $#1$}} {\raise-0.05em\hbox{\Large $#1$}}{\hbox{\large $#1$}}{#1}}}

\makeatletter\@addtoreset{equation}{section}\makeatother
\makeatletter\@addtoreset{figure}{section}\makeatother
\makeatletter\@addtoreset{table}{section}\makeatother
\begin{document}

\title{Geometric Structures in Tensor Representations}
\author{Antonio Falc\'{o}\footnote{Corresponding author}$^{\, \, \,,1,3},$ Wolfgang Hackbusch$^{2}$ and Anthony Nouy$%
^{3}$ \\
$^{1}$ Departamento de Ciencias F\'{\i}sicas, Matem\'aticas y de la
Computaci\'on\\
Universidad CEU Cardenal Herrera \\
San Bartolom\'e 55 \\
46115 Alfara del Patriarca (Valencia), Spain\\
e-mail: \texttt{afalco@uch.ceu.es}\\
$^{2}$ Max-Planck-Institut \emph{Mathematik in den Naturwissenschaften}\\
Inselstr. 22, D-04103 Leipzig, Germany \\
e-mail: \texttt{{wh@mis.mpg.de} }\\
$^{3}$ Ecole Centrale Nantes, \\
GeM UMR CNRS 6183, LUNAM Universit\'e\\
1 rue de la No\"e, BP 92101,\\
44321 Nantes Cedex 3, France.\\
e-mail: \texttt{anthony.nouy@ec-nantes.fr}}
\date{}
\maketitle

\begin{abstract}
The main goal of this paper is to study the geometric structures associated
with the representation of tensors in subspace based formats. To do this we
use a property of the so-called minimal subspaces which allows us to
describe the tensor representation by means of a rooted tree. By using the
tree structure and the dimensions of the associated minimal subspaces, we
introduce, in the underlying algebraic tensor space, the set of tensors in a
tree-based format with either bounded or fixed tree-based rank. This class
contains the Tucker format and the Hierarchical Tucker format (including the
Tensor Train format). In particular, we show that the set of tensors in the
tree-based format with bounded (respectively, fixed) tree-based rank of an
algebraic tensor product of normed vector spaces is an analytic Banach
manifold. Indeed, the manifold geometry for the set of tensors with fixed
tree-based rank is induced by a fibre bundle structure and the manifold
geometry for the set of tensors with bounded tree-based rank is given by a
finite union of connected components where each of them is a manifold of
tensors in the tree-based format with a fixed tree-based rank. The local
chart representation of these manifolds is often crucial for an algorithmic
treatment of high-dimensional PDEs and minimization problems.
In order to describe the relationship between these manifolds and the
natural ambient space, we introduce the definition of topological tensor
spaces in the tree-based format. We prove under natural conditions that any
tensor of the topological tensor space under consideration admits best
approximations in the manifold of tensors in the tree-based format with
bounded tree-based rank. In this framework, we also show that the tangent
(Banach) space at a given tensor is a complemented subspace in the natural
ambient tensor Banach space and hence the set of tensors in the tree-based
format with bounded (respectively, fixed) tree-based rank is an immersed
submanifold. This fact allows us to extend the Dirac-Frenkel variational
principle in the bodywork of topological tensor spaces.
\end{abstract}

\tableofcontents

\noindent\emph{2010 AMS Subject Classifications:} 15A69, 46B28, 46A32.

\noindent \emph{Key words:} Tensor spaces, Banach manifolds, Tensor formats, Tensor rank.

\section{Introduction}


Tensor approximation methods play a central role in the numerical solution of high dimensional problems arising in a wide range of applications. 
Low-rank tensor formats based on subspaces are widely used for complexity reduction in the representation of high-order tensors.
The construction of these formats are usually based on a hierarchy of tensor product
subspaces spanned by orthonormal bases, because in most cases a hierarchical
representation fits with the structure of the mathematical model and facilitates
its computational implementation. Two of the most popular formats are the
Tucker format and the Hierarchical Tucker format \cite{HaKuehn2009} (HT for
short). It is possible to show that the Tensor Train format \cite{Osedelets1}
(TT for short), introduced originally by Vidal \cite{Vidal}, is a particular
case of the HT format (see e.g. Chapter 12 in \cite{Hackbusch}). An
important feature of these formats, in the framework of topological tensor
spaces, is the existence of a best approximation in each fixed set of
tensors with bounded rank \cite{FALHACK}. In particular, it allows to
construct, on a theoretical level, iterative minimisation methods for
nonlinear convex problems over reflexive tensor Banach spaces \cite%
{FalcoNouy}.

Tucker tensors of fixed rank are also used for the discretisation of differential equations arising in
quantum chemical problems or in the multireference Hartree and Hartree-Fock
methods (MR-HF) in quantum dynamics \cite{Lubish}. In particular, for finite dimensional ambient tensor spaces, 
it can be shown that the set of Tucker tensors of fixed rank forms an immersed
finite-dimensional quotient manifold \cite{KoLu1}. A similar approach in a
complex Hilbert space setting for Tucker tensors of fixed rank is given in 
\cite{BCMT}. Then the numerical treatment of this class of problems follows
the general concepts of differential equations on manifolds \cite{HLW}.
Recently, similar results have been obtained for the TT format \cite{HRS}
and the HT format \cite{Uschmajew2012} (see also \cite{AJ}). The term
"matrix-product state" (MPS) was introduced in quantum physics (see, e.g., 
\cite{VC}). The related tensor representation can be found already in \cite%
{Vidal} without a special naming of the representation. The method has been
reinvented by Oseledets and Tyrtyshnikov (see \cite{Osedelets}, \cite%
{Osedelets1}, and \cite{OT}) and called "TT decomposition". For matrix
product states (MPS), the differential geometry in a finite-dimensional
complex Hilbert space setting is covered in \cite{HMOV}.

As we will show below, the Tucker and the HT formats are completely
characterised by a rooted tree together with a finite sequence of natural
numbers associated to each node on the tree, denominated the tree-based
rank. Each number in the tree-based rank is associated with a class of
subspaces of fixed dimension. Moreover, it can be shown that for a given
tree, every element in the tensor space possesses a unique  tree-based rank.
In consequence, given a tree, a tensor space is a union of sets indexed by
the tree-based ranks. It allows to consider for a given tree two kinds of
sets in a tensor space: the set of tensors of fixed tree-based rank and the
set of tensors of bounded tree-based rank. Two commonly accepted facts are the following.
\begin{enumerate}
\item[(a)] Even if it can be shown in finite dimension that the set of Tucker (respectively, HT)
tensors with bounded tree-based rank is closed, the existence of a manifold structure for this set is
an open question. Thus the existence of minimisers over these
sets can be shown, however, no first order optimality conditions are
available from a geometric point of view.

\item[(b)] Even if either in finite dimension or in a Hilbert space
setting it can be shown that the set of Tucker (respectively, in finite
dimensions HT) tensors with fixed tree-based rank is a quotient manifold,
the construction of an explicit parametrisation in order to provide a
manifold structure is not known.
\end{enumerate}

In our opinion, these two facts are due to the lack of a common mathematical
frame for developing a mathematical analysis of these
abstract objects. The main goal of this paper is to provide this common framework by
means of the theory for algebraic and topological tensor spaces developed
in \cite{FALHACK} by some of the authors of this article.

Our starting point are the following natural questions that arise in 
the mathematical theory of tensor spaces. The first one is: is
it possible to introduce a class of tensors containing Tucker, HT (and hence
TT) tensors with fixed and bounded rank? A second question is: if such a
class exists, is it possible to construct a parametrisation for the set of
tensors of bounded (respectively, fixed) rank in order to show that it is a
true manifold even in the infinite-dimensional case? Finally, if the answers
to the first two questions are positive, we would like to ask: is the set of tensors
of bounded (respectively, fixed) rank an immersed submanifold of the
topological tensor space, as ambient manifold, under consideration?
\\

The paper is organised as follows. 

\begin{itemize}
\item In Sect. \ref{Tensor_TBF}, we introduce the tree-based tensors as a
generalisation, at algebraic level, of the hierarchical tensor format. This
class contains the Tucker tensors (among others). Moreover, we characterise
the minimal subspaces for tree-based tensors extending the previous results
obtained in \cite{FALHACK} and introducing the definition of tree-based
rank. In particular, the main result of this section, Theorem~\ref%
{characterization_FT}, is a characterisation of the set of parameters needed
to provide an explicit geometric representation for the set of tensors with
fixed tree-based rank.

\item In Sect. \ref{sec:banach_manifold_tucker_fixed_rank}, by the help of
Theorem~\ref{characterization_FT}, we show that in an algebraic tensor
product of normed spaces the set of tensors with fixed tree-based rank 
is an analytic Banach manifold. Indeed, we give an explicit atlas
and we prove that this atlas is induced by a fibre bundle structure. This 
result allows us to deduce that the set of tensors with bounded
tree-based rank is also an analytic Banach manifold. An important
fact is that the geometric structure of these manifolds is independent on
the ambient tensor Banach space under consideration.

\item In Sect. \ref{embedded_manifold}, we discuss the choice of a norm in
the ambient tensor Banach space (a) to show the existence of a best approximation for the set of 
tensors with bounded tree-based rank and
(b) to prove that the set of tensors with fixed tree-based rank is an
immersed submanifold of that space (considered as Banach manifold).
To this end we assume the existence of a norm at each node of the tree not
weaker than the injective norm constructed from the Banach spaces associated
with the sons of that node. This assumption generalises the condition used
in \cite{FALHACK} to prove the existence of a best approximation in the
Tucker case. More precisely, under this assumption,

\begin{itemize}
\item we provide a proof of the existence of best approximation in the
manifold of tensors with bounded tree-based rank,

\item we construct a linear isomorphism, at each point in the manifold of
tensors with fixed tree-based rank, from the tangent space at that point to a closed
linear subspace of the ambient tensor Banach space, this subspace being
given explicitly,

\item we show that the set of tensors with fixed tree-based rank is an
immersed submanifold,

\item we also deduce that the set of tensors with bounded tree-based rank is
an immersed submanifold.
\end{itemize}

\item In Sect. \ref{Dirac_Frenkel}, we give a formalisation in this
framework of the multi--configuration time--dependent Hartree MCTDH method
(see \cite{Lubish}) in tensor Banach spaces.
\end{itemize}

\section{Algebraic tensors spaces in the tree-based Format}

\label{Tensor_TBF}

\subsection{Preliminary definitions and notations}

Concerning the definition of the algebraic tensor space $_{a}\bigotimes
_{j=1}^{d}V_{j}$ generated from vector spaces $V_{j}$ $\left( 1\leq j\leq
d\right) $, we refer to Greub \cite{Greub}. As underlying field we choose $%
\mathbb{R},$ but the results hold also for $\mathbb{C}$. The suffix `$a$' in 
$_{a}\bigotimes_{j=1}^{d}V_{j}$ refers to the `algebraic' nature. By
definition, all elements of 
\begin{equation*}
\mathbf{V}:=\left. _{a}\bigotimes_{j=1}^{d}V_{j}\right.
\end{equation*}
are \emph{finite} linear combinations of elementary tensors $\mathbf{v}%
=\bigotimes_{j=1}^{d}v_{j}$ $\left( v_{j}\in V_{j}\right) .$ Let $%
D:=\{1,\ldots ,d\}$ be the index set of the `spatial directions'. In the
sequel, the index sets $D\backslash \{j\}$ will appear. Here, we use the
abbreviations%
\begin{equation*}
\mathbf{V}_{[j]}:=\left. _{a}\bigotimes_{k\neq j}V_{k}\right. \text{,\qquad
where }\bigotimes_{k\neq j}\text{ means}\bigotimes_{k\in D\backslash \{j\}}.
\label{(V[j]}
\end{equation*}%
Similarly, elementary tensors $\bigotimes_{k\neq j}v_{k}$ are denoted by $%
\mathbf{v}_{[j]}$. The following notations and definitions will be useful.

\bigskip

For vector spaces $V_{j}$ and $W_{j}$ over $\mathbb{R},$ let linear mappings 
$A_{j}:V_{j}\rightarrow W_{j}$ $\left( 1\leq j\leq d\right) $ be given. Then
the definition of the elementary tensor%
\begin{equation*}
\mathbf{A}=\bigotimes_{j=1}^{d}A_{j}:\;\mathbf{V}=\left. _{a}\bigotimes
_{j=1}^{d}V_{j}\right. \longrightarrow\mathbf{W}=\left. _{a}\bigotimes
_{j=1}^{d}W_{j}\right.
\end{equation*}
is given by%
\begin{equation}
\mathbf{A}\left( \bigotimes_{j=1}^{d}v_{j}\right)
:=\bigotimes_{j=1}^{d}\left( A_{j}v_{j}\right) .
\label{(A als Tensorprodukt der Aj}
\end{equation}
Note that \eqref{(A als Tensorprodukt der Aj} uniquely defines the linear
mapping $\mathbf{A}:\mathbf{V}\rightarrow\mathbf{W}.$ We recall that $L(V,W)$
is the space of linear maps from $V$ into $W,$ while $V^{\prime}=L(V,\mathbb{%
R})$ is the algebraic dual of $V$. For metric spaces, $\mathcal{L}(V,W)$
denotes the continuous linear maps, while $V^{\ast}=\mathcal{L}(V,\mathbb{R}%
) $ is the topological dual of $V$. Often, mappings $\mathbf{A}%
=\bigotimes_{j=1}^{d}A_{j}$ will appear, where most of the $A_{j}$ are the
identity (and therefore $V_{j}=W_{j}$). If $A_{k}\in L(V_{k},W_{k})$ for one 
$k$ and $A_j=id$ for $j\neq k$, we use the following notation:%
\begin{subequations}
\label{(Notation Aj}%
\begin{equation*}
\mathbf{id}_{[k]}\otimes A_{k}:=\underset{k-1\text{ factors}}{\underbrace {%
id\otimes\ldots\otimes id}}\otimes A_{k}\otimes\underset{d-k\text{ factors}}{%
\underbrace{id\otimes\ldots\otimes id}}\in L(\mathbf{V},\mathbf{V}%
_{[k]}\otimes_{a}W_{k}),  \label{(Notation Aja}
\end{equation*}
provided that it is obvious which component $k$ is meant. By the
multiplication rule $\left( \bigotimes_{j=1}^{d}A_{j}\right) \circ\left(
\bigotimes _{j=1}^{d}B_{j}\right) =\bigotimes_{j=1}^{d}\left( A_{j}\circ
B_{j}\right) $ and since $id\circ A_{j}=A_{j}\circ id,$ the following
identity\footnote{%
Note that the meaning of $\mathbf{id}_{[j]}$ and $\mathbf{id}_{[k]}$ may
differ: in the second line of \eqref{Notation_eq1}, $(\mathbf{id}%
_{[k]}\otimes A_{k})\in L(\mathbf{V},\mathbf{V}_{[k]}\otimes _{a}W_{k})$ and 
$(\mathbf{id}_{[j]}\otimes A_{j})\in L\left( \mathbf{V}_{[k]}%
\otimes_{a}W_{k},\mathbf{V}_{[j,k]}\otimes_{a}W_{j}\otimes_{a}W_{k}\right) ,$
whereas in the third one $(\mathbf{id}_{[j]}\otimes A_{j})\in L(\mathbf{V},%
\mathbf{V}_{[j]}\otimes_{a}W_{j})$ and $(\mathbf{id}_{[k]}\otimes A_{k})\in
L\left( \mathbf{V}_{[j]}\otimes_{a}W_{j},\mathbf{V}_{[j,k]}\otimes_{a}W_{j}%
\otimes_{a}W_{k}\right).$ Here $\mathbf{V}_{[j,k]} = \left.
_{a}\bigotimes_{l \in D \setminus\{j,k\}}V_{l}\right.$.} holds for $j\neq k$:%
\begin{equation*}
\begin{array}{l}
id\otimes\ldots\otimes id\otimes A_{j}\otimes id\otimes\ldots\otimes
id\otimes A_{k}\otimes id\otimes\ldots\otimes id \\ 
=(\mathbf{id}_{[j]}\otimes A_{j})\circ(\mathbf{id}_{[k]}\otimes A_{k}) \\ 
=(\mathbf{id}_{[k]}\otimes A_{k})\circ(\mathbf{id}_{[j]}\otimes A_{j})%
\end{array}
\label{Notation_eq1}
\end{equation*}
(in the first line we assume $j<k$). Proceeding inductively with this
argument over all indices, we obtain%
\begin{equation*}
\mathbf{A}=\bigotimes_{j=1}^{d}A_{j}=(\mathbf{id}_{[1]}\otimes
A_{1})\circ\cdots\circ(\mathbf{id}_{[d]}\otimes A_{d}).
\end{equation*}

If $W_{j}=\mathbb{R},$ i.e., if $A_{j}=\varphi_{j}\in V_{j}^{\prime}$ is a
linear form, then $\mathbf{id}_{[j]}\otimes\varphi_{j}\in L(\mathbf{V},%
\mathbf{V}_{[j]})$ is used as symbol for $id\otimes\ldots\otimes
id\otimes\varphi_{j}\otimes id\otimes\ldots\otimes id$ defined by%
\begin{equation*}
(\mathbf{id}_{[j]}\otimes\varphi_{j})\left( \bigotimes_{k=1}^{d}v_{k}\right)
=\varphi_{j}(v_{j})\cdot\bigotimes_{k\neq j}v_{k}.  \label{(Notation Ajb}
\end{equation*}
Thus, if $\boldsymbol{\varphi}=\otimes_{j=1}^{d}\varphi_{j}\in\bigotimes
_{j=1}^{d}V_{j}^{\prime},$ we can also write%
\begin{equation}
\boldsymbol{\varphi}=\otimes_{j=1}^{d}\varphi_{j}=(\mathbf{id}%
_{[1]}\otimes\varphi_{1})\circ\cdots\circ(\mathbf{id}_{[d]}\otimes\varphi
_{d}).  \label{(Notation Aje}
\end{equation}
Consider again the splitting of $\mathbf{V}=\left.
_{a}\bigotimes_{j=1}^{d}V_{j}\right. $ into $\mathbf{V}=V_{j}\otimes_{a}%
\mathbf{V}_{[j]}$ with $\mathbf{V}_{[j]}:=\left. _{a}\bigotimes_{k\neq
j}V_{k}\right. $. For a linear form $\boldsymbol{\varphi}_{[j]}\in\mathbf{V}%
_{[j]}^{\prime}$, the notation $id_{j}\otimes\boldsymbol{\varphi}_{[j]}\in L(%
\mathbf{V},V_{j})$ is used for the mapping%
\begin{equation}
(id_{j}\otimes\boldsymbol{\varphi}_{[j]})\left(
\bigotimes_{k=1}^{d}v_{k}\right) =\boldsymbol{\varphi}_{[j]}%
\bigg(%
\bigotimes_{k\neq j}v_{k}%
\bigg)%
\cdot v_{j}.  \label{(Notation Ajc}
\end{equation}
If $\boldsymbol{\varphi}_{[j]}=\bigotimes_{k\neq j}\varphi_{k}\in\left.
_{a}\bigotimes_{k\neq j}V_{k}^{\prime}\right. $ is an elementary tensor%
, $\boldsymbol{\varphi}_{[j]}\left(
\bigotimes_{k\neq j} v^{(k)}\right) =\prod_{k\neq j}\varphi_{k}\left(
v^{(k)}\right) $ holds in \eqref{(Notation Ajc}. Finally, we can write (\ref%
{(Notation Aje}) as 
\begin{equation*}
\boldsymbol{\varphi}=\otimes_{j=1}^{d}\varphi_{j}=\varphi_{j}\circ
(id_{j}\otimes\boldsymbol{\varphi}_{[j]})\qquad\text{for }1\leq j\leq d.
\label{(Notation Ajh}
\end{equation*}%
\end{subequations}%

\subsection{Algebraic tensor spaces in the tree-based format}

We introduce the abbreviation TBF for `tree-based format'. For instance, a
TBF tensor is a tensor represented in the tree-based format, etc. The
tree-based rank will be abbreviated by TB\ rank. To introduce the underlying
tree we use the following example.

\begin{example}
Let us consider $D=\{1,2,3,4,5,6\},$ then 
\begin{equation*}
\mathbf{V}_{D}=\left. _{a}\bigotimes_{j=1}^{6}V_{j}\right. =\left( \left.
_{a}\bigotimes_{j=1}^{3}V_{j}\right. \right) \otimes _{a}\left( \left.
_{a}\bigotimes_{j=4}^{5}V_{j}\right. \right) \otimes _{a}V_{6}=\mathbf{V}%
_{123}\otimes _{a}\mathbf{V}_{45}\otimes _{a}V_{6}.
\end{equation*}%
Observe that $\mathbf{V}_{D}=\left. _{a}\bigotimes_{j=1}^{6}V_{j}\right.$
can be represented by the tree given in Figure~\ref{fig2p} and $\mathbf{V}%
_{D} = \mathbf{V}_{123}\otimes _{a}\mathbf{V}_{45}\otimes _{a}V_{6}$ by the
tree given in Figure~\ref{fig1p}. We point out that there are other trees to
describe the tensor representation $\mathbf{V}_{D} = \mathbf{V}_{123}\otimes
_{a}\mathbf{V}_{45}\otimes _{a}V_{6},$ because 
\begin{equation*}
\mathbf{V}_{D}=\left( \left. _{a}\bigotimes_{j=1}^{3}V_{j}\right. \right)
\otimes _{a}\left( \left. _{a}\bigotimes_{j=4}^{5}V_{j}\right. \right)
\otimes _{a}V_{6} = \left(V_1 \otimes_a \left(\left.
_{a}\bigotimes_{j=2}^{3}V_{j}\right. \right) \right) \otimes _{a}\left(
\left. _{a}\bigotimes_{j=4}^{5}V_{j}\right. \right) \otimes _{a}V_{6},
\end{equation*}
that is, $\mathbf{V}_{123} = \left. _{a}\bigotimes_{j=1}^{3}V_{j}\right. =
V_1 \otimes_a \mathbf{V}_{23}$ (see Figure~\ref{fig0p}).
\end{example}

The above example motivates the following definition.

\begin{definition}
\label{partition_tree} The tree $T_{D}$ is called \emph{a dimension
partition tree of $D$} if

\begin{enumerate}
\item[(a)] all vertices $\alpha \in T_D$ are non--empty subsets of $D,$

\item[(b)] $D$ is the root of $T_D,$

\item[(c)] every vertex $\alpha \in T_{D}$ with $\#\alpha \geq 2$ has at
least two sons. Moreover, if $S(\alpha )\subset 2^{D}$ denotes the set of
sons of $\alpha $ then $\alpha =\cup _{\beta \in S(\alpha )}\beta $ where $%
\beta \cap \beta ^{\prime }=\emptyset $ for all $\beta ,\beta ^{\prime }\in
S(\alpha ),$ $\beta \neq \beta ^{\prime },$

\item[(d)] every vertex $\alpha\in T_D$ with $\#\alpha = 1$ has no son.
\end{enumerate}
\end{definition}

If $S(\alpha )=\emptyset ,$ $\alpha $ is called a \emph{leaf}. The set of
leaves is denoted by $\mathcal{L}(T_{D}).$ An easy consequence of Definition~%
\ref{partition_tree} is that the set of leaves $\mathcal{L}(T_{D})$
coincides with the singletons of $D,$ i.e., $\mathcal{L}(T_{D})=\{\{j\}:j\in
D\}.$

\begin{example}
Consider $D=\{1,2,3,4,5,6\}.$ Take 
\begin{equation*}
T_D=\{D,\{1\},\{2\},\{3\},\{4\},\{5\},\{6\}\} \text{ and } S(D) =
\{\{1\},\{2\},\{3\},\{4\},\{5\},\{6\}\}
\end{equation*}
(see Figure~\ref{fig2p}). Then $S(D) = \mathcal{L}(T_D).$
\end{example}

\begin{example}
\label{example_nobinary} In Figure~\ref{fig1p} we have a tree which
corresponds to $\mathbf{V}_{D} = \mathbf{V}_{123}\otimes _{a}\mathbf{V}%
_{45}\otimes _{a}V_{6}.$ Here $D=\{1,2,3,4,5,6\}$ and 
\begin{equation*}
T_D=\{D,\{1,2,3\},\{4,5\},\{1\},\{2\},\{3\},\{4\},\{5\},\{6\}\},
\end{equation*}
\begin{equation*}
S(D)=\{\{1,2,3\},\{4,5\},\{6\}\}, \, S(\{4,5\})=\{\{4\},\{5\}\}, \,
S(\{1,2,3\})=\{\{1\},\{2\},\{3\}\}.
\end{equation*}
Moreover $\mathcal{L}(T_D)=\{\{1\},\{2\},\{3\},\{4\},\{5\},\{6\}\}.$
\end{example}

\begin{figure}[tbp]
\centering
\synttree[$\{1,2,3,4,5,6\}$[$\{1\}$][$\{2\}$][$\{3\}$]
[$\{4\}$][$\{5\}$][$\{6\}$]]
\caption{A dimension partition tree related to $\mathbf{V}_{D}=\left.
_{a}\bigotimes_{j=1}^{6}V_{j}\right..$}
\label{fig2p}
\end{figure}

\begin{figure}[tbp]
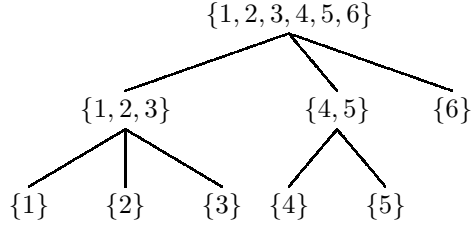

\centering
\synttree[$\{1,2,3,4,5,6\}$[$\{1,2,3\}$[$\{1\}$][$\{2\}$][$\{3\}$]]
[$\{4,5\}$[$\{4\}$][$\{5\}$]][$\{6\}$ ]]
\caption{A dimension partition tree related to $\mathbf{V}_{D} = \mathbf{V}%
_{123}\otimes _{a}\mathbf{V}_{45}\otimes _{a}V_{6}.$}
\label{fig1p}
\end{figure}

\begin{figure}[tbp]
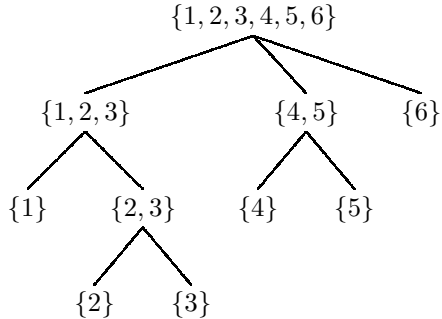

\centering
\synttree[$\{1,2,3,4,5,6\}$[$\{1,2,3\}$[$\{1\}$][$\{2,3\}$[$\{2\}$][$\{3\}$]]]
[$\{4,5\}$[$\{4\}$][$\{5\}$]][$\{6\}$ ]]
\caption{A dimension partition tree related to $\mathbf{V}_{D} = \mathbf{V}%
_{123}\otimes _{a}\mathbf{V}_{45}\otimes _{a}V_{6}$ where $\mathbf{V}_{123}
= V_1 \otimes_a \mathbf{V}_{23}.$}
\label{fig0p}
\end{figure}

Finally we give the definition of a TBF tensor.

\begin{definition}
Let $D$ be a finite index set and $T_{D}$ be a partition tree. Let $V_{j}$
be a vector space for $j\in D,$ and consider for each $\alpha \in
T_{D}\setminus \mathcal{L}(T_{D})$ a tensor space $\mathbf{V}_{\alpha
}:=\left. _{a}\bigotimes_{\beta \in S(\alpha )}\mathbf{V}_{\beta }\right. .$
Then the collection of vector spaces $\{\mathbf{V}_{\alpha }\}_{\alpha \in
T_{D}\setminus \{D\}}$ is called a representation of the tensor space $%
\mathbf{V}_{D}=\left. _{a}\bigotimes_{\alpha \in S(D)}\mathbf{V}_{\alpha
}\right. $ in \emph{\ tree-based format.}
\end{definition}

Observe that we can write $\mathbf{V}_D = \left._a \bigotimes_{\alpha \in
S(D)} \mathbf{V}_{\alpha}\right. = \left._a \bigotimes_{j \in D}
V_{j}\right..$ A first property of TBF tensors is the independence of the
representation of the algebraic tensor space $\mathbf{V}_D$ with respect to
the tree $T_D.$

\begin{lemma}
Let $D$ be a finite index set and $T_{D}$ be a partition tree. Let $V_{j}$
be a vector space for $j\in D.$ Assume that $\{\mathbf{V}_{\alpha
}\}_{\alpha \in T_{D}\setminus \{D\}}$ is a representation of the tensor
space $\mathbf{V}_{D}=\left. _{a}\bigotimes_{\alpha \in S(D)}\mathbf{V}%
_{\alpha }\right. $ in the tree-based format. Then for each $\alpha _{1}\in
T_{D}\setminus \{D\}$ there exist $\alpha _{2},\ldots ,\alpha _{m}\in
T_{D}\setminus \{D,\alpha _{1}\}$ such that $D=\cup _{i=1}^{m}\alpha _{i},$ $%
\alpha _{i}\cap \alpha _{j}=\emptyset $ and $\mathbf{V}_{D}=\left.
_{a}\bigotimes_{i=1}^{m}\mathbf{V}_{\alpha _{i}}\right. .$
\end{lemma}

\subsection{Minimal subspaces for TBF tensors}

Let $V_{j}$ be a vector space for $j\in D,$ where $D$ is a finite index set,
and $\alpha_1,\ldots,\alpha_m \subset 2^D \setminus \{D,\emptyset\},$ be
such that $\alpha_i \cap \alpha_j = \emptyset$ for all $i \neq j$ and $D =
\bigcup_{j=1}^m \alpha_j.$ For $\mathbf{v} \in \left._a \bigotimes_{i=1}^m 
\mathbf{V}_{\alpha_i}\right.$ we define the minimal subspace of $\mathbf{v}$
on each $\mathbf{V}_{\alpha_i} := \left._a \bigotimes_{j \in \alpha_i} V_j
\right.$ for $1 \le i \le m,$ as follows.

\begin{definition}
\label{Def Vmin,Wmin} For a tensor $\mathbf{v}\in \left.
_{a}\bigotimes_{j\in D}V_{j}\right. =\left. _{a}\bigotimes_{i=1}^{m}\mathbf{V%
}_{\alpha _{i}}\right. $, the \emph{minimal subspaces} denoted by $U_{\alpha
_{i}}^{\min }(\mathbf{v})\subset \mathbf{V}_{\alpha _{i}},$ for $1\leq i\leq
m,$ are defined by the properties that $\mathbf{v}%
\in \left. _{a}\bigotimes_{i=1}^{m}U_{\alpha _{i}}^{\min }(\mathbf{v}%
)\right. $ and  $\mathbf{v}\in \left.
_{a}\bigotimes_{i=1}^{m}\mathbf{U}_{\alpha _{i}}\right. $ implies $U_{\alpha
_{i}}^{\min }(\mathbf{v})\subset \mathbf{U}_{\alpha _{i}}.$ 
\end{definition}

The minimal subspaces are useful to introduce the following sets of tensor
representations based on subspaces. Fix $\mathbf{r}=(r_1,\ldots,r_d) \in 
\mathbb{N}^d$. Then we define \emph{the set of Tucker tensors with bounded
rank $\mathbf{r}$ in $\mathbf{V} = \left._a \bigotimes_{j=1}^d V_j \right.$}
by 
\begin{equation*}
\mathcal{T}_{\mathbf{r}}(\mathbf{V}):= \left\{ \mathbf{v} \in \mathbf{V}:
\dim U_j^{\min}(\mathbf{v}) \le r_j, \, 1 \le j \le d \right\},
\end{equation*}
and \emph{the set of Tucker tensors with fixed rank $\mathbf{r}$ in $\mathbf{%
V} = \left._a \bigotimes_{j=1}^d V_j \right.$} by 
\begin{equation*}
\mathcal{M}_{\mathbf{r}}(\mathbf{V}):= \left\{ \mathbf{v} \in \mathbf{V}:
\dim U_j^{\min}(\mathbf{v}) = r_j, \, 1 \le j \le d \right\}.
\end{equation*}
Then $\mathcal{M}_{\mathbf{r}}(\mathbf{V}) \subset \mathcal{T}_{\mathbf{r}}(%
\mathbf{V}) \subset \mathbf{V} $ holds.

\bigskip

The next characterisation of $U_{\alpha_j}^{\min }(\mathbf{v})$ for $1 \le j
\le m$ is due to \cite{Hackbusch} (it is included in the proof of Lemma
6.12). Since we assume that $\mathbf{V}_{\alpha_j}$ are vector spaces for $1
\le j \le m$, then we may consider the subspaces 
\begin{equation*}
U_{\alpha_j}^{I}(\mathbf{v}) :=\left\{ (id_{\alpha_j}\otimes \boldsymbol{%
\varphi }_{[\alpha_j]})(\mathbf{v}):\text{\ }\boldsymbol{\varphi }%
_{[\alpha_j]}\in \left. _{a}\bigotimes\nolimits_{k\neq j}\mathbf{V}%
_{\alpha_k}^{\prime }\right. \right\}
\end{equation*}%
and 
\begin{equation*}
U_{\alpha_j}^{II}(\mathbf{v}) :=\left\{ (id_{\alpha_j}\otimes \boldsymbol{%
\varphi }_{[\alpha_j]})(\mathbf{v}):\text{\ }\boldsymbol{\varphi }%
_{[\alpha_j]}\in \left. _{a}\bigotimes\nolimits_{k\neq j}U_{\alpha_k}^{\min}(%
\mathbf{v})^{\prime }\right. \right\},
\end{equation*}%
for $1 \le j \le m.$ Moreover, if $\mathbf{V}_{\alpha_j}$ are normed spaces
for $1 \le j \le m$ we can also consider 
\begin{equation*}
U_{\alpha_j}^{III}(\mathbf{v}) :=\left\{ (id_{\alpha_j}\otimes \boldsymbol{%
\varphi }_{[\alpha_j]})(\mathbf{v}):\text{\ }\boldsymbol{\varphi }%
_{[\alpha_j]}\in \left. _{a}\bigotimes\nolimits_{k \neq j} \mathbf{V}%
_{\alpha_k}^{\ast }\right. \right\},
\end{equation*}%
and 
\begin{equation*}
U_{\alpha_j}^{IV}(\mathbf{v}) :=\left\{ (id_{\alpha_j}\otimes \boldsymbol{%
\varphi }_{[\alpha_j]})(\mathbf{v}):\text{\ }\boldsymbol{\varphi }%
_{[\alpha_j]}\in \left. _{a}\bigotimes\nolimits_{k \neq
j}U_{\alpha_k}^{\min}(\mathbf{v})^{\ast }\right. \right\},
\end{equation*}

\begin{theorem}
\label{Satz Vjmin} Assume that $\mathbf{V}_{\alpha_j}$ are vector spaces for 
$1 \le j \le m.$ Then the following statements hold.

\begin{enumerate}
\item[(a)] For any $\mathbf{v}\in \mathbf{V}=\left. _{a}\bigotimes_{j=1}^m%
\mathbf{V}_{\alpha_j}\right.$, it holds 
\begin{equation*}
U_{\alpha_j}^{\min }(\mathbf{v})=U_{\alpha_j}^{I}(\mathbf{v}) =
U_{\alpha_j}^{II}(\mathbf{v}),
\end{equation*}
for $1 \le j \le m.$

\item[(b)] Assume that $\mathbf{V}_{\alpha_j}$ are normed spaces for $1 \le
j \le m.$ Then for any $\mathbf{v}\in \mathbf{V}=\left.
_{a}\bigotimes_{j=1}^m\mathbf{V}_{\alpha_j}\right.,$ it holds 
\begin{equation*}
U_{\alpha_j}^{\min }(\mathbf{v})=U_{\alpha_j}^{III}(\mathbf{v}) =
U_{\alpha_j}^{IV}(\mathbf{v}),
\end{equation*}
for $1 \le j \le m.$
\end{enumerate}
\end{theorem}

Let $D=\cup_{i=1}^m \alpha_i$ be a given partition. Assume that $%
\alpha_1=\cup_{j=1}^n \beta_j$ is also a given partition, then we have
minimal subspaces $U_{\beta_j}^{\min}(\mathbf{v}) \subset \mathbf{V}%
_{\beta_j} = \left._a \bigotimes_{k \in \beta_j} V_k\right.$ for $1 \le j
\le n$ and $U_{\alpha_i}^{\min}(\mathbf{v}) \subset \mathbf{V}_{\alpha_i} =
\left._a \bigotimes_{k \in \alpha_i} V_{k}\right.$ for $1 \le i \le m.$
Observe that $\mathbf{V}_{\alpha_1} = \left._a \bigotimes_{j=1}^n \mathbf{V}%
_{\beta_j} \right.,$ where 
\begin{equation*}
\mathbf{v} \in \left._a \bigotimes_{i=1}^m U_{\alpha_i}^{\min}(\mathbf{v}%
)\right. \text{ and } \mathbf{v} \in \left(\left._a \bigotimes_{j=1}^n
U_{\beta_j}^{\min}(\mathbf{v}) \right. \right) \otimes_ a \left(\left._a
\bigotimes_{i=2}^m U_{\alpha_i}^{\min}(\mathbf{v}) \right. \right).
\end{equation*}

\begin{example}
Let us consider $D=\{1,2,3,4,5,6\}$ and the partition tree $T_D$ given in
Figure \ref{fig1p}. Take $\mathbf{v}\in\left. _{a}\bigotimes_{j\in
D}V_{j}\right. =\mathbf{V}_{\alpha_{1}}\otimes_a\mathbf{V}%
_{\alpha_{2}}\otimes_a\mathbf{V}_{\alpha_{3}},$ where $\alpha_{1}=\{1,2,3\},$
$\alpha_{2}=\{4,5\},$ and $\alpha_{3}=\{6\}.$ Then we can conclude that
there are minimal subspaces ${U}_{\alpha_{\nu}}^{\min}(\mathbf{v})$
for $\nu=1,2,3,$ such that $\mathbf{v}\in\left. _{a}\bigotimes_{\nu=1}^{3}%
{U}_{\alpha_{\nu}}^{\min}(\mathbf{v})\right. $ and also minimal
subspaces $U_{j}^{\min}(\mathbf{v})$ for $j \in D,$ such that $\mathbf{v}
\in \left._a \bigotimes_{j\in D} U_{j}^{\min}(\mathbf{v})\right. $
\end{example}

The relation between $U_{j}^{\min}(\mathbf{v})$ and ${U}%
_{\alpha_{\nu}}^{\min}(\mathbf{v})$ is as follows (see Corollary~2.9 of \cite%
{FALHACK}).

\begin{proposition}
\label{inclusin_Umin} Let $V_{j}$ be a vector space for $j\in D,$ where $D$
is a finite index set, and $D=\cup _{i=1}^{m}\alpha _{i}$ be a given
partition. Let $\mathbf{v}\in \left. _{a}\bigotimes_{j\in D}V_{j}\right. .$
For a partition $\alpha _{1}=\cup _{j=1}^{m}\beta _{j}$ it holds 
\begin{equation*}
U_{\alpha _{1}}^{\min }(\mathbf{v})\subset \left.
_{a}\bigotimes_{j=1}^{m}U_{\beta _{j}}^{\min }(\mathbf{v})\right. .
\end{equation*}
\end{proposition}

The following result gives us the relationship between a basis of $%
U_{\alpha_1}^{\min}(\mathbf{v})$ and a basis of $U_{\beta_j}^{\min}(\mathbf{v%
})$ for $1 \le j \le m.$

\begin{proposition}
\label{(Ualpha in Tensor Uj coro} Let $V_{j}$ be a vector space for $j\in D,$
where $D$ is a finite index set. Let $\alpha \subset D$ such that $\alpha =
\bigcup_{i=1}^m \alpha_i$, where $\emptyset \neq \alpha_i$ are pairwise
disjoint for $1\le i\le m.$ 
Let $\mathbf{v}\in \left. _{a}\bigotimes_{j\in D}V_{j}\right..$ 
The following statements hold.

\begin{enumerate}
\item[(a)] For each $1 \le i \le m$, it holds 
\begin{align*}
U_{\alpha_i}^{\min}(\mathbf{v}) & = \mathrm{span}\, \left\{ \left(
id_{\alpha_i} \otimes\boldsymbol{\varphi}^{(\alpha \setminus \alpha_i)}
\right)(\mathbf{v}_{\alpha}): \mathbf{v}_{\alpha} \in U_{\alpha }^{\min }(%
\mathbf{v}) \text{ and } \boldsymbol{\varphi}^{(\alpha \setminus \alpha_i)}
\in \left._{a}\bigotimes_{k \neq i } U_{\alpha_k}^{\min}(\mathbf{v})^{\prime
}\right. \right\} \\
& = \mathrm{span}\, \left\{ \left( id_{\alpha_i} \otimes\boldsymbol{\varphi}%
^{(\alpha \setminus \alpha_i)} \right)(\mathbf{v}_{\alpha}): \mathbf{v}%
_{\alpha} \in U_{\alpha }^{\min }(\mathbf{v}) \text{ and } \boldsymbol{%
\varphi}^{(\alpha \setminus \alpha_i)} \in \left._{a}\bigotimes_{k\neq i} 
\mathbf{V}_{\alpha_k}^{\prime }\right. \right\}.
\end{align*}

\item[(b)] Assume that $\mathbf{V}_{\alpha}:= \left._{a}\bigotimes_{i=1 }^m 
\mathbf{V}_{\alpha_i}\right.$ 
and $\mathbf{V}_{\alpha_i}$, for $1 \le i \le m
$, are normed spaces. For each $1 \le i \le m$ it holds 
\begin{align*}
U_{\alpha_i}^{\min}(\mathbf{v}) & = \mathrm{span}\, \left\{ \left(
id_{\alpha_i} \otimes\boldsymbol{\varphi}^{(\alpha \setminus \alpha_i)}
\right)(\mathbf{v}_{\alpha}): \mathbf{v}_{\alpha} \in U_{\alpha }^{\min }(%
\mathbf{v}) \text{ and } \boldsymbol{\varphi}^{(\alpha \setminus \alpha_i)}
\in \left._{a}\bigotimes_{k \neq i} U_{\alpha_k}^{\min}(\mathbf{v})^*
\right. \right\} \\
& = \mathrm{span}\, \left\{ \left( id_{\alpha_i} \otimes\boldsymbol{\varphi}%
^{(\alpha \setminus \alpha_i)} \right)(\mathbf{v}_{\alpha}): \mathbf{v}%
_{\alpha} \in U_{\alpha }^{\min }(\mathbf{v}) \text{ and } \boldsymbol{%
\varphi}^{(\alpha \setminus \alpha_i)} \in \left._{a}\bigotimes_{k\neq i} 
\mathbf{V}_{\alpha_k}^* \right. \right\}
\end{align*}
\end{enumerate}
\end{proposition}

\begin{proof}
Statements (a) and (b) are proved in a similar way. Let $\gamma =D\setminus
\alpha $ and write $\gamma =\bigcup_{i=1}^{n}\gamma _{i}$, where $\emptyset
\neq \gamma _{i}\subset D$ are pairwise disjoint for $i=1,2,\ldots ,n.$ In
particular, to prove (b), we observe that 
\begin{equation*}
\mathbf{V}_{D}=\mathbf{V}_{\alpha }\otimes _{a}\mathbf{V}_{\gamma }=\left(
\left. _{a}\bigotimes_{i=1}^{m}\mathbf{V}_{\alpha _{i}}\right. \right)
\otimes _{a}\left( \left. _{a}\bigotimes_{j=1}^{n}\mathbf{V}_{\gamma
_{j}}\right. \right) .
\end{equation*}%
Then, by Theorem~\ref{Satz Vjmin}(b), using $U_{\alpha _{i}}^{IV}(\mathbf{v}%
),$ we have 
\begin{align*}
U_{\alpha }^{\min }(\mathbf{v})& =\left\{ (id_{\alpha }\otimes \boldsymbol{%
\varphi }^{(\gamma )})(\mathbf{v}):\boldsymbol{\varphi }^{(\gamma )}\in
\left. _{a}\bigotimes_{j=1}^{m}U_{\gamma _{j}}^{\min }(\mathbf{v})^{\ast
}\right. \right\} \text{ and } \\
U_{\alpha _{i}}^{\min }(\mathbf{v})& =\left\{ (id_{\alpha _{i}}\otimes 
\boldsymbol{\varphi }^{(D\setminus \alpha _{i})})(\mathbf{v}):\boldsymbol{%
\varphi }^{(D\setminus \alpha _{i})}\in \left( \left. _{a}\bigotimes_{k\neq
i}U_{\alpha _{k}}^{\min }(\mathbf{v})^{\ast }\right. \right) \otimes
_{a}\left( \left. _{a}\bigotimes_{j=1}^{m}U_{\gamma _{j}}^{\min }(\mathbf{v}%
)^{\ast }\right. \right) \right\}
\end{align*}%
for $1\leq i\leq m.$ Take $\mathbf{v}_{\alpha }\in {U}_{\alpha }^{\min }(%
\mathbf{v}).$ Then there exists $\boldsymbol{\varphi }^{(\gamma )}\in \left.
_{a}\bigotimes_{j=1}^{m}U_{\gamma _{j}}^{\min }(\mathbf{v})^{\ast }\right. $
such that $\mathbf{v}_{\alpha }=\left( id_{\alpha }\otimes \boldsymbol{%
\varphi }^{(\gamma )}\right) (\mathbf{v}).$ Now, for $\boldsymbol{\varphi }%
^{(\alpha \setminus \alpha _{i})}\in \left. _{a}\bigotimes_{k\neq
i}U_{\alpha _{k}}^{\min }(\mathbf{v})^{\ast }\right. ,$ we have 
\begin{equation*}
\left( id_{\alpha _{i}}\otimes \boldsymbol{\varphi }^{(\alpha \setminus
\alpha _{i})}\right) (\mathbf{v}_{\alpha })=\left( id_{\alpha _{i}}\otimes 
\boldsymbol{\varphi }^{(\alpha \setminus \alpha _{i})}\otimes \boldsymbol{%
\varphi }^{(D\setminus \alpha )}\right) (\mathbf{v}),
\end{equation*}%
and hence $\left( id_{\alpha _{i}}\otimes \boldsymbol{\varphi }^{(\alpha
\setminus \alpha _{i})}\right) (\mathbf{v}_{\alpha })\in U_{\alpha
_{i}}^{\min }(\mathbf{v}).$ Now, take $\mathbf{v}_{\alpha _{i}}\in U_{\alpha
_{i}}^{\min }(\mathbf{v}),$ then there exists 
\begin{equation*}
\boldsymbol{\varphi }^{(D\setminus \alpha _{i})}\in \left( \left.
_{a}\bigotimes_{k\neq i}U_{\alpha _{k}}^{\min }(\mathbf{v})^{\ast }\right.
\right) \otimes _{a}\left( \left. _{a}\bigotimes_{j=1}^{m}U_{\gamma
_{j}}^{\min }(\mathbf{v})^{\ast }\right. \right)
\end{equation*}%
such that $\mathbf{v}_{\alpha _{i}}=\left( id_{\alpha _{i}}\otimes 
\boldsymbol{\varphi }^{(D\setminus \alpha _{i})}\right) (\mathbf{v}).$ Then $%
\boldsymbol{\varphi }^{(D\setminus \alpha _{i})}=\sum_{l=1}^{r}\boldsymbol{%
\psi }_{l}^{(\alpha \setminus \alpha _{i})}\otimes \boldsymbol{\phi }%
_{l}^{(\gamma )},$ where $\boldsymbol{\phi }_{l}^{(\gamma )}\in \left.
_{a}\bigotimes_{j=1}^{m}U_{\gamma _{j}}^{\min }(\mathbf{v})^{\ast }\right. $
and $\boldsymbol{\psi }_{l}^{(\alpha \setminus \alpha _{i})}\in \left.
_{a}\bigotimes_{k\neq i}U_{\alpha _{k}}^{\min }(\mathbf{v})^{\ast }\right. $
for $1\leq l\leq r.$ Thus, 
\begin{align*}
\mathbf{v}_{\alpha _{i}}& =\left( id_{\alpha _{i}}\otimes \boldsymbol{%
\varphi }^{(D\setminus \alpha _{i})}\right) (\mathbf{v}) \\
& =\sum_{i=1}^{r}\left( id_{\alpha _{i}}\otimes \boldsymbol{\psi }%
_{i}^{(\alpha \setminus \alpha _{i})}\otimes \boldsymbol{\phi }_{i}^{(\gamma
)}\right) (\mathbf{v}) \\
& =\sum_{i=1}^{r}\left( id_{\alpha _{i}}\otimes \boldsymbol{\psi }%
_{i}^{(\alpha \setminus \alpha _{i})}\right) \left( (id_{\alpha }\otimes 
\boldsymbol{\phi }_{i}^{(\gamma )})(\mathbf{v})\right) .
\end{align*}%
Observe that $(id_{\alpha }\otimes \boldsymbol{\phi }_{l}^{(\gamma )})(%
\mathbf{v})\in U_{\alpha }^{\min }(\mathbf{v}).$ Hence the other inclusion
holds and the first equality of  statement (b) is proved. To show the
second inequality of statement (b), we proceed in a similar way by using Theorem~\ref{Satz Vjmin}(b)
and the definition of $U_{\alpha _{j}}^{III}(\mathbf{v}).$
\end{proof}

\bigskip

From now on, given $\emptyset \neq \alpha \subset D,$ we will denote $%
\mathbf{V}_{\alpha }:=\left. _{a}\bigotimes_{j\in \alpha }V_{j}\right. ,$ $%
r_{\alpha }:=\dim U_{\alpha }^{\min }(\mathbf{v})$ and $U_{D}^{\min }(%
\mathbf{v}):=\mathrm{span}\,\{\mathbf{v}\}.$ Observe that for each $\mathbf{v%
} \in \mathbf{V}_D$ we have that $(\dim {U}_{\alpha}^{\min}(\mathbf{v}%
))_{\alpha \in 2^D \setminus \{\emptyset\}}$ is in $\mathbb{N}^{2^{\#D}-1}.$

\begin{definition}
Let $D$ be a finite index set and $T_{D}$ be a partition tree. Let $V_{j}$
be a vector space for $j\in D,$ Assume that $\{\mathbf{V}_{\alpha
}\}_{\alpha \in T_{D}\setminus \{D\}}$ is a representation of the tensor
space $\mathbf{V}_{D}=\left. _{a}\bigotimes_{\alpha \in S(D)}\mathbf{V}%
_{\alpha }\right. $ in the tree-based format. Then for each $\mathbf{v}\in 
\mathbf{V}_{D}=\left. _{a}\bigotimes_{j\in D}V_{j}\right. $ we define its 
\emph{tree-based rank (TB rank)} by the tuple $(\dim {U}_{\alpha }^{\min }(%
\mathbf{v}))_{\alpha \in T_{D}}\in \mathbb{N}^{\#T_{D}}. $
\end{definition}

In order to characterise the tensors $\mathbf{v}\in \mathbf{V}_{D}$
satisfying $(\dim \mathbf{U}_{\alpha }^{\min }(\mathbf{v}))_{\alpha \in
T_{D}}=\mathfrak{r},$ for a fixed $\mathfrak{r}:=(r_{\alpha })_{\alpha \in
T_{D}}\in \mathbb{N}^{\#T_{D}},$ we introduce the following definition.

\begin{definition}
We will say that $\mathfrak{r}:=(r_{\alpha })_{\alpha \in T_{D}}\in \mathbb{N%
}^{\#T_{D}}$ is an \emph{admissible tuple for $T_{D}$,} if there exists $%
\mathbf{v}\in \mathbf{V}_{D}\setminus \{\mathbf{0}\}$ such that $\dim
U_{\alpha }^{\min }(\mathbf{v})=r_{\alpha }$ for all $\alpha \in
T_{D}\setminus \{D\}.$
\end{definition}

Necessary conditions for $\mathfrak{r}\in \mathbb{N}^{\#T_{D}}$ to be
admissible are 
\begin{equation*}
\begin{array}{ll}
r_{D}=1, &  \\ 
r_{\{j\} }\leq \dim V_{j} & \text{ for }\{j\}\in \mathcal{L}(T_{D}), \\ 
r_{\alpha} \le \prod_{\beta \in S(\alpha)}r_{\beta} & \text{ for }\alpha \in
T_D \setminus \mathcal{L}(T_{D}), \\ 
r_{\delta} \le r_{\alpha}\prod_{\substack{ \beta \in S(\alpha)\setminus
\{\delta\}}}r_{\beta} & \text{ for }\alpha \in T_D \setminus \mathcal{L}%
(T_{D}) \text{ and } \delta \in S(\alpha).%
\end{array}%
\end{equation*}

\subsection{The representations of tensors of fixed TB rank}

\label{Sec_Hierar}

Before introducing the representation of a tensor of fixed TB rank we need
to define the set of coefficients of that tensors. To this end, we recall
the definition of the `matricisation' (or `unfolding') of a tensor in a
finite-dimensional setting.

\begin{definition}
\label{Def Malpha}For $\alpha \subset 2^D,$ and $\beta \subset \alpha$ the
map $\mathcal{M}_{\beta}$ is defined as the isomorphism%
\begin{equation*}
\begin{tabular}{llll}
$\mathcal{M}_{\beta}:$ & $\mathbb{R}^{%
\mathop{\mathchoice{\raise-0.22em\hbox{\huge $\times$}} {\raise-0.05em\hbox{\Large $\times$}}{\hbox{\large
$\times$}}{\times}}_{\mu \in \alpha}r_{\mu}} $ & $\rightarrow $ & $\mathbb{R}%
^{\left(\prod_{\mu \in \beta}r_{\mu } \right)\times \left( \prod_{\delta \in
\alpha \setminus \beta}r_{\delta}\right) },$ \\ 
& $C_{(i_{\mu})_{\mu \in \alpha}}$ & $\mapsto $ & $C_{(i_{\mu})_{\mu \in
\beta},(i_{\delta})_{\delta \in \alpha \setminus \beta}}$%
\end{tabular}%
\ \ \ \ \ 
\end{equation*}
\end{definition}

It allows to introduce the following definition.

\begin{definition}
For $\alpha \subset 2^D,$ let $C^{(\alpha)}\in \mathbb{R}^{%
\mathop{\mathchoice{\raise-0.22em\hbox{\huge $\times$}}
{\raise-0.05em\hbox{\Large $\times$}}{\hbox{\large
$\times$}}{\times}}_{\mu \in \alpha}r_{\mu }}.$ We say that $C^{(\alpha)}\in 
\mathbb{R}_*^{%
\mathop{\mathchoice{\raise-0.22em\hbox{\huge $\times$}}
{\raise-0.05em\hbox{\Large $\times$}}{\hbox{\large
$\times$}}{\times}}_{\mu \in \alpha}r_{\mu }}$ if and only if 
\begin{align*}
\prod_{\mu \in \alpha} \left(\det \left(\mathcal{M}_{\mu}(C^{(\alpha)}) 
\mathcal{M}_{\mu}(C^{(\alpha)})^T \right)+\det \left(\mathcal{M}%
_{\mu}(C^{(\alpha)})^T \mathcal{M}_{\mu}(C^{(\alpha)}) \right)\right) > 0,
\end{align*}
where $\mathcal{M}_{\mu}(C^{(\alpha)}) \in \mathbb{R}^{r_{\mu} \times \left(
\prod_{\delta \in \alpha\setminus \{\mu\}}r_{\delta }\right) }$ for each $%
\mu \in \alpha.$ We point out that this condition is equivalent to the condition that all $%
\mathcal{M}_{\mu}(C^{(\alpha)})$ have maximal rank.
\end{definition}

Since the determinant is a continuous function, $\mathbb{R}_*^{%
\mathop{\mathchoice{\raise-0.22em\hbox{\huge $\times$}}
{\raise-0.05em\hbox{\Large $\times$}}{\hbox{\large
$\times$}}{\times}}_{\mu \in \alpha}r_{\mu }}$ is an open set in $\mathbb{R}%
^{%
\mathop{\mathchoice{\raise-0.22em\hbox{\huge $\times$}}
{\raise-0.05em\hbox{\Large $\times$}}{\hbox{\large
$\times$}}{\times}}_{\mu \in \alpha}r_{\mu }}.$

\bigskip

\begin{definition}
Let $T_{D}$ be a given dimension partition tree and fix some tuple $%
\mathfrak{r}\in \mathbb{N}^{T_{D}}$. Then \emph{the set of TBF
tensors of fixed TB rank $\mathfrak{r}$} is defined by 
\begin{equation}
\mathcal{FT}_{\mathfrak{r}}(\mathbf{V}_{D}):=\left\{ \mathbf{v}\in  \mathbf{V}_{D}:\dim {U}_{\alpha }^{\min }(\mathbf{v}%
)=r_{\alpha }\text{ for all }\alpha \in T_{D}\right\}
\end{equation}%
and the \emph{set of TBF tensors of bounded TB rank $\mathfrak{r}$} is
defined by 
\begin{equation}
\mathcal{FT}_{\leq \mathfrak{r}}(\mathbf{V}_{D}):=\left\{ \mathbf{v}\in 
\mathbf{V}_{D}:%
\begin{array}{l}
\dim {U}_{\alpha }^{\min }(\mathbf{v})\leq r_{\alpha }\text{ for all }\alpha
\in T_{D}%
\end{array}%
\right\} .  \label{(Hr}
\end{equation}
\end{definition}

Note that $\mathcal{FT}_{\mathfrak{r}}(\mathbf{V}_{D})=\emptyset $ for an
inadmissible tuple $\mathfrak{r}.$ For $\mathfrak{r},\mathfrak{s}\in \mathbb{%
N}^{T_{D}}$ we write $\mathfrak{s}\leq \mathfrak{r}$ if and only if $%
s_{\alpha }\leq r_{\alpha }$ for all $\alpha \in T_{D}.$ Then we can also
use the following notation 
\begin{equation}  \label{connected_id}
\mathcal{FT}_{\le \mathfrak{r}}(\mathbf{V}_{D}) := \{\mathbf{0}\} \cup
\bigcup_{\mathfrak{s}\leq \mathfrak{r}}\mathcal{FT}_{\mathfrak{s}}(\mathbf{V}%
_{D}).
\end{equation}%
Next we give some useful examples.

\begin{example}[Tucker format]
\label{example_tucker} Consider the dimension partition tree of $D:=\{1,\ldots ,d\},$
where $S(D)=\mathcal{L}(T_{D})=\{\{j\}:1\leq j\leq d\}.$ Let $%
(r_{D},r_{1},\ldots ,r_{d})$ be admissible, then $r_{D}=1$ and $r_{j}\leq
\dim V_{j}$ for $1\leq j\leq d.$ Thus we can write 
\begin{equation*}
\mathcal{FT}_{\le (1,r_{1},\ldots ,r_{d})}(\mathbf{V}_{D})=\mathcal{T}%
_{(r_{1},\ldots ,r_{d})}(\mathbf{V}_{D})
\end{equation*}%
and 
\begin{equation*}
\mathcal{FT}_{(1,r_{1},\ldots ,r_{d})}(\mathbf{V}_{D})=\mathcal{M}%
_{(r_{1},\ldots ,r_{d})}(\mathbf{V}_{D}).
\end{equation*}
\end{example}

\begin{example}[Tensor Train format]
\label{example_TT} Consider a binary partition tree of $D:=\{1,\ldots ,d\}$
given by 
\begin{equation*}
T_{D}=\{D,\{\{j\}:1\leq j\leq d\},\{\{j+1,\ldots ,d\}:1\leq j\leq d-2\}\}.
\end{equation*}%
In particular, $S(\{j,\ldots ,d\})=\{\{j\},\{j+1,\ldots ,d\}\}$ for $1\leq
j\leq d-1.$ This tree-based format is related to the following chain of
inclusions: 
\begin{equation*}
\mathbf{U}_{D}^{\min }(\mathbf{v})\subset \mathbf{U}_{1}^{\min }(\mathbf{v}%
)\otimes _{a}\mathbf{U}_{2\cdots d}^{\min }(\mathbf{v})\subset \mathbf{U}%
_{1}^{\min }(\mathbf{v})\otimes _{a}\mathbf{U}_{2}^{\min }(\mathbf{v}%
)\otimes _{a}\mathbf{U}_{3\cdots d}^{\min }(\mathbf{v})\subset \cdots
\subset \left. _{a}\bigotimes_{j\in D}\mathbf{U}_{j}^{\min }(\mathbf{v}%
)\right. .
\end{equation*}
\end{example}

\bigskip

The next result gives us a characterisation of the tensors in $\mathcal{FT}_{%
\mathfrak{r}}(\mathbf{V}_{D}).$

\bigskip

\begin{theorem}
\label{characterization_FT} Let $V_{j}$ be vector spaces for $j\in D$ and $%
T_{D}$ be a dimension partition tree of $D$. Then the following two statements are equivalent.

\begin{enumerate}
\item[(a)] $\mathbf{v}\in \mathcal{FT}_{\mathfrak{r}}(\mathbf{V}_{D}).$

\item[(b)] There exists $\{u_{i_k}^{(k)}:1 \le i_k \le r_k\}$ a basis of $%
U_k^{\min}(\mathbf{v})$ for $k\in \mathcal{L}(T_D)$ where for each $\mu \in
T_D \setminus \mathcal{L}(T_D)$ there exists a unique ${C}^{(\mu)}
\in \mathbb{R}_{\ast }^{r_{\mu} \times 
\mathop{\mathchoice{\raise-0.22em\hbox{\huge $\times$}}
{\raise-0.05em\hbox{\Large $\times$}}{\hbox{\large
$\times$}}{\times}}_{\beta \in S(\mu)}r_{\beta }}$ such that the set $\{%
\mathbf{u}_{i_{\mu }}^{(\mu )}:1\leq i_{\mu }\leq r_{\alpha }\},$ with
\begin{equation}
\mathbf{u}_{i_{\mu }}^{(\mu )}=\sum_{\substack{ 1\leq i_{\beta }\leq
r_{\beta }  \\ \beta \in S(\mu )}}C_{i_{\mu },(i_{\beta })_{\beta \in S(\mu
)}}^{(\mu )}\bigotimes_{\beta \in S(\mu )}\mathbf{u}_{i_{\beta }}^{(\beta )}
\label{TBRT2}
\end{equation}%
for $1\leq i_{\mu }\leq r_{\mu },$ is a basis of $U_{\mu}^{\min }(\mathbf{v}%
) $ and 
\begin{equation}
\mathbf{v}=\sum_{\substack{ 1\leq i_{\alpha }\leq r_{\alpha }  \\ \alpha \in
S(D)}}C_{(i_{\alpha })_{\alpha \in S(D)}}^{(D)}\bigotimes_{\alpha \in S(D)}%
\mathbf{u}_{i_{\alpha }}^{(\alpha )}.  \label{TBRT1}
\end{equation}
\end{enumerate}

Furthermore, if $\mathbf{v}\in \mathcal{FT}_{\mathfrak{r}}(\mathbf{V}_{D})$
then \eqref{TBRT1} can be written for each $\alpha \in S(D)$ as 
\begin{equation}  \label{binary_representation_level0}
\mathbf{v}=\sum_{\substack{ 1\leq i_{\alpha }\leq r_{\alpha }}}\mathbf{u}%
_{i_{\alpha }}^{(\alpha )}\otimes \mathbf{U}_{i_{\alpha }}^{(\alpha )}.
\end{equation}%
where $U_{S(D)\setminus \{\alpha\}}^{\min}(\mathbf{v}) = \mathrm{span}\, \{%
\mathbf{U}_{i_{\alpha }}^{(\alpha )}:1\leq i_{\alpha }\leq r_{\alpha }\},$
and for each $\mu \in T_D \setminus \mathcal{L}(T_D)$ we have 
\begin{equation*}
\mathbf{u}_{i_{\mu }}^{(\mu )}=\sum_{\substack{ 1\leq i_{\beta }\leq
r_{\beta }}}\mathbf{u}_{i_{\beta }}^{(\beta )}\otimes \mathbf{U}_{i_{\mu
},i_{\beta }}^{(\beta )},
\end{equation*}%
where 
\begin{equation}
\mathbf{U}_{i_{\mu },i_{\beta }}^{(\beta )}:=\sum_{\substack{ 1\leq
i_{\delta }\leq r_{\delta }  \\ \delta \in S(\mu )  \\ \delta \neq \beta }}%
C_{i_{\mu },(i_{\delta })_{\delta \in S(\mu )}}^{(\mu )}\bigotimes 
_{\substack{ \delta \in S(\mu )  \\ \delta \neq \beta}}\mathbf{u}_{i_{\delta
}}^{(\delta )},  \label{Ualphas1}
\end{equation}%
and $U_{S(\mu) \setminus \{\beta\} }^{\min }(\mathbf{v}) = U_{S(\mu)
\setminus \{\beta\} }^{\min }(\mathbf{u}_{i_{\mu}}^{(\mu)})=\mathrm{span}%
\,\left\{ \mathbf{U}_{i_{\mu },i_{\beta }}^{(\beta )}: 1 \le i_{\beta} \le
r_{\beta} \right\}$ for $1\leq i_{\mu }\leq r_{\mu }.$
\end{theorem}

\begin{proof}
Assuming first that (b) is true, (a) follows by the definition of $\mathcal{%
FT}_{\mathfrak{r}}(\mathbf{V}_{D})$. Now, assume that (a) holds. Since $%
\mathbf{v}\in \left. _{a}\bigotimes_{\alpha \in S(D)}U_{\alpha }^{\min }(%
\mathbf{v})\right. ,$ there exists a unique $C^{(D)}\in \mathbb{R}^{%
\mathop{\mathchoice{\raise-0.22em\hbox{\huge $\times$}}
{\raise-0.05em\hbox{\Large $\times$}}{\hbox{\large
$\times$}}{\times}}_{\alpha \in S(D)}r_{\alpha }}$ such that 
\begin{equation*}
\mathbf{v}=\sum_{\substack{ 1\leq i_{\alpha }\leq r_{\alpha }  \\ \alpha \in
S(D)}}C_{(i_{\alpha })_{\alpha \in S(D)}}^{(D)}\bigotimes_{\alpha \in S(D)}%
\mathbf{u}_{i_{\alpha }}^{(\alpha )},
\end{equation*}%
where $\{\mathbf{u}_{i_{\alpha }}^{(\alpha )}:1\leq i_{\alpha }\leq
r_{\alpha }\}$ is a basis\footnote{There are a small issue with the bold notation when $\alpha \in S(\mu)$ and $\alpha$ is a leaf, then $u_{i_\alpha}^{(\alpha)}$ should not be bold.} of $U_{\alpha }^{\min }(\mathbf{v}).$ For each $%
\alpha \in S(D)$ we set 
\begin{equation}
\mathbf{U}_{i_{\alpha }}^{(\alpha )}:=\sum_{\substack{ 1\leq i_{\beta }\leq
r_{\beta }  \\ \beta \in S(D)  \\ \beta \neq \alpha }}C_{(i_{\beta })_{\beta
\in S(D)}}^{(D)}\bigotimes_{\substack{ \beta \in S(D)  \\ \beta \neq \alpha 
}}u_{i_{\beta }}^{(\beta )},  \label{Ualphas}
\end{equation}%
then \eqref{TBRT1} can be written as \eqref{binary_representation_level0}.
From the definition of minimal subspaces we can write 
\begin{equation*}
U_{S(D)\setminus \{\alpha \}}^{\min }(\mathbf{v})=\{(\mathbf{id}_{[\alpha
]}\otimes \varphi _{\alpha })(\mathbf{v}):\varphi _{\alpha }\in U_{\alpha
}^{\min }(\mathbf{v})^{\ast }\}.
\end{equation*}%
We claim that $\{\mathbf{U}_{i_{\alpha }}^{(\alpha )}:1\leq i_{\alpha }\leq
r_{\alpha }\}$ is a basis of $U_{S(D)\setminus \{\alpha \}}^{\min }(\mathbf{v%
}).$ To prove the claim assume that $\mathbf{U}_{1}^{(\alpha )}$ is a linear
combination of $\{\mathbf{U}_{i_{\alpha }}^{(\alpha )}:2\leq i_{\alpha }\leq
r_{\alpha }\},$ then $\mathbf{U}_{1}^{(\alpha )}=\sum_{2\leq i_{\alpha }\leq
r_{\alpha }}\lambda _{i_{\alpha }}\mathbf{U}_{i_{\alpha }}^{(\alpha )}$
where $\lambda _{i_{\alpha }}\neq 0$ for some $2\leq i_{\alpha }\leq
r_{\alpha }.$ Thus, 
\begin{equation*}
\mathbf{v}=\sum_{\substack{ 2\leq i_{\alpha }\leq r_{\alpha }}}(\mathbf{u}%
_{i_{\alpha }}^{(\alpha )}+\lambda _{i_{\alpha }}\mathbf{u}_{1}^{(\alpha
)})\otimes \mathbf{U}_{i_{\alpha }}^{(\alpha )},
\end{equation*}%
since $\{\mathbf{u}_{i_{\alpha }}^{(\alpha )}+\lambda _{i_{\alpha }}\mathbf{u%
}_{1}^{(\alpha )}:2\leq i_{\alpha }\leq r_{\alpha }\}$ are linearly
independent we have $\dim U_{\alpha }^{\min }(\mathbf{v})<r_{\alpha },$ a
contradiction. Since $\{\mathbf{U}_{i_{\alpha }}^{(\alpha )}:1\leq i_{\alpha
}\leq r_{\alpha }\}$ are linearly independent for each $\alpha \in S(D),$
from \eqref{binary_representation_level0} we have that 
\begin{equation*}
U_{S(D)\setminus \{\alpha \}}^{\min }(\mathbf{v})=\mathrm{span}\,\{\mathbf{U}%
_{i_{\alpha }}^{(\alpha )}:1\leq i_{\alpha }\leq r_{\alpha }\},
\end{equation*}%
and from \eqref{Ualphas}, we deduce that $\mathcal{M}_{\alpha }(C^{(D)})$
maps a basis into another one for each $\alpha \in S(D)$ and hence $%
C^{(D)}\in \mathbb{R}_{\ast }^{%
\mathop{\mathchoice{\raise-0.22em\hbox{\huge $\times$}}
{\raise-0.05em\hbox{\Large $\times$}}{\hbox{\large
$\times$}}{\times}}_{\beta \in S(D)}r_{\beta }}.$ In consequence, when $S(D)=%
\mathcal{L}(T_{D})$ statement (a) holds and then \eqref{TBRT1} gives us the
classical Tucker representation.

Next, assume $S(D)\neq \mathcal{L}(T_{D}).$ Then, for each $\mu \in
T_{D}\setminus \{D\}$ such that $S(\mu )\neq \emptyset ,$ thanks to
Proposition~\ref{inclusin_Umin}, we have 
\begin{equation*}
U_{\mu }^{\min }(\mathbf{v})\subset \left. _{a}\bigotimes_{\beta \in S(\mu
)}U_{\beta }^{\min }(\mathbf{v})\right. .
\end{equation*}%
Consider $\{\mathbf{u}_{i_{\mu }}^{(\mu )}:1\leq i_{\mu }\leq r_{\mu }\}$ a
basis of $U_{\mu }^{\min }(\mathbf{v})$ and $\{\mathbf{u}_{i_{\beta
}}^{(\beta )}:1\leq i_{\beta }\leq r_{\beta }\}$ a basis of $U_{\beta
}^{\min }(\mathbf{v})$ for $\beta \in S(\mu )$ and $1\leq i_{\mu }\leq
r_{\mu }.$ Then, there exists a unique $C^{(\mu )}\in \mathbb{R}^{r_{\mu
}\times \left( 
\mathop{\mathchoice{\raise-0.22em\hbox{\huge $\times$}}
{\raise-0.05em\hbox{\Large $\times$}}{\hbox{\large
$\times$}}{\times}}_{\beta \in S(\alpha )}r_{\beta }\right) }$ such that 
\begin{equation*}
\mathbf{u}_{i_{\mu }}^{(\mu )}=\sum_{\substack{ 1\leq i_{\beta }\leq
r_{\beta }  \\ \beta \in S(\mu )}}C_{i_{\mu },(i_{\beta })_{\beta \in S(\mu
)}}^{(\mu )}\bigotimes_{\beta \in S(\mu )}\mathbf{u}_{i_{\beta }}^{(\beta )},
\end{equation*}%
for $1\leq i_{\mu }\leq r_{\mu }.$ Since $\{\mathbf{u}_{i_{\mu }}^{(\mu
)}:1\leq i_{\mu }\leq r_{\mu }\}$ is a basis, we can identify $C^{(\mu )}$
with the matrix $\mathcal{M}_{\mu }(C^{(\mu )}),$ in the non-compact Stiefel
manifold $\mathbb{R}_{\ast }^{r_{\mu }\times \left( \prod_{\beta \in S(\mu
)}r_{\beta }\right) },$ which is the set of matrices in $\mathbb{R}^{r_{\mu
}\times \left( \prod_{\beta \in S(\alpha )}r_{\beta }\right) }$ whose rows
are linearly independent (see 3.1.5 in \cite{AMS}). From \eqref{TBRT1} and %
\eqref{TBRT2} we obtain the Tucker representation of $\mathbf{v},$ when $%
S(D)\neq \mathcal{L}(T_{D}),$ as 
\begin{equation}
\mathbf{v}=\sum_{\substack{ 1\leq i_{k}\leq r_{k}  \\ k\in \mathcal{L}%
(T_{D}) }}\left( \sum_{\substack{ 1\leq i_{\alpha }\leq r_{\alpha }  \\ %
\alpha \in T_{D}\setminus \{D\}  \\ \alpha \notin \mathcal{L}(T_{D})}}%
C_{(i_{\alpha })_{\alpha \in S(D)}}^{(D)}\prod_{\substack{ \mu \in
T_{D}\setminus \{D\}  \\ S(\mu )\neq \emptyset }}C_{i_{\mu },(i_{\beta
})_{\beta \in S(\mu )}}^{(\mu )}\right) \bigotimes_{k\in \mathcal{L}%
(T_{D})}u_{i_{k}}^{(k)},  \label{Tucker_tree_representation}
\end{equation}%
here $\{u_{i_{k}}^{(k)}:1\leq i_{k}\leq r_{k}\}$ is a basis of $U_{k}^{\min
}(\mathbf{v})$ for each $k\in \mathcal{L}(T_{D}).$ To conclude, we claim
that $C^{(\mu )}\in \mathbb{R}_{\ast }^{r_{\mu }\times \left( 
\mathop{\mathchoice{\raise-0.22em\hbox{\huge $\times$}}
{\raise-0.05em\hbox{\Large $\times$}}{\hbox{\large
$\times$}}{\times}}_{\beta \in S(\alpha )}r_{\beta }\right) }$ for all $\mu
\in T_{D}\setminus \mathcal{L}(T_{D}).$ To prove the claim we proceed in a
similar way as in the root case, for each fixed $1\leq i_{\mu }\leq r_{\mu }$
and $\beta \in S(\mu ),$ we introduce \eqref{Ualphas1}. Hence, we can write (%
\ref{TBRT2}) as 
\begin{equation*}
\mathbf{u}_{i_{\mu }}^{(\mu )}=\sum_{\substack{ 1\leq i_{\beta }\leq
r_{\beta }}}\mathbf{u}_{i_{\beta }}^{(\beta )}\otimes \mathbf{U}_{i_{\mu
},i_{\beta }}^{(\beta )},
\end{equation*}%
where $1\leq i_{\mu }\leq r_{\mu }$ and $\beta \in S(\mu ).$ From
Proposition~\ref{(Ualpha in Tensor Uj coro}(a), we have 
\begin{align*}
U_{\beta }^{\min }(\mathbf{v})& =\mathrm{span}\,\left\{ (id_{\beta }\otimes
\varphi ^{(\mu \setminus \beta )})(\mathbf{u}_{i_{\mu }}^{(\mu )}):1\leq
i_{\mu }\leq r_{\mu }\text{ and }\varphi ^{(\mu \setminus \beta )}\in \left.
_{a}\bigotimes_{\delta \in S(\mu )\setminus \{\beta \}}U_{\delta }^{\min }(%
\mathbf{v})^{\prime }\right. \right\} \\
& =\mathrm{span}\,\left\{ (id_{\beta }\otimes \varphi ^{(\mu \setminus \beta
)})(\mathbf{u}_{i_{\mu }}^{(\mu )}):1\leq i_{\mu }\leq r_{\mu }\text{ and }%
\varphi ^{(\mu \setminus \beta )}\in \left. _{a}\bigotimes_{\delta \in S(\mu
)\setminus \{\beta \}}\mathbf{V}_{\delta }^{\prime }\right. \right\} ,
\end{align*}%
and hence $U_{\beta }^{\min }(\mathbf{u}_{i_{\mu }}^{(\mu )})\subset
U_{\beta }^{\min }(\mathbf{v})$ for $1\leq i_{\mu }\leq r_{\mu }.$ Let us
consider $\{\varphi _{i_{\beta }}^{(\beta )}:1\leq i_{\beta }\leq r_{\beta
}\}\subset U_{\beta }^{\min }(\mathbf{v})^{\prime }$ a dual basis of the
finite-dimensional space $\{\mathbf{u}_{i_{\beta }}^{(\beta )}:1\leq
i_{\beta }\leq r_{\beta }\},$ that is, $\varphi _{i_{\beta }}^{(\beta )}(%
\mathbf{u}_{j_{\beta }}^{(\beta )})=\delta _{i_{\beta },j_{\beta }}$ for all 
$1\leq i_{\beta },j_{\beta }\leq r_{\beta },$ and $\beta \in S(\mu ).$ Thus,
we have 
\begin{equation*}
\left( id_{\beta }\otimes \bigotimes_{\substack{ \delta \in S(\mu )  \\ %
\delta \neq \beta }}\varphi _{j_{\delta }}^{(\delta )}\right) (\mathbf{u}%
_{i_{\mu }}^{(\mu )})=\sum_{\substack{ 1\leq j_{\beta }\leq r_{\beta }}}%
C_{i_{\mu },(j_{\delta })_{\delta \in S(\mu )}}^{(\mu )}\mathbf{u}_{j_{\beta
}}^{(\beta )}\in U_{\beta }^{\min }(\mathbf{v})
\end{equation*}%
for each multi-index $(j_{\delta })_{\delta \in S(\mu )\setminus \beta }\in 
\mathop{\mathchoice{\raise-0.22em\hbox{\huge $\times$}}
{\raise-0.05em\hbox{\Large $\times$}}{\hbox{\large
$\times$}}{\times}}_{\substack{ \delta \in S(\mu )  \\ \delta \neq \beta }}%
\{1,\ldots ,r_{\delta }\}.$ Then, for $\beta \in S(\mu )$, 
\begin{equation*}
U_{\beta }^{\min }(\mathbf{v})=\mathrm{span}\,\left\{ \left( id_{\beta
}\otimes \bigotimes_{\substack{ \delta \in S(\mu )  \\ \delta \neq \beta }}%
\varphi _{j_{\delta }}^{(\delta )}\right) (\mathbf{u}_{i_{\mu }}^{(\mu
)}):(j_{\delta })_{\delta \in S(\mu )\setminus \beta }\in 
\mathop{\mathchoice{\raise-0.22em\hbox{\huge $\times$}} {\raise-0.05em\hbox{\Large $\times$}}{\hbox{\large
$\times$}}{\times}}_{\substack{ \delta \in S(\mu )  \\ \delta \neq \beta }}%
\{1,\ldots ,r_{\delta }\},\,1\leq i_{\mu }\leq r_{\mu }\right\}
\end{equation*}%
with $\dim U_{\beta }^{\min }(\mathbf{v})=r_{\beta }$ if and only if $%
\mathrm{rank}\,\mathcal{M}_{\beta }(C^{(\mu )})=r_{\beta }$ for $\beta \in
S(\mu ).$ Finally, we have $C^{(\mu )}\in \mathbb{R}_{\ast }^{r_{\mu }\times
\left( 
\mathop{\mathchoice{\raise-0.22em\hbox{\huge $\times$}}
{\raise-0.05em\hbox{\Large $\times$}}{\hbox{\large
$\times$}}{\times}}_{\delta \in S(\mu )}r_{\delta }\right) }$ for all $\mu
\in T_{D}\setminus \mathcal{L}(T_{D})$ and the claim follows. Thus,
statement (b) holds.

\bigskip

To end the proof of the theorem, observe that in a similar way as above and
by using $id_{S(\mu) \setminus \beta} \otimes \varphi _{j_{\beta}}^{(\beta)}$
for $1 \le j_{\beta} \le r_{\beta},$ over $\mathbf{u}_{i_{\mu}}^{(\mu)}$ it
can be proved that 
\begin{equation*}
U_{S(\mu) \setminus \{\beta\} }^{\min }(\mathbf{u}_{i_{\mu}}^{(\mu)})=%
\mathrm{span}\,\left\{ \mathbf{U}_{i_{\mu },i_{\beta }}^{(\beta )}: 1 \le
i_{\beta} \le r_{\beta} \right\}
\end{equation*}
for $1\leq i_{\mu }\leq r_{\mu }$ and also 
\begin{equation*}
U_{S(\mu) \setminus \{\beta\} }^{\min }(\mathbf{v})=\mathrm{span}\,\left\{ 
\mathbf{U}_{i_{\mu },i_{\beta }}^{(\beta )}: 1 \le i_{\beta} \le r_{\beta},
\, 1\leq i_{\mu }\leq r_{\mu }. \right\}.
\end{equation*}
Now, we claim that $\left\{ \mathbf{U}_{i_{\mu },i_{\beta }}^{(\beta )}: 1
\le i_{\beta} \le r_{\beta} \right\}$ are linearly independent in $\left._a
\bigotimes_{\delta \neq \beta } \mathbf{V}_{\delta} \right.$ for $1 \le
i_{\mu} \le r_{\mu}$ and $\beta \in S(\mu).$ Otherwise, there exist $%
\lambda_{i_{\beta}}$ for $1 \le i_{\beta} \le r_{\beta}$ not all identically
zero such that $\sum_{1 \le i_{\beta} \le r_{\beta}}\lambda_{i_{\beta}} 
\mathbf{U}_{i_{\mu },i_{\beta }}^{(\beta )} = \mathbf{0}.$ Take $\mathbf{w}%
_{\beta} \in \mathbf{V}_{\beta}\setminus \{\mathbf{0}\}$ and then 
\begin{equation*}
\mathbf{w}_{\beta} \otimes \left(\sum_{1 \le i_{\beta} \le
r_{\beta}}\lambda_{i_{\beta}} \mathbf{U}_{i_{\mu },i_{\beta }}^{(\beta
)}\right) = \sum_{1 \le i_{\beta} \le r_{\beta}}\lambda_{i_{\beta}} \mathbf{w%
}_{\beta} \otimes \mathbf{U}_{i_{\mu },i_{\beta }}^{(\beta )} = \mathbf{0}.
\end{equation*}
Observe that 
\begin{equation*}
\sum_{\substack{ 1\leq i_{\beta }\leq r_{\beta }}} \left(
\lambda_{i_{\beta}} \mathbf{w}_{\beta} \otimes \mathbf{U}_{i_{\mu},i_{\beta
}}^{(\beta )}\right) = \sum_{\substack{ 1\leq i_{\delta }\leq r_{\delta } 
\\ \delta \in S(\mu ) }} C_{i_{\mu },(i_{\delta })_{\delta \in S(\mu
)}}^{(\mu )} \lambda_{i_{\beta}} \mathbf{w}_{\beta} \otimes \left(\bigotimes 
_{\substack{ \delta \neq \beta  \\ \delta \in S(\mu )}}\mathbf{u}_{i_{\delta
}}^{(\delta )}\right) = \mathbf{0},
\end{equation*}
for $1 \le i_{\mu} \le r_{\mu}$ and $\beta \in S(\mu ),$ take a dual basis
of $\{\varphi_{i_{\delta }}^{(\delta )}: 1 \le i_{\delta} \le r_{\delta}\}
\subset \mathbf{V}_{\delta}^*$ of $\{\mathbf{u}_{i_{\delta }}^{(\delta )}: 1
\le i_{\delta} \le r_{\delta}\} \subset \mathbf{V}_{\delta} $ where $%
\varphi_{i_{\delta }}^{(\delta )}(\mathbf{u}_{j_{\delta }}^{(\delta )}) =
\delta_{i_{\delta},j_{\delta}}$ for all $1 \le i_{\delta},j_{\delta} \le
r_{\delta}.$ Then we obtain 
\begin{align*}
id_{\beta} \otimes \left(\bigotimes_{\substack{ \delta \in S(\mu ) \setminus
\{\beta\}}}\varphi_{i_{\delta }}^{(\delta )}\right) \left(\sum _{\substack{ %
1\leq i_{\beta }\leq r_{\beta }}} \left( \lambda_{i_{\beta}} \mathbf{w}%
_{\beta} \otimes \mathbf{U}_{i_{\mu},i_{\beta }}^{(\beta )}\right)\right) =
\sum_{1\le i_\beta\le r_\beta } C_{i_{\mu },(i_{\delta })_{\delta \in S(\mu
)}}^{(\mu )} \lambda_{i_{\beta}} \mathbf{w}_{\beta} = \mathbf{0},
\end{align*}
that is, $\mathcal{M}_{\beta}(C^{(\mu)})^T \mathbf{z}_{\beta} = \mathbf{0},$
where $\mathbf{z}_{\beta}:=(\lambda_{i_{\beta}} \mathbf{w}%
_{\beta})_{i_{\beta}=1}^{r_{\beta}}.$ Since $\mathrm{rank}\, \mathcal{M}%
_{\beta}(C^{(\mu)}) = r_{\beta},$ then $\dim \mathrm{Ker}\, \mathcal{M}%
_{\beta}(C^{(\mu)})^T = 0,$ and hence $\mathbf{z}_{\beta}
=(\lambda_{i_{\beta}} \mathbf{w}_{\beta} )_{i_{\beta}=1}^{r_{\beta}} = (%
\mathbf{0})_{i_{\beta}=1}^{r_{\beta}}$ for $\beta \in S(\gamma),$ a
contradiction. In consequence, 
\begin{equation*}
\dim U_{S(\mu) \setminus \{\beta\} }^{\min }(\mathbf{u}_{i_{\mu}}^{(\mu)}) =
\dim U_{\beta }^{\min }(\mathbf{u}_{i_{\mu}}^{(\mu)}) = r_{\beta}
\end{equation*}
for $1 \le i_{\mu} \le r_{\mu}$ and $\beta \in S(\mu ).$ Hence $U_{\beta
}^{\min }(\mathbf{v}) = U_{\beta }^{\min }(\mathbf{u}_{i_{\mu}}^{(\mu)}) $
holds for $1 \le i_{\mu} \le r_{\mu}$ and $\beta \in S(\mu ).$
\end{proof}

\section{Geometric structures for TBF tensors}
\label{sec:banach_manifold_tucker_fixed_rank}
Before characterising the "local coordinates" of a tensor $\mathbf{v}\in 
\mathcal{FT}_{\mathfrak{r}}(\mathbf{V}_{D})$ we need to introduce the
Banach-Grassmann manifold and its relatives.

\subsection{The Grassmann-Banach manifold and its relatives}

\label{preliminary}

In the following, $X$ is a Banach space with norm $\left\Vert
\cdot\right\Vert .$ The dual norm $\left\Vert \cdot\right\Vert _{X^{\ast}}$
of $X^{\ast}$ is%
\begin{equation}
\left\Vert \varphi\right\Vert _{X^{\ast}}=\sup\left\{ \left\vert
\varphi(x)\right\vert :x\in X\text{ with }\left\Vert x\right\Vert
_{X}\leq1\right\} =\sup\left\{ \left\vert \varphi(x)\right\vert /\left\Vert
x\right\Vert _{X}:0\neq x\in X\right\} .  \label{(Dualnorm}
\end{equation}

By $\mathcal{L}(X,Y)$ we denote the space of continuous linear mappings from 
$X$ into $Y.$ The corresponding operator norm is written as $\left\Vert
\cdot\right\Vert _{Y\leftarrow X}.$

\begin{definition}
Let $X$ be a Banach space. We say that $P \in \mathcal{L}(X,X)$ is a
projection if $P\circ P = P.$ In this situation we also say that $P$ is a
projection from $X$ onto $P(X):=\mathrm{Im}\,P$ parallel to $\mathrm{Ker}\, P.$
\end{definition}

From now on, we will denote $P \circ P = P^2.$ Observe that if $P$ is a
projection then $I_X-P$ is also a projection. Moreover, $I_X-P$ is parallel
to $\mathrm{Im}\,P.$

Observe that each projection gives rise to a pair of closed subspaces,
namely $U=\mathrm{Im}\,P$ and $V=\mathrm{Ker}\,P$ such that $X=U\oplus V.$
It allows us to introduce the following two definitions.

\begin{definition}
We will say that a subspace $U$ of a Banach space $X$ is a complemented
subspace if $U$ is closed and there exists $V$ in $X$ such that $X=U\oplus V$
and $V$ is also a closed subspace of $X.$ This subspace $V$ is called a
(topological) complement of $U$ and $(U,V)$ is a pair of complementary
subspaces.
\end{definition}

Corresponding to each pair $(U,V)$ of complementary subspaces, there is a
projection $P$ mapping $X$ onto $U$ along $V,$ defined as follows. Since for
each $x$ there exists a unique decomposition $x=u+v,$ where $u\in U$ and $%
v\in V,$ we can define a linear map $P(u+v):=u,$ where $\mathrm{Im}\,P=U$
and $\mathrm{Ker}\,P=V.$ Moreover, $P^{2}=P.$

\begin{definition}
The \emph{Grassmann manifold} of a Banach space $X,$ denoted by $\mathbb{G}%
(X),$ is the set of all complemented subspaces of $X.$
\end{definition}

$U\in \mathbb{G}(X)$ holds if and only if $U$ is a closed subspace and there
exists a closed subspace $V$ in $X$ such that $X=U\oplus V.$ Observe that $X$
and $\{0\}$ are in $\mathbb{G}(X).$ Moreover, by the proof of Proposition
4.2 of \cite{FHHM}, the following result can be shown.

\begin{proposition}
\label{characterize_P} Let $X$ be a Banach space. The following conditions
are equivalent:

\begin{enumerate}
\item[(a)] $U \in \mathbb{G}(X).$

\item[(b)] There exists $P \in \mathcal{L}(X,X)$ such that $P^2=P$ and $%
\mathrm{Im}\, P = U.$

\item[(c)] There exists $Q \in \mathcal{L}(X,X)$ such that $Q^2=Q$ and $%
\mathrm{Ker}\, Q = U.$
\end{enumerate}
\end{proposition}

Moreover, from Theorem 4.5 in \cite{FHHM}, the following result can be shown.

\begin{proposition}
\label{finite_dim_subspaces} Let $X$ be a Banach space. Then every
finite-dimensional subspace $U$ belongs to $\mathbb{G}(X).$
\end{proposition}

Let $V$ and $U$ be closed subspaces of a Banach space $X$ such that $%
X=U\oplus V.$ From now on, we will denote by $P_{_{U\oplus V}}$ the
projection onto $U$ along $V.$ 
Then we have $P_{_{V\oplus U}}=I_{X}-P_{_{U\oplus V}}.$ Let $U,U^{\prime
}\in \mathbb{G}(X).$ We say that $U$ and $U^{\prime }$ have a common
complementary subspace in $X,$ if $X=U\oplus W=U^{\prime }\oplus W$ for some 
$W\in \mathbb{G}(X).$ The following result will be useful (see Lemma 2.1 in 
\cite{DRW}).

\begin{lemma}
\label{Char_Projections} Let $X$ be a Banach space and assume that $W$, $U$,
and $U^{\prime }$ are in $\mathbb{G}(X).$ Then the following statements are
equivalent:

\begin{enumerate}
\item[(a)] $X=U\oplus W=U^{\prime }\oplus W,$ i.e., $U$ and $U^{\prime }$
have a common complement in $X$.

\item[(b)] $P_{_{U\oplus W}}|_{U^{\prime }}:U^{\prime }\rightarrow U$ has an
inverse.
\end{enumerate}

Furthermore, if $Q = \left(P_{_{U \oplus W}}|_{_{U^{\prime }}}\right)^{-1},$
then $Q$ is bounded and $Q=P_{_{U^{\prime }\oplus W}}|_{_{U}}.$
\end{lemma}

Next, we recall the definition of a Banach manifold.

\begin{definition}
\label{banach_manifold_definition} Let $\mathbb{M}$ be a set. An atlas of
class $C^p\, (p \ge 0)$ on $\mathbb{M} $ is a family of charts with some
indexing set $A,$ namely $\{(M_{\alpha},u_{\alpha}): \alpha \in A\},$ having
the following properties:

\begin{enumerate}
\item[AT1] $\{M_{\alpha}\}_{\alpha \in A}$ is a covering\footnote{%
The condition of an \emph{open} covering is not needed, see \cite{Lang}.} of 
$\mathbb{M},$ that is, $M_{\alpha} \subset \mathbb{M}$ for all $\alpha \in A$
and $\cup_{\alpha \in A}M_{\alpha}=\mathbb{M}.$

\item[AT2] For each $\alpha \in A, (M_{\alpha},u_{\alpha})$ stands for a
bijection $u_{\alpha}:M_{\alpha} \rightarrow U_{\alpha}$ of $M_{\alpha}$
onto an open set $U_{\alpha}$ of a Banach space $X_{\alpha},$ and for any $%
\alpha$ and $\beta$ the set $u_{\alpha}(M_{\alpha} \cap M_{\beta})$ is open
in $X_{\alpha}.$

\item[AT3] Finally, if we let $M_{\alpha} \cap M_{\beta} = M_{\alpha\beta}$
and $u_{\alpha}(M_{\alpha\beta} )= U_{\alpha\beta},$ the transition mapping $%
u_{\beta}\circ u_{\alpha}^{-1}:U_{\alpha\beta} \rightarrow U_{\beta \alpha}$
is a $C^p$-diffeomorphism.
\end{enumerate}
\end{definition}

Since different atlases can give the same manifold, we say that two atlases
are \emph{compatible} if each chart of one atlas is compatible with the
charts of the other atlas in the sense of AT3. One verifies that the
relation of compatibility between atlases is an equivalence relation.

\begin{definition}
An equivalence class of atlases of class $C^{p}$ on $\mathbb{M}$ is said to
define a structure of a $C^{p}$-Banach manifold on $\mathbb{M},$ and hence
we say that $\mathbb{M}$ is a Banach manifold. In a similar way, if an
equivalence class of atlases is given by analytic maps, then we say that $%
\mathbb{M}$ is an analytic Banach manifold. If $X_{\alpha }$ is a Hilbert
space for all $\alpha \in A,$ then we say that $\mathbb{M}$ is a Hilbert
manifold.
\end{definition}

In condition AT2 we do not require that the Banach spaces are the same for
all indices $\alpha,$ or even that they are isomorphic. If $X_{\alpha}$ is
linearly isomorphic to some Banach space $X$ for all $\alpha,$ we have the
following definition.

\begin{definition}
Let $\mathbb{M}$ be a set and $X$ be a Banach space. We say that $\mathbb{M}$
is a $C^{p}$ Banach manifold modelled on $X$ if there exists an atlas of
class $C^{p}$ over $\mathbb{M}$ with $X_{\alpha }$ linearly isomorphic to $X$
for all $\alpha \in A.$
\end{definition}

\begin{example}
Every Banach space is a Banach manifold modelled on itself (for a Banach
space $Y$, simply take $(Y,I_{Y})$ as atlas, where $I_{Y}$ is the identity
map on $Y$). In particular, the set of all bounded linear maps $\mathcal{L}%
(X,X)$ of a Banach space $X$ is also a Banach manifold modelled on itself.
\end{example}

If $X$ is a Banach space, then the set of all bounded linear automorphisms
of $X$ will be denoted by 
\begin{equation*}
\mathrm{GL}(X):=\left\{ A\in \mathcal{L}(X,X):A\text{ invertible }\right\} .
\end{equation*}

\begin{example}
If $X$ is a Banach space, then $\mathrm{GL}(X)$ is a Banach manifold
modelled on $\mathcal{L}(X,X),$ because it is an open set in $\mathcal{L}%
(X,X).$ Moreover, the map $A\mapsto A^{-1}$ is analytic (see 2.7 in \cite%
{Upmeier}).
\end{example}

The next example is a Banach manifold not modelled on a particular Banach
space.

\begin{example}[Grassmann--Banach manifold]
\label{Grassmann_Example} Let $X$ be a Banach space. Then, following \cite%
{Douady66} (see also \cite{Upmeier} and \cite{MRA}), it is possible to
construct an atlas for $\mathbb{G}(X).$ To do this, denote one of the
complements of $U\in \mathbb{G}(X)$ by $W, $ i.e., $X=U\oplus W$. Then we
define the \emph{Banach Grassmannian of $U$ relative to $W$} by 
\begin{equation*}
\mathbb{G}(W,X):=\left\{ V\in \mathbb{G}(X):X=V\oplus W\right\} .
\end{equation*}%
By using Lemma~\ref{Char_Projections} it is possible to introduce a
bijection 
\begin{equation*}
\Psi _{U\oplus W}:\mathbb{G}(W,X)\longrightarrow \mathcal{L}(U,W)
\end{equation*}%
defined by 
\begin{equation*}
\Psi _{U\oplus W}(U^{\prime }) = P_{W \oplus U}|_{U^{\prime }} \circ
P_{U^{\prime }\oplus W}|_{U} = P_{W \oplus U}|_{U^{\prime }} \circ
(P_{U\oplus W}|_{U^{\prime }})^{-1} .
\end{equation*}
It can be shown that the inverse 
\begin{equation*}
\Psi _{U\oplus W}^{-1}:\mathcal{L}(U,W)\longrightarrow \mathbb{G}(W,X),
\end{equation*}%
is given by 
\begin{equation*}
\Psi _{U\oplus W}^{-1}(L)=G(L):=\left\{ u+L(u):u\in U\right\} .
\end{equation*}%
Observe that $G(0)=U$ and $G(L)\oplus W=X$ for all $L\in \mathcal{L}(U,W).$
Now, to prove that this manifold is analytic we need to describe the overlap
maps. To explain the behaviour of one overlap map, assume that $X=U\oplus W
= U^{\prime }\oplus W^{\prime }$ and the existence of $U^{\prime \prime }\in 
\mathbb{G}(W,X) \cap \mathbb{G}(W^{\prime },X).$ Let $L \in \mathcal{L}(U,W)$
be such that 
\begin{equation*}
U^{\prime \prime }=G(L)=\Psi_{U \oplus W}^{-1}(L).
\end{equation*}
and then 
\begin{equation*}
X = U\oplus W = U^{\prime }\oplus W^{\prime }= G(L) \oplus W = G(L) \oplus
W^{\prime }.
\end{equation*}
Since $(id+L)$ is a linear isomorphism from $U$ to $U^{\prime \prime }=G(L)$
then $T:=P_{U^{\prime }\oplus W^{\prime }} \circ (id+L)$ is a linear
isomorphism from $U$ to $U^{\prime }.$ It follows that the map $%
(\Psi_{U^{\prime }\oplus W^{\prime }} \circ \Psi_{U \oplus W}^{-1}): 
\mathcal{L}(U,W) \rightarrow \mathcal{L}(U^{\prime },W^{\prime })$ given by 
\begin{align*}
(\Psi_{U^{\prime }\oplus W^{\prime }} \circ \Psi_{U \oplus W}^{-1})(L) & =
\Psi_{U^{\prime }\oplus W^{\prime }}(G(L)) = P_{W^{\prime }\oplus U^{\prime
}}|_{G(L)} \circ (P_{U^{\prime }\oplus W^{\prime }}|_{G(L)})^{-1} \\
& = \Psi_{U^{\prime }\oplus W^{\prime }}(G(L)) = P_{W^{\prime }\oplus
U^{\prime }}|_{G(L)} \circ P_{G(L)\oplus W^{\prime }}|_{U^{\prime }} \circ T
\circ T^{-1} \\
& = P_{W^{\prime }\oplus U^{\prime }}|_{G(L)} \circ P_{G(L)\oplus W^{\prime
}}|_{U^{\prime }} \circ P_{U^{\prime }\oplus W^{\prime }} \circ (id+L) \circ
T^{-1} \\
& = P_{W^{\prime }\oplus U^{\prime }} \circ (id+L) \circ (P_{U^{\prime
}\oplus W^{\prime }} \circ (id+L))^{-1}.
\end{align*}
is analytic. Then we say that the collection $\{\Psi _{U\oplus W},\mathbb{G}%
(W,X)\}_{U\in \mathbb{G}(X)}$ is an analytic atlas, and therefore, $\mathbb{G%
}(X)$ is an analytic Banach manifold. In particular, for each $U \in \mathbb{%
G}(X)$ the set $\mathbb{G}(W,X)\overset{\Psi _{U\oplus W}}{\cong }\mathcal{L}%
(U,W)$ is a Banach manifold modelled on $\mathcal{L}(U,W).$ Observe that if $%
U$ and $U^{\prime }$ are finite-dimensional subspaces of $X$ such that $\dim
U \neq \dim U^{\prime }$ and $X=U \oplus W = U^{\prime }\oplus W^{\prime },$
then $\mathcal{L}(U,W)$ is not linearly isomorphic to $\mathcal{L}(U^{\prime
},W^{\prime }).$
\end{example}

\begin{example}
\label{fixed_rank_grassmann} Let $X$ be a Banach space. From Proposition~\ref%
{finite_dim_subspaces}, every finite-dimensional subspace belongs to $%
\mathbb{G}(X).$ It allows to introduce $\mathbb{G}_{n}(X),$ the space of all 
$n$-dimensional subspaces of $X$ $(n\geq 0).$ It can be shown (see \cite{MRA}%
) that $\mathbb{G}_{n}(X)$ is a connected component of $\mathbb{G}(X),$ and
hence it is also a Banach manifold modelled on $\mathcal{L}(U,W),$ here $%
U\in \mathbb{G}_{n}(X)$ and $X=U\oplus W.$ Moreover, 
\begin{equation*}
\mathbb{G}_{\leq r}(X):=\bigcup_{n\leq r}\mathbb{G}_{n}(X)
\end{equation*}%
is also a Banach manifold for each fixed $r<\infty .$
\end{example}

The next example introduce the Banach-Grassmannian manifold for a normed
(non-Banach) space. To the authors knowledge there are not references in the
literature about this (nontrivial) Banach manifold structure. We need the
following lemma.

\begin{lemma}
\label{normed_G} Assume that $(X,\Vert \cdot \Vert )$ is a normed space and
let $\overline{X}$ be the Banach space obtained as the completion of $X.$
Let $U \in \mathbb{G}_n(\overline{X})$ be such that $U \subset X$ and $%
\overline{X}=U\oplus W$ for some $W \in \mathbb{G}(\overline{X}).$ Then
every subspace $U^{\prime }\in \mathbb{G}(W,\overline{X})$ is a subspace of $%
X,$ that is, $U^{\prime }\subset X.$
\end{lemma}

\begin{proof}
First at all observe that $X=U\oplus (W\cap X)$ where $W\cap X$ is a linear
subspace dense in $W=W\cap \overline{X}.$ Assume that the lemma is not true.
Then there exists $U^{\prime }\in \mathbb{G}(W,\overline{X})$ such that $%
U^{\prime }\oplus W=\overline{X}$ and $U^{\prime }\cap X\neq U^{\prime }.$
Clearly $U^{\prime }\cap X\neq \{0\},$ otherwise $W\cap X=X$ a
contradiction. We have $X=(U^{\prime }\cap X)\oplus (W\cap X),$ which
implies $\overline{X}=(U^{\prime }\cap X)\oplus W,$ a contradiction and the
lemma follows.
\end{proof}

\begin{example}
\label{normed_grassmann} Assume that $(X,\Vert \cdot \Vert )$ is a normed
space and let $\overline{X}$ be the Banach space obtained as the completion
of $X.$ We define the set $\mathbb{G}_n(X)$ as follows. We say that $U\in 
\mathbb{G}_{n}(X)$ if and only if $U\in \mathbb{G}_{n}(\overline{X})$ and $%
U\subset X.$ Then $\mathbb{G}_{n}(X)$ is also a Banach manifold. To see this
observe that, by Lemma~\ref{normed_G}, for each $U \in \mathbb{G}_{n}(X)$ 
such that $\overline{X}=U\oplus W$ for some $W \in \mathbb{G}(\overline{X}),
$ we have $\mathbb{G}(W,\overline{X}) \subset \mathbb{G}_n(X).$ Then the
collection $\{\Psi _{U\oplus W},\mathbb{G}(W,\overline{X})\}_{U\in \mathbb{G}%
_{n}(X)}$ is an analytic atlas on $\mathbb{G}_{n}(X),$ and therefore, $%
\mathbb{G}_{n}(X)$ is an analytic Banach manifold modelled on $\mathcal{L}%
(U,W),$ here $U\in \mathbb{G}_{n}(X)$ and $\overline{X}=U\oplus W.$
Moreover, as in Example~\ref{fixed_rank_grassmann}, we can define a Banach
manifold $\mathbb{G}_{\leq r}(X)$ for each fixed $r<\infty .$
\end{example}

Let $\mathbb{M}$ be a Banach manifold of class $\mathcal{C}^{p},$ $p\geq 1.$
Let $m$ be a point of $\mathbb{M}$. We consider triples $(U,\varphi ,v)$
where $(U,\varphi )$ is a chart at $m$ and $v$ is an element of the vector
space in which $\varphi (U)$ lies. We say that two of such triples $%
(U,\varphi ,v)$ and $(V,\psi ,w)$ are \emph{equivalent} if the derivative of 
$\psi \varphi ^{-1}$ at $\varphi (m)$ maps $v$ on $w.$ Thanks to the chain
rule it is an equivalence relation. An equivalence class of such triples is
called a \emph{tangent vector of $\mathbb{M}$ at $m$.}

\begin{definition}
\label{Def T}The set of such tangent vectors is called the \emph{tangent space
of $\mathbb{M}$ at $m$} and it is denoted by $\mathbb{T}_{m}(\mathbb{M}).$
\end{definition}

Each chart $(U,\varphi )$ determines a bijection of $\mathbb{T}_{m}(\mathbb{M%
})$ on a Banach space, namely the equivalence class of $(U,\varphi ,v)$
corresponds to the vector $v.$ By means of such a bijection it is possible
to equip $\mathbb{T}_{m}(\mathbb{M})$ with the structure of a topological
vector space given by the chart, and it is immediate that this structure is
independent of the selected chart.

\begin{example}
If $X$ is a Banach space, then $\mathbb{T}_{x}(X)=X$ for all $x\in X.$
\end{example}

\begin{example}
Let $X$ be a Banach space and take $A\in \mathrm{GL}(X).$ Then $\mathbb{T}%
_{A}(\mathrm{GL}(X))=\mathcal{L}(X,X).$
\end{example}

\begin{example}
For $U\in \mathbb{G}(X)$ such that $X=U\oplus W$ for some $W \in \mathbb{G}%
(X)$, we have $\mathbb{T}_{U}(\mathbb{G}(X))=\mathcal{L}(U,W).$
\end{example}

\begin{example}
\label{unit_sphere} We point out that for a Hilbert space $X$ with
associated inner product $\langle \cdot ,\cdot \rangle $ and norm $\Vert
\cdot \Vert ,$ its unit sphere denoted by 
\begin{equation*}
\mathbb{S}_{X}:=\{x\in X:\Vert x\Vert =1\},
\end{equation*}%
is a Hilbert manifold of codimension one. Moreover, for each $x\in \mathbb{S}%
_{X},$ its tangent space is 
\begin{equation*}
\mathbb{T}_{x}(\mathbb{S}_{X})=\mathrm{span}\,\{x\}^{\bot }=\{x^{\prime }\in
X:\langle x,x^{\prime }\rangle =0\}.
\end{equation*}
\end{example}

\subsection{The manifold of TBF tensors of fixed TB rank}

Assume that $\{\mathbf{V}_{\alpha }\}_{\alpha \in T_{D}\setminus \{D\}}$ is
a representation of the tensor space $\mathbf{V}_{D}=\left.
_{a}\bigotimes_{\alpha \in S(D)}\mathbf{V}_{\alpha }\right. $ in the 
tree-based format  where for each $k\in \mathcal{L}(T_{D})$ the vector
space $V_{k}$ is a normed space with a norm $\Vert \cdot \Vert _{k}.$ As
usual $V_{k_{\Vert \cdot \Vert _{k}}}$ denotes the corresponding Banach
space obtained from $V_{k}$ for $k\in \mathcal{L}(T_{D}).$ From now on, to
simplify the notation, we introduce for an admissible $\mathfrak{r}\in 
\mathbb{N}^{T_{D}}$ the product vector space 
\begin{equation*}
\mathbb{R}^{\mathfrak{r}}:=%
\mathop{\mathchoice{\raise-0.22em\hbox{\huge
$\times$}} {\raise-0.05em\hbox{\Large $\times$}}{\hbox{\large
$\times$}}{\times}}_{\alpha \in T_{D}\setminus \mathcal{L}(T_{D})}\mathbb{R}%
^{r_{\alpha }\times \left( 
\mathop{\mathchoice{\raise-0.22em\hbox{\huge
$\times$}} {\raise-0.05em\hbox{\Large $\times$}}{\hbox{\large
$\times$}}{\times}}_{\beta \in S(\alpha )}r_{\beta }\right) },
\end{equation*}%
with $r_{D}=1$. It allows us to introduce its open subset $\mathbb{R}_{\ast
}^{\mathfrak{r}},$ and hence a manifold, defined as 
\begin{equation*}
\mathbb{R}_{\ast }^{\mathfrak{r}}:=\left\{ \mathfrak{C}\in \mathbb{R}^{%
\mathfrak{r}}:%
\begin{array}{l}
C^{(D)}\in \mathbb{R}_{\ast }^{%
\mathop{\mathchoice{\raise-0.22em\hbox{\huge $\times$}} {\raise-0.05em\hbox{\Large $\times$}}{\hbox{\large
$\times$}}{\times}}_{\alpha \in S(D)}r_{\alpha }}\text{and $C^{(\mu )}\in 
\mathbb{R}_{\ast }^{r_{\mu }\times \left( 
\mathop{\mathchoice{\raise-0.22em\hbox{\huge $\times$}} {\raise-0.05em\hbox{\Large $\times$}}{\hbox{\large
$\times$}}{\times}}_{\beta \in S(\mu )}r_{\beta }\right) }$} \\ 
\text{for each $\mu \in T_{D}\setminus \{D\}$ such that $S(\mu )\neq
\emptyset .$}%
\end{array}%
\right\} .
\end{equation*}%
\bigskip

From Theorem~\ref{characterization_FT} we know that each $\mathbf{v}\in 
\mathcal{FT}_{\mathfrak{r}}(\mathbf{V}_{D})$ is totally characterised by $%
\mathfrak{C}=\mathfrak{C}(\mathbf{v})\in \mathbb{R}_{\ast }^{\mathfrak{r}}$
and a basis $\{u_{i_{k}}^{(k)}:1\leq i_{k}\leq r_{k}\}$ of $U_{k}^{\min }(%
\mathbf{v})$ for $k\in \mathcal{L}(T_{D}).$ Recall that in Example~\ref%
{normed_grassmann}, the finite-dimensional subspace $U_{k}^{\min }(\mathbf{v}%
)\subset V_{k}\subset V_{k_{\Vert \cdot \Vert _{k}}}$ belongs to the Banach
manifold $\mathbb{G}_{r_{k}}(V_{k})$ for $k\in \mathcal{L}(T_{D})$ (see also
Example \ref{fixed_rank_grassmann}) and for each $k\in \mathcal{L}(T_{D}),$
there exists a bijection (local chart) 
\begin{equation*}
\Psi _{U_{k}^{\min }(\mathbf{v})\oplus W_{k}^{\min }(\mathbf{v})}:\mathbb{G}%
(W_{k}^{\min }(\mathbf{v}),\mathbf{V}_{k_{\Vert \cdot \Vert
_{k}}})\rightarrow \mathcal{L}(U_{k}^{\min }(\mathbf{v}),W_{k}^{\min }(%
\mathbf{v}))
\end{equation*}%
given by 
\begin{equation*}
\Psi _{U_{k}^{\min }(\mathbf{v})\oplus W_{k}^{\min }(\mathbf{v}%
)}(U_{k})=L_{k}:=P_{W_{k}^{\min }(\mathbf{v})\oplus U_{k}^{\min }(\mathbf{v}%
)}|_{U_{k}}\circ (P_{U_{k}^{\min }(\mathbf{v})\oplus W_{k}^{\min }(\mathbf{v}%
)}|_{U_{k}})^{-1}.
\end{equation*}%
Moreover, $U_{k}=G(L_{k})=\mathrm{span}\{u_{k}+L_{k}(u_{k}):u_{k}\in
U_{k}^{\min }(\mathbf{v})\}.$ Clearly, the map 
\begin{equation*}
\boldsymbol{\Psi }_{\mathbf{v}}:%
\mathop{\mathchoice{\raise-0.22em\hbox{\huge
$\times$}} {\raise-0.05em\hbox{\Large $\times$}}{\hbox{\large
$\times$}}{\times}}_{k\in \mathcal{L}(T_{D})}\mathbb{G}(W_{k}^{\min }(%
\mathbf{v}),\mathbf{V}_{k_{\Vert \cdot \Vert _{k}}})\rightarrow %
\mathop{\mathchoice{\raise-0.22em\hbox{\huge $\times$}}
{\raise-0.05em\hbox{\Large $\times$}}{\hbox{\large $\times$}}{\times}}_{k\in 
\mathcal{L}(T_{D})}\mathcal{L}(U_{k}^{\min }(\mathbf{v}),W_{k}^{\min }(%
\mathbf{v})),
\end{equation*}%
defined as $\boldsymbol{\Psi }_{\mathbf{v}}:=%
\mathop{\mathchoice{\raise-0.22em\hbox{\huge $\times$}} {\raise-0.05em\hbox{\Large
$\times$}}{\hbox{\large $\times$}}{\times}}_{k\in \mathcal{L}(T_{D})}\Psi
_{U_{k}^{\min }(\mathbf{v})\oplus W_{k}^{\min }(\mathbf{v})}$ is also
bijective. Furthermore, it is a local chart for an element $\mathfrak{U}(%
\mathbf{v})=\{U_{k}^{\min }(\mathbf{v})\}_{k\in \mathcal{L}(T_{D})}$ in the
product manifold such that $\boldsymbol{\Psi }_{\mathbf{v}}(\mathfrak{U}(%
\mathbf{v}))=\mathfrak{0}:=(0)_{k\in \mathcal{L}(T_{D})}.$ It allows us to
introduce the surjective map 
\begin{equation*}
\varrho _{\mathfrak{r}}:\mathcal{FT}_{\mathfrak{r}}(\mathbf{V}%
_{D})\rightarrow 
\mathop{\mathchoice{\raise-0.22em\hbox{\huge
$\times$}} {\raise-0.05em\hbox{\Large $\times$}}{\hbox{\large
$\times$}}{\times}}_{j\in \mathcal{L}(T_{D})}\mathbb{G}_{r_{j}}(V_{j})
\end{equation*}%
defined by $\varrho _{\mathfrak{r}}(\mathbf{v})=\mathfrak{U}(\mathbf{v}%
):=(U_{k}^{\min }(\mathbf{v}))_{k\in \mathcal{L}(T_{D})}.$ Now, for each $%
\mathbf{v}\in \mathcal{FT}_{\mathfrak{r}}(\mathbf{V}_{D})$ introduce the set 
\begin{equation*}
\mathcal{U}(\mathbf{v}):=\varrho _{\mathfrak{r}}^{-1}\left( 
\mathop{\mathchoice{\raise-0.22em\hbox{\huge
$\times$}} {\raise-0.05em\hbox{\Large $\times$}}{\hbox{\large
$\times$}}{\times}}_{j\in \mathcal{L}(T_{D})}\mathbb{G}(W_{j}^{\min }(%
\mathbf{v}),V_{j})\right) =\left\{ \mathbf{w}\in \mathcal{FT}_{\mathfrak{r}}(%
\mathbf{V}_{D}):U_{k}^{\min }(\mathbf{w})\in \mathbb{G}(W_{k}^{\min }(%
\mathbf{v}),V_{k}),\,1\leq k\leq d\right\} .
\end{equation*}%
Our next step is to construct the following natural bijection. Let 
\begin{equation*}
\chi _{\mathfrak{r}}(\mathbf{v}):\mathcal{U}(\mathbf{v})\rightarrow \left( 
\mathop{\mathchoice{\raise-0.22em\hbox{\huge
$\times$}} {\raise-0.05em\hbox{\Large $\times$}}{\hbox{\large
$\times$}}{\times}}_{j\in \mathcal{L}(T_{D})}\mathbb{G}(W_{j}^{\min }(%
\mathbf{v}),V_{j})\right) \times \mathbb{R}_{\ast }^{\mathfrak{r}},\quad 
\mathbf{w}\mapsto (\chi _{1}(\mathbf{v})(\mathbf{w}),\chi _{2}(\mathbf{v})(%
\mathbf{w}))
\end{equation*}%
defined as follows. Let $\mathbf{w}\in \mathcal{U}(\mathbf{v}).$ From
Theorem~\ref{characterization_FT} we have the following.

\begin{enumerate}
\item[(a)] There exists a basis of $U_k^{\min}(\mathbf{w}) \in \mathbb{G}%
(W_k^{\min}(\mathbf{v}),V_{k_{\|\cdot\|_k}}),$ for each $k \in \mathcal{L}%
(T_D)$ and hence a unique 
\begin{equation*}
\mathfrak{L}=(L_{k})_{k \in \mathcal{L}(T_D)} \in 
\mathop{\mathchoice{\raise-0.22em\hbox{\huge
$\times$}} {\raise-0.05em\hbox{\Large $\times$}}{\hbox{\large
$\times$}}{\times}}_{k \in \mathcal{L}(T_D)}\mathcal{L}(U_k^{\min}(\mathbf{v}%
),W_k^{\min}(\mathbf{v}))
\end{equation*}
such that $\boldsymbol{\Psi }_{\mathbf{v}}(\varrho_{\mathfrak{r}}(\mathbf{w}%
)) = \mathfrak{L},$ that is, $U_{k}^{\min}(\mathbf{w})=G(L_k)$ for all $k
\in \mathcal{L}(T_D).$ Then $\chi_1(\mathbf{v})(\mathbf{w}):= \boldsymbol{%
\Psi }_{\mathbf{v}}^{-1}(\mathfrak{L})$ and $U_k^{\min}(\mathbf{w})=G(L_k) = 
\mathrm{span}\,\{(id_{k}+L_{k})(u_{i_{k}}^{(k)}): 1 \le i_k \le r_k \}$
where $U_k^{\min}(\mathbf{v})=\mathrm{span}\,\{u_{i_{k}}^{(k)}: 1 \le i_k
\le r_k \}$ is a fixed basis for $k \in \mathcal{L}(T_D)$ and hence $%
\boldsymbol{\Psi }_{\mathbf{v}}(\varrho_{\mathfrak{r}}(\mathbf{v})) =(0)_{k
\in \mathcal{L}(T_D)}.$

\item[(b)] There exists a unique $\chi_2(\mathbf{v})(\mathbf{w}):= \mathfrak{%
C} = (C^{(\alpha)})_{\alpha \in T_D \setminus \mathcal{L}(T_D)} \in \mathbb{R%
}_*^{\mathfrak{r}}$ such that 
\begin{equation}  \label{w_equation}
\mathbf{w}=\sum_{\substack{ 1\leq i_{\alpha }\leq r_{\alpha }  \\ \alpha \in
S(D)}}C_{(i_{\alpha })_{\alpha \in S(D)}}^{(D)}\bigotimes_{\alpha \in S(D)}%
\mathbf{w}_{i_{\alpha }}^{(\alpha )},.
\end{equation}
and where for each $\beta \in T_D \setminus (\{D\} \cup \mathcal{L}(T_D))$
we have 
\begin{equation*}
U_{\beta}^{\min}(\mathbf{w})=\mathrm{span}\,\{\mathbf{w}_{i_{\beta}}:1 \le
i_{\beta} \le r_{\beta}\}
\end{equation*}
with 
\begin{equation*}
\mathbf{w}_{i_{\beta }}^{(\beta )} = \left\{ 
\begin{array}{ll}
(id_{\beta}+L_{\beta})(u_{i_{\beta}}^{(\beta)}) & \text{ if } \beta \in 
\mathcal{L}(T_D) \\ 
&  \\ 
\sum_{\substack{ 1\leq i_{\delta }\leq r_{\delta }  \\ \delta \in S(\beta )}}%
C_{i_{\beta },(i_{\delta })_{\delta \in S(\beta )}}^{(\beta
)}\bigotimes_{\delta \in S(\beta )}\mathbf{w}_{i_{\delta }}^{(\delta )} & 
\text{ otherwise. }%
\end{array}
\right.
\end{equation*}
\end{enumerate}

Finally, let 
\begin{equation*}
p_{\mathbf{v}}:\left( 
\mathop{\mathchoice{\raise-0.22em\hbox{\huge
$\times$}} {\raise-0.05em\hbox{\Large $\times$}}{\hbox{\large
$\times$}}{\times}}_{j\in \mathcal{L}(T_{D})}\mathbb{G}(W_{j}^{\min }(%
\mathbf{v}),V_{j})\right) \times \mathbb{R}_{\ast }^{\mathfrak{r}%
}\rightarrow 
\mathop{\mathchoice{\raise-0.22em\hbox{\huge
$\times$}} {\raise-0.05em\hbox{\Large $\times$}}{\hbox{\large
$\times$}}{\times}}_{j\in \mathcal{L}(T_{D})}\mathbb{G}(W_{j}^{\min }(%
\mathbf{v}),V_{j})
\end{equation*}%
be the projection $p_{\mathbf{v}}(\mathfrak{U},\mathfrak{C})=\mathfrak{U}$
then $p_{\mathbf{v}}\circ \chi _{\mathfrak{r}}(\mathbf{v})=\varrho _{%
\mathfrak{r}}.$

\bigskip

A very useful remark is the following. Recall that $(id_{k}+L_{k})$ is a
linear isomorphism from $U_{k}^{\min }(\mathbf{v})$ to $U_{k}^{\min }(%
\mathbf{w})=G(L_{k})$ for all $k\in \mathcal{L}(T_{D}).$ From Proposition
3.49 of \cite{Hackbusch} we have 
\begin{equation*}
\left. _{a}\bigotimes_{k\in \mathcal{L}(T_{D})}\mathcal{L}(U_{k}^{\min }(%
\mathbf{v}),U_{k}^{\min }(\mathbf{w}))\right. =\mathcal{L}\left( \left.
_{a}\bigotimes_{k\in \mathcal{L}(T_{D})}U_{k}^{\min }(\mathbf{v})\right.
,\left. _{a}\bigotimes_{k\in \mathcal{L}(T_{D})}U_{k}^{\min }(\mathbf{w}%
)\right. \right)
\end{equation*}%
and denote by $\mathrm{GL}\left( \left. _{a}\bigotimes_{k\in \mathcal{L}%
(T_{D})}U_{k}^{\min }(\mathbf{v})\right. ,\left. _{a}\bigotimes_{k\in 
\mathcal{L}(T_{D})}U_{k}^{\min }(\mathbf{w})\right. \right) $ the set of
linear isomorphisms of 
\begin{equation*}
\mathcal{L}\left( \left. _{a}\bigotimes_{k\in \mathcal{L}(T_{D})}U_{k}^{\min
}(\mathbf{v})\right. ,\left. _{a}\bigotimes_{k\in \mathcal{L}%
(T_{D})}U_{k}^{\min }(\mathbf{w})\right. \right) .
\end{equation*}%
Let us define 
\begin{align*}
\mathrm{GL}_{\mathbf{1}}\left( \left. _{a}\bigotimes_{k\in \mathcal{L}%
(T_{D})}U_{k}^{\min }(\mathbf{v})\right. ,\left. _{a}\bigotimes_{k\in 
\mathcal{L}(T_{D})}U_{k}^{\min }(\mathbf{w})\right. \right) := \\
\mathrm{GL}\left( \left. _{a}\bigotimes_{k\in \mathcal{L}(T_{D})}U_{k}^{\min
}(\mathbf{v})\right. ,\left. _{a}\bigotimes_{k\in \mathcal{L}%
(T_{D})}U_{k}^{\min }(\mathbf{w})\right. \right) \cap \mathcal{M}_{\leq 
\mathbf{1}}\left( \left. _{a}\bigotimes_{k\in \mathcal{L}(T_{D})}\mathcal{L}%
(U_{k}^{\min }(\mathbf{v}),U_{k}^{\min }(\mathbf{w}))\right. \right) .
\end{align*}%
Then 
\begin{equation*}
\bigotimes_{k\in \mathcal{L}(T_{D})}(id_{k}+L_{k})\in \mathrm{GL}_{\mathbf{1}%
}\left( \left. _{a}\bigotimes_{k\in \mathcal{L}(T_{D})}U_{k}^{\min }(\mathbf{%
v})\right. ,\left. _{a}\bigotimes_{k\in \mathcal{L}(T_{D})}U_{k}^{\min }(%
\mathbf{w})\right. \right) .
\end{equation*}%
Observe that for each given $\mathbf{v}\in \mathcal{FT}_{\mathfrak{r}}(%
\mathbf{V}_{D})$ the map 
\begin{equation*}
\Theta _{\mathbf{v}}:\left( 
\mathop{\mathchoice{\raise-0.22em\hbox{\huge
$\times$}} {\raise-0.05em\hbox{\Large $\times$}}{\hbox{\large
$\times$}}{\times}}_{j\in \mathcal{L}(T_{D})}\mathbb{G}(W_{j}^{\min }(%
\mathbf{v}),V_{j})\right) \times \mathbb{R}_{\ast }^{\mathfrak{r}%
}\rightarrow \left( 
\mathop{\mathchoice{\raise-0.22em\hbox{\huge
$\times$}} {\raise-0.05em\hbox{\Large $\times$}}{\hbox{\large
$\times$}}{\times}}_{j\in \mathcal{L}(T_{D})}\mathcal{L}(U_{j}^{\min }(%
\mathbf{v}),W_{j}^{\min }(\mathbf{v}))\right) \times \mathbb{R}_{\ast }^{%
\mathfrak{r}}
\end{equation*}%
where $\Theta _{\mathbf{v}}:=\boldsymbol{\Psi }_{\mathbf{v}}\times id$ is a
bijection. Then 
\begin{equation*}
\Theta _{\mathbf{v}}\circ \chi _{\mathfrak{r}}(\mathbf{v}):\mathcal{U}(%
\mathbf{v})\rightarrow \left( 
\mathop{\mathchoice{\raise-0.22em\hbox{\huge
$\times$}} {\raise-0.05em\hbox{\Large $\times$}}{\hbox{\large
$\times$}}{\times}}_{j\in \mathcal{L}(T_{D})}\mathcal{L}(U_{j}^{\min }(%
\mathbf{v}),W_{j}^{\min }(\mathbf{v}))\right) \times \mathbb{R}_{\ast }^{%
\mathfrak{r}}
\end{equation*}%
is also a bijection where 
\begin{equation*}
(\Theta _{\mathbf{v}}\circ \chi _{\mathfrak{r}}(\mathbf{v}))^{-1}(\mathfrak{L%
},\mathfrak{C})=\mathbf{w}=\left( \bigotimes_{k\in \mathcal{L}%
(T_{D})}(id_{k}+L_{k})\right) (\mathbf{u})=\left( \bigotimes_{k\in \mathcal{L%
}(T_{D})}(id_{k}+L_{k})\right) (\Theta _{\mathbf{v}}\circ \chi _{\mathfrak{r}%
}(\mathbf{v}))^{-1}(\mathfrak{0},\mathfrak{C}).
\end{equation*}%
We can interpret this last equality as follows. $\mathbf{w}\in \mathcal{U}(%
\mathbf{v})$ holds if and only if 
\begin{equation*}
\mathbf{w}\in \mathcal{FT}_{\mathfrak{r}}\left( \left( \bigotimes_{k\in 
\mathcal{L}(T_{D})}(id_{k}+L_{k})\right) \left( \left. _{a}\bigotimes_{k\in 
\mathcal{L}(T_{D})}U_{k}^{\min }(\mathbf{v})\right. \right) \right)
\end{equation*}%
for some $\mathfrak{L}\in 
\mathop{\mathchoice{\raise-0.22em\hbox{\huge
$\times$}} {\raise-0.05em\hbox{\Large $\times$}}{\hbox{\large
$\times$}}{\times}}_{k\in \mathcal{L}(T_{D})}\mathcal{L}(U_{k}^{\min }(%
\mathbf{v}),W_{k}^{\min }(\mathbf{v})).$ In consequence, each neighbourhood
of $\mathbf{v}$ in $\mathcal{FT}_{\mathfrak{r}}(\mathbf{V}_{D})$ can be
written as 
\begin{equation*}
\mathcal{U}(\mathbf{v})=\bigcup_{\mathfrak{L}\in 
\mathop{\mathchoice{\raise-0.22em\hbox{\huge
$\times$}} {\raise-0.05em\hbox{\Large $\times$}}{\hbox{\large
$\times$}}{\times}}_{k\in \mathcal{L}(T_{D})}\mathcal{L}(U_{k}^{\min }(%
\mathbf{v}),W_{k}^{\min }(\mathbf{v}))}\mathcal{FT}_{\mathfrak{r}}\left(
\left( \bigotimes_{k\in \mathcal{L}(T_{D})}(id_{k}+L_{k})\right) \left(
\left. _{a}\bigotimes_{k\in \mathcal{L}(T_{D})}U_{k}^{\min }(\mathbf{v}%
)\right. \right) \right) ,
\end{equation*}%
that is, a union of copies of $\mathcal{FT}_{\mathfrak{r}}\left( \left._a
\bigotimes_{k\in \mathcal{L}(T_{D})}\mathbb{R}^{r_{k}}\right.\right) $
indexed by a Banach manifold. Before stating the next result, we introduce
the following definition.

\begin{definition}
Let $X$ and $Y$ be two Banach manifolds. Let $F:X \rightarrow Y$ be a map.
We shall say that $F$ is a $\mathcal{C}^r$ (respectively, analytic) \emph{%
morphism} if given $x \in X$ there exists a chart $(U,\varphi)$ at $x$ and a
chart $(W,\psi)$ at $F(x)$ such that $F(U) \subset W,$ and the map 
\begin{equation*}
\psi \circ F \circ \varphi^{-1}:\varphi(U) \rightarrow \psi(W)
\end{equation*}
is a $\mathcal{C}^r$-Fr\'echet differentiable (respectively, analytic) map.
\end{definition}

\begin{lemma}
\label{premanifold} Let $\mathbf{v},\mathbf{v}^{\prime }\in \mathcal{FT}_{%
\mathfrak{r}}(\mathbf{V}_{D})$ be such that $\mathcal{U}(\mathbf{v}) \cap 
\mathcal{U}(\mathbf{v}^{\prime }) \neq \emptyset.$ Then the bijective map 
\begin{equation*}
\chi_{\mathfrak{r}}(\mathbf{v}^{\prime }) \circ \chi_{\mathfrak{r}}(\mathbf{v%
})^{-1}: \left(%
\mathop{\mathchoice{\raise-0.22em\hbox{\huge
$\times$}} {\raise-0.05em\hbox{\Large $\times$}}{\hbox{\large
$\times$}}{\times}}_{j \in \mathcal{L}(T_D)}\mathbb{G}(W_j^{\min}(\mathbf{v}%
),V_{j})\right) \times \mathbb{R}_{*}^{\mathfrak{r}} \rightarrow \left(%
\mathop{\mathchoice{\raise-0.22em\hbox{\huge
$\times$}} {\raise-0.05em\hbox{\Large $\times$}}{\hbox{\large
$\times$}}{\times}}_{j \in \mathcal{L}(T_D)}\mathbb{G}(W_j^{\min}(\mathbf{%
v^{\prime }}),V_{j})\right) \times \mathbb{R}_{*}^{\mathfrak{r}}
\end{equation*}
is an analytic morphism. Furthermore, it is an analytic diffeomorphism.
\end{lemma}

\begin{proof}
Let $\mathbf{v},\mathbf{v}^{\prime }\in \mathcal{FT}_{\mathfrak{r}}(\mathbf{V%
}_{D})$ be given. To prove the lemma we need to check that the map 
\begin{equation*}
\Theta _{\mathbf{v}^{\prime }}\circ \chi _{\mathfrak{r}}(\mathbf{v}^{\prime
})\circ \chi _{\mathfrak{r}}(\mathbf{v})^{-1}\circ \Theta _{\mathbf{v}}^{-1}:%
\mathop{\mathchoice{\raise-0.22em\hbox{\huge $\times$}}
{\raise-0.05em\hbox{\Large $\times$}}{\hbox{\large $\times$}}{\times}}_{k\in 
\mathcal{L}(T_{D})}\mathcal{L}(U_{k}^{\min }(\mathbf{v}),W_{k}^{\min }(%
\mathbf{v}))\times \mathbb{R}_{\ast }^{\mathfrak{r}}\rightarrow %
\mathop{\mathchoice{\raise-0.22em\hbox{\huge $\times$}}
{\raise-0.05em\hbox{\Large $\times$}}{\hbox{\large $\times$}}{\times}}_{k\in 
\mathcal{L}(T_{D})}\mathcal{L}(U_{k}^{\min }(\mathbf{v}^{\prime
}),W_{k}^{\min }(\mathbf{v}^{\prime }))\times \mathbb{R}_{\ast }^{\mathfrak{r%
}}
\end{equation*}%
is analytic whenever $\mathcal{U}(\mathbf{v})\cap \mathcal{U}(\mathbf{v}%
^{\prime })\neq \emptyset .$ Let $\mathbf{w}\in \mathcal{U}(\mathbf{v})\cap 
\mathcal{U}(\mathbf{v}^{\prime })$ be such that $(\chi _{\mathfrak{r}}(%
\mathbf{v})\circ \Theta _{\mathbf{v}})(\mathbf{w})=(\mathfrak{L},\mathfrak{C}%
)$ and $(\chi _{\mathfrak{r}}(\mathbf{v}^{\prime })\circ \Theta _{\mathbf{v}^{\prime }})(\mathbf{w})=(\mathfrak{L}^{\prime },\mathfrak{D}),$ that is,
$$
(\Theta _{\mathbf{v}^{\prime }}\circ \chi _{\mathfrak{r}}(\mathbf{v}^{\prime
})\circ \chi _{\mathfrak{r}}(\mathbf{v})^{-1}\circ \Theta _{\mathbf{v}}^{-1})(\mathfrak{L},\mathfrak{C}) = (\mathfrak{L}^{\prime },\mathfrak{D}).
$$
Since $%
\mathbf{w}\in \mathcal{U}(\mathbf{v})\cap \mathcal{U}(\mathbf{v}^{\prime })$
then 
\begin{equation*}
\varrho _{\mathfrak{r}}(\mathbf{w})=(U_{k}^{\min }(\mathbf{w}))_{k\in 
\mathcal{L}(T_{D})}\in \left( 
\mathop{\mathchoice{\raise-0.22em\hbox{\huge
$\times$}} {\raise-0.05em\hbox{\Large $\times$}}{\hbox{\large
$\times$}}{\times}}_{j\in \mathcal{L}(T_{D})}\mathbb{G}(W_{j}^{\min }(%
\mathbf{v}),V_{j})\right) \cap \left( 
\mathop{\mathchoice{\raise-0.22em\hbox{\huge
$\times$}} {\raise-0.05em\hbox{\Large $\times$}}{\hbox{\large
$\times$}}{\times}}_{j\in \mathcal{L}(T_{D})}\mathbb{G}(W_{j}^{\min }(%
\mathbf{v}^{\prime }),V_{j})\right)
\end{equation*}
and 
\begin{equation*}
(\Psi _{\mathbf{v}^{\prime }}\circ \Psi _{\mathbf{v}}^{-1})(\Psi _{\mathbf{v}%
}((U_{k}^{\min }(\mathbf{w}))_{k\in \mathcal{L}(T_{D})}))=\Psi _{\mathbf{v}%
^{\prime }}(U_{k}^{\min }(\mathbf{w}))_{k\in \mathcal{L}(T_{D})}),
\end{equation*}%
that is, 
\begin{equation*}
(\Psi _{\mathbf{v}^{\prime }}\circ \Psi _{\mathbf{v}}^{-1})(\mathfrak{L})=%
\mathfrak{L}^{\prime }.
\end{equation*}%
Hence 
\begin{equation*}
\left( \Theta _{\mathbf{v}^{\prime }}\circ \chi _{\mathfrak{r}}(\mathbf{v}%
^{\prime })\right) (\mathbf{w})=((\Psi _{\mathbf{v}^{\prime }}\circ \Psi _{%
\mathbf{v}}^{-1})(\mathfrak{L}),\mathfrak{D}),
\end{equation*}%
where $\Psi _{\mathbf{v}^{\prime }}\circ \Psi _{\mathbf{v}}^{-1}$ is an
analytic map. Let $\mathbf{u}=(\chi _{\mathfrak{r}}(\mathbf{v})^{-1}\circ
\Theta _{\mathbf{v}}^{-1})(\mathfrak{0},\mathfrak{C})$ and $\mathbf{u}%
^{\prime }=(\chi _{\mathfrak{r}}(\mathbf{v}^{\prime -1}\circ \Theta _{%
\mathbf{v}^{\prime }}^{-1})(\mathfrak{0},\mathfrak{D}).$ Then $\mathbf{u}\in
\left. _{a}\bigotimes_{k\in \mathcal{L}(T_{D})}U_{k}^{\min }(\mathbf{v}%
)\right. ,$ $\mathbf{u}^{\prime }\in \left. _{a}\bigotimes_{k\in \mathcal{L}%
(T_{D})}U_{k}^{\min }(\mathbf{v}^{\prime })\right. $ and 
\begin{align*}
\mathbf{w} & = \left( \Theta _{\mathbf{v}}\circ \chi _{\mathfrak{r}}(\mathbf{v}%
)\right)^{-1}(\mathfrak{L},\mathfrak{C}) =\left( \bigotimes_{k\in \mathcal{L}(T_{D})}(id_{k}+L_{k})\right) (%
\mathbf{u}) = \left( \bigotimes_{k\in \mathcal{L}(T_{D})}(id_{k}+L_{k})\right) \circ 
\left( \Theta _{\mathbf{v}}\circ \chi _{\mathfrak{r}}(\mathbf{v}%
)\right)^{-1}(\mathfrak{0},\mathfrak{C}) \\
& = \left( \Theta _{\mathbf{v}^{\prime }}\circ \chi _{\mathfrak{r}}(\mathbf{v}%
^{\prime })\right)^{-1}(\mathfrak{L}',\mathfrak{D}) = \left( \bigotimes_{k\in \mathcal{L}(T_{D})}(id_{k}+L_{k}^{\prime
})\right) (\mathbf{u}^{\prime }) = \left( \bigotimes_{k\in \mathcal{L}(T_{D})}(id_{k}+L_{k}^{\prime
})\right) \circ \left( \Theta _{\mathbf{v}^{\prime }}\circ \chi _{\mathfrak{r}}(\mathbf{v}%
^{\prime })\right)^{-1}(\mathfrak{0},\mathfrak{D}).
\end{align*}%
Hence, 
\begin{equation*}
(\mathfrak{0},\mathfrak{D}) = \left( \Theta _{\mathbf{v}^{\prime }}\circ \chi _{\mathfrak{r}}(\mathbf{v}%
^{\prime })\right)\circ
\left( \bigotimes_{k\in \mathcal{L}(T_{D})}(id_{k}+L_{k}')^{-1}\circ (id_{k}+L_{k})\right) \circ 
\left( \Theta _{\mathbf{v}}\circ \chi _{\mathfrak{r}}(\mathbf{v}%
)\right)^{-1}(\mathfrak{0},\mathfrak{C}).
\end{equation*}%
In consequence, we can write 
\begin{equation*}
(\mathfrak{0},\mathfrak{D})=f(\mathfrak{L},\mathfrak{C}):=\left( \Theta _{\mathbf{v}^{\prime }}\circ \chi _{\mathfrak{r}}(\mathbf{v}%
^{\prime })\right)\circ
\left( \bigotimes_{k\in \mathcal{L}(T_{D})}(id_{k}+L_{k}')^{-1}\circ (id_{k}+L_{k})\right) \circ 
\left( \Theta _{\mathbf{v}}\circ \chi _{\mathfrak{r}}(\mathbf{v}%
)\right)^{-1}(\mathfrak{0},\mathfrak{C}).
\end{equation*}%
and the map 
\begin{equation*}
f:\mathop{\mathchoice{\raise-0.22em\hbox{\huge $\times$}}
{\raise-0.05em\hbox{\Large $\times$}}{\hbox{\large $\times$}}{\times}}_{k\in 
\mathcal{L}(T_{D})}\mathcal{L}(U_{k}^{\min }(\mathbf{v}),W_{k}^{\min }(%
\mathbf{v}))\times \mathbb{R}_{\ast }^{\mathfrak{r}}\rightarrow \{\mathfrak{0}\}\times \mathbb{R}%
_{\ast }^{\mathfrak{r}}
\end{equation*}
is an analytic morphism.  Thus the lemma is proved.
\end{proof}

\bigskip

The next result will help us to show that the collection $\{\Theta _{\mathbf{%
v}}\circ \chi _{\mathfrak{r}}(\mathbf{v}),\mathcal{U}(\mathbf{v})\}_{\mathbf{%
v}\in \mathcal{FT}_{\mathfrak{r}}(\mathbf{V}_{D})}$ is an atlas for $%
\mathcal{FT}_{\mathfrak{r}}(\mathbf{V}_{D}).$ Indeed, it is the unique
manifold structure for which $\varrho _{\mathfrak{r}}:\mathcal{FT}_{%
\mathfrak{r}}(\mathbf{V}_{D})\rightarrow 
\mathop{\mathchoice{\raise-0.22em\hbox{\huge
$\times$}} {\raise-0.05em\hbox{\Large $\times$}}{\hbox{\large
$\times$}}{\times}}_{j\in \mathcal{L}(T_{D})}\mathbb{G}_{r_{j}}(V_{j})$
defines a locally trivial fibre bundle with typical fibre $\mathbb{R}_{\ast
}^{\mathfrak{r}}.$ To this end we will use Lemma~\ref{premanifold} and the
following classical result (see Proposition~3.4.28 in \cite{MRA}).

\begin{theorem}
\label{bundle_manifold} Let $E$ be a set, $B$ and $F$ be $\mathcal{C}^k$
manifolds, and let $\pi : E \rightarrow B$ be a surjective map. Assume that

\begin{enumerate}
\item[(a)] there is a $\mathcal{C}^k$ atlas $\{(U_{\alpha},
\varphi_{\alpha}):\alpha \in I\}$ of $B$ and a family of bijective maps $%
\chi_{\alpha}: \pi^{-1}(U_{\alpha}) \rightarrow U_{\alpha} \times F$
satisfying $p_{\alpha} \circ \chi_{\alpha} = \pi,$ where $p_{\alpha}:
U_{\alpha} \times F \rightarrow U_{\alpha}$ is the projection, and that

\item[(b)] the maps $\chi_{\alpha^{\prime }}\circ
\chi_{\alpha}^{-1}:U_{\alpha} \times F \rightarrow U_{\alpha^{\prime }}
\times F$ are $\mathcal{C}^k$ diffeomorphisms whenever $U_{\alpha} \cap
U_{\alpha^{\prime }} \neq \emptyset.$
\end{enumerate}

Then there is a $\mathcal{C}^{k}$ atlas $\{(V_{\beta },\psi _{\alpha
}):\beta \in J\}$ of $F$ and a unique $\mathcal{C}^{k}$ manifold structure
on $E$ given by 
\begin{equation*}
\{(\chi _{\alpha }^{-1}(U_{\alpha }\times V_{\beta }),(\varphi _{\alpha
}\times \psi _{\beta }))\circ \chi _{\alpha }:\alpha \in I,\,\beta \in J\}
\end{equation*}%
for which $\pi :E\rightarrow B$ is a $\mathcal{C}^{k}$ locally trivial fibre
bundle with typical fibre $F.$
\end{theorem}

Let us mention the following two mathematical objects related to the above
theorem. Let $B$ and $F$ be $\mathcal{C}^{k}$ manifolds, and let $\pi
:E\rightarrow B$ be a surjective map satisfying the conditions (a)-(b) of
Theorem~\ref{bundle_manifold}. Then $(E,B,\pi )$ is called a \textit{fibre
bundle} with typical fibre $F,$ and if $F$ is also a Banach space, then it
is called a \textit{vector bundle} (see Chapters 6 and 7 in \cite{B}). In consequence,
we can state the following result.

\begin{theorem}
\label{Tucker_Banach_Manifold} Assume that $\{\mathbf{V}_{\alpha }\}_{\alpha
\in T_{D}\setminus \{D\}}$ is a representation of the tensor space $\mathbf{V%
}_{D}=\left. _{a}\bigotimes_{\alpha \in S(D)}\mathbf{V}_{\alpha }\right. $
in the tree-based format  where for each $k\in \mathcal{L}(T_{D})$ the
vector space $V_{k}$ is a normed space with a norm $\Vert \cdot \Vert _{k}.$
Then the collection $\{\Theta _{\mathbf{v}}\circ \chi _{\mathfrak{r}}(%
\mathbf{v}),\mathcal{U}(\mathbf{v})\}_{\mathbf{v}\in \mathcal{FT}_{\mathfrak{%
r}}(\mathbf{V}_{D})}$ is an analytic atlas for $\mathcal{FT}_{\mathfrak{r}}(%
\mathbf{V}_{D}).$ Furthermore, the set $\mathcal{FT}_{\mathfrak{r}}(\mathbf{V%
}_{D})$ of TBF tensors with fixed TB rank is an analytic Banach manifold and 
\begin{equation*}
\left( \mathcal{FT}_{\mathfrak{r}}(\mathbf{V}_{D}),%
\mathop{\mathchoice{\raise-0.22em\hbox{\huge
$\times$}} {\raise-0.05em\hbox{\Large $\times$}}{\hbox{\large
$\times$}}{\times}}_{j\in \mathcal{L}(T_{D})}\mathbb{G}_{r_{j}}(V_{j}),%
\varrho _{\mathfrak{r}}\right)
\end{equation*}%
is a fibre bundle with typical fibre $\mathbb{R}_{\ast }^{\mathfrak{r}}.$
\end{theorem}

\begin{proof}
Take the set $E = \mathcal{FT}_{\mathfrak{r}}(\mathbf{V}_{D})$ and the
analytic Banach manifolds $B=%
\mathop{\mathchoice{\raise-0.22em\hbox{\huge
$\times$}} {\raise-0.05em\hbox{\Large $\times$}}{\hbox{\large
$\times$}}{\times}}_{j \in \mathcal{L}(T_D)}\mathbb{G}_{r_j}(V_{j})$ and $F
= \mathbb{R}_*^{\mathfrak{r}}.$ Let us consider the surjective map $\varrho_{%
\mathfrak{r}}:\mathcal{FT}_{\mathfrak{r}}(\mathbf{V}_{D}) \rightarrow 
\mathop{\mathchoice{\raise-0.22em\hbox{\huge
$\times$}} {\raise-0.05em\hbox{\Large $\times$}}{\hbox{\large
$\times$}}{\times}}_{j \in \mathcal{L}(T_D)}\mathbb{G}_{r_j}(V_{j}).$ The
theorem follows from Theorem~\ref{bundle_manifold} because Theorem~\ref%
{bundle_manifold}(a) is true by the definition of $\chi_{\mathfrak{r}}(%
\mathbf{v})$ and Theorem~\ref{bundle_manifold}(b) is a consequence of Lemma~%
\ref{premanifold}.
\end{proof}

\begin{remark}
We observe that the geometric structure of manifold is independent of the
choice of the norm $\|\cdot\|_D$ over the tensor space $\mathbf{V}_D.$
\end{remark}


\begin{corollary}
\label{hilbert_manifold_tensor} Assume that $V_{k_{\Vert \cdot \Vert _{k}}}$
is a Hilbert space with norm $\Vert \cdot \Vert _{k}$ for $k\in \mathcal{L}%
(T_{D}).$ Then $\mathcal{FT}_{\mathfrak{r}}(\mathbf{V}_{D})$ is an analytic
Hilbert manifold.
\end{corollary}

\begin{proof}
We can identify each $L_{k }\in \mathcal{L}\left( U_{k}^{\min }(\mathbf{v}%
),W_{k }^{\min }(\mathbf{v})\right) $ with a $(\mathbf{w}_{s_{k }}^{(k
)})_{s_{k }=1}^{s_{k }=r_{k }}\in W_{k }^{\min }(\mathbf{v})^{r_{k }},$
where $\mathbf{w}_{s_{k }}^{(k )}=L_{k }(\mathbf{u}_{(s_{k })}^{k })$ and $%
U_{k }^{\min }(\mathbf{v})=\mathrm{span}\,\{\mathbf{u}_{(1)}^{k },\ldots ,%
\mathbf{u}_{(r_{k })}^{k }\}$ for $k \in \mathcal{L}(T_D).$ Thus we can
identify each $\left(\mathfrak{L},\mathfrak{C}\right) \in \mathcal{U}(%
\mathbf{v})$ with a pair%
\begin{equation*}
\left(\mathfrak{W},\mathfrak{C}\right) \in 
\mathop{\mathchoice{\raise-0.22em\hbox{\huge $\times$}} {\raise-0.05em\hbox{\Large $\times$}}{\hbox{\large
$\times$}}{\times}}_{k \in \mathcal{L}(T_D)}W_{k }^{\min }(\mathbf{v})^{r_{k
}} \times \mathbb{R}_{\ast }^{\mathfrak{r}},
\end{equation*}%
where $\mathfrak{W}:=((\mathbf{w}_{s_{k }}^{(k )})_{s_{k }=1}^{s_{k }=r_{k
}})_{k \in \mathcal{L}(T_D)}.$ Take $%
\mathop{\mathchoice{\raise-0.22em\hbox{\huge $\times$}}
{\raise-0.05em\hbox{\Large $\times$}}{\hbox{\large
$\times$}}{\times}}_{k \in \mathcal{L}(T_D)}W_{k }^{\min }(\mathbf{v})^{r_{k
}} \times \mathbb{R}_{\ast }^{\mathfrak{r}}$ an open subset of the
Hilbert space $%
\mathop{\mathchoice{\raise-0.22em\hbox{\huge $\times$}}
{\raise-0.05em\hbox{\Large $\times$}}{\hbox{\large
$\times$}}{\times}}_{k \in \mathcal{L}(T_D)}W_{\alpha }^{\min }(\mathbf{v}%
)^{r_{k }} \times \mathbb{R}^{\mathfrak{r}}$ endowed with the product norm 
\begin{equation*}
\Vert \left(\mathfrak{W},\mathfrak{C}\right) \Vert _{\times }:=\sum_{\alpha
\in T_{D}\setminus \mathcal{L}(T_{D})}\Vert C^{\alpha }\Vert _{F}+\sum_{k
\in \mathcal{L}(T_D)}\sum_{s_{k }=1}^{r_{k }}\Vert \mathbf{w}_{s_{k }}^{(k
)}\Vert _{k},
\end{equation*}%
with $\Vert \cdot \Vert _{F}$ the Frobenius norm. 
It allows us to define local charts, also denoted by $\Theta _{\mathbf{v}}
\circ \chi_{\mathfrak{r}}(\mathbf{v}),$ by 
\begin{equation*}
\chi_{\mathfrak{r}}^{-1}(\mathbf{v}) \circ \Theta _{\mathbf{v}}^{-1}: 
\mathop{\mathchoice{\raise-0.22em\hbox{\huge $\times$}} {\raise-0.05em\hbox{\Large $\times$}}{\hbox{\large
$\times$}}{\times}}_{k \in \mathcal{L}(T_D)}W_{\alpha }^{\min }(\mathbf{v}%
)^{r_{k }} \times \mathbb{R}_{\ast}^{\mathfrak{r}} \longrightarrow \mathcal{U%
}(\mathbf{v}),
\end{equation*}%
where $(\chi_{\mathfrak{r}}^{-1}(\mathbf{v}) \circ \Theta _{\mathbf{v}%
}^{-1})\left( \mathfrak{W},\mathfrak{C}\right) =\mathbf{w}$ putting $L_{k }(%
\mathbf{u}_{i_{k }}^{(k )})=\mathbf{w}_{i_{k }}^{(k )},$ $1\leq i_{k }\leq
r_{k } $ and $k \in \mathcal{L}(T_D).$ Since each local chart is defined
over an open subset of the Hilbert space $%
\mathop{\mathchoice{\raise-0.22em\hbox{\huge $\times$}}
{\raise-0.05em\hbox{\Large $\times$}}{\hbox{\large
$\times$}}{\times}}_{k \in \mathcal{L}(T_D)}W_{k }^{\min }(\mathbf{v})^{r_{k
}} \times \mathbb{R}^{\mathfrak{r}},$ the corollary follows.
\end{proof}

\bigskip

Using the geometric structure of local charts for the manifold $\mathcal{FT}%
_{\mathfrak{r}}(\mathbf{V}_{D}),$ we can identify its tangent space at $%
\mathbf{v}$ with $\mathbb{T}_{\mathbf{v}}(\mathcal{FT}_{\mathfrak{r}}(%
\mathbf{V}_{D})):=%
\mathop{\mathchoice{\raise-0.22em\hbox{\huge $\times$}}
{\raise-0.05em\hbox{\Large $\times$}}{\hbox{\large $\times$}}{\times}}_{k\in 
\mathcal{L}(T_{D})}\mathcal{L}(U_{k}^{\min }(\mathbf{v}),W_{k}^{\min }(%
\mathbf{v}))\times \mathbb{R}^{\mathfrak{r}}.$ We will consider $\mathbb{T}_{%
\mathbf{v}}(\mathcal{FT}_{\mathfrak{r}}(\mathbf{V}_{D}))$ endowed with the
product norm 
\begin{equation*}
\Vert \Vert (\mathfrak{L},\mathfrak{C})\Vert \Vert :=\sum_{\alpha \in
T_{D}\setminus \mathcal{L}(T_{D})}\Vert C^{(\alpha )}\Vert _{F}+\sum_{k\in 
\mathcal{L}(T_{D})}\Vert L_{k}\Vert _{W_{k}^{\min }(\mathbf{v})\leftarrow
U_{k}^{\min }(\mathbf{v})}.
\end{equation*}%
Moreover, the map $\varrho
_{\mathfrak{r}}$ is an analytic morphism and 
\begin{equation*}
\mathrm{T}_{\mathbf{v}}\varrho _{\mathfrak{r}}:%
\mathop{\mathchoice{\raise-0.22em\hbox{\huge $\times$}} {\raise-0.05em\hbox{\Large
$\times$}}{\hbox{\large $\times$}}{\times}}_{k\in \mathcal{L}(T_{D})}%
\mathcal{L}(U_{k}^{\min }(\mathbf{v}),W_{k}^{\min }(\mathbf{v}))\times 
\mathbb{R}^{\mathfrak{r}}\rightarrow 
\mathop{\mathchoice{\raise-0.22em\hbox{\huge $\times$}} {\raise-0.05em\hbox{\Large $\times$}}{\hbox{\large
$\times$}}{\times}}_{k\in \mathcal{L}(T_{D})}\mathcal{L}(U_{k}^{\min }(%
\mathbf{v}),W_{k}^{\min }(\mathbf{v})),\quad (\mathfrak{L},\mathfrak{C}%
)\mapsto \mathfrak{L}.
\end{equation*}

\bigskip

Finally, the same argument used to provide a Banach manifold structure to
the set $\mathbb{G}_{\leq n}(X)$ used with $\mathcal{FT}_{\leq \mathfrak{r}}(%
\mathbf{V}_{D})$ and \eqref{connected_id}, allows us to state the following.

\begin{theorem}
\label{Bounded_Banach_Manifold} Assume that $\{\mathbf{V}_{\alpha
}\}_{\alpha \in T_{D}\setminus \{D\}}$ is a representation of the tensor
space $\mathbf{V}_{D}=\left. _{a}\bigotimes_{\alpha \in S(D)}\mathbf{V}%
_{\alpha }\right. $ in the tree-based format  where for each $k\in 
\mathcal{L}(T_{D})$ the vector space $V_{k}$ is a normed space with a norm $%
\Vert \cdot \Vert _{k}.$ Then the set $\mathcal{FT}_{\leq \mathfrak{r}}(%
\mathbf{V}_{D})$ of TBF tensors with bounded TB rank is an analytic Banach
(Hilbert) manifold.
\end{theorem}

\section{The TBF tensors and its natural ambient tensor Banach space}

\label{embedded_manifold}

Assume that $\{\mathbf{V}_{\alpha }\}_{\alpha \in T_{D}\setminus \{D\}}$ is
a representation of the tensor space $\mathbf{V}_{D}=\left.
_{a}\bigotimes_{\alpha \in S(D)}\mathbf{V}_{\alpha }\right. $ in the 
tree-based format and that for each $k\in \mathcal{L}(T_{D})$ the vector
space $V_{k}$ is a normed space with a norm $\Vert \cdot \Vert _{k}.$ We
start with a brief discussion about the choice of the ambient manifold for $%
\mathcal{FT}_{\mathfrak{r}}(\mathbf{V}_{D}).$ To this end assume the
existence of two norms $\Vert \cdot \Vert _{D,1}$ and $\Vert \cdot \Vert
_{D,2}$ on $\mathbf{V}_{D}.$ Then we have $\mathbf{V}_{D}\subset \overline{%
\mathbf{V}_{D}}^{\Vert \cdot \Vert _{D,1}}$ and $\mathbf{V}_{D}\subset 
\overline{\mathbf{V}_{D}}^{\Vert \cdot \Vert _{D,2}}.$ The next example
illustrates this situation.

\begin{example}
\label{example_BM1} Let $V_{1_{\|\cdot\|_1}} := H^{1,p}(I_1)$ and $%
V_{2_{\|\cdot\|_2}}= H^{1,p}(I_2).$ Take $\mathbf{V}_D:= H^{1,p}(I_1)
\otimes_a H^{1,p}(I_2).$ From Theorem~\ref{Tucker_Banach_Manifold} we obtain
that $\mathcal{FT}_{\mathfrak{r}}(\mathbf{V}_D)$ is a Banach manifold.
However, we can consider as ambient manifold either $\overline{\mathbf{V}_{D}%
}^{\|\cdot\|_{D,1}} := H^{1,p}(I_1 \times I_2)$ or $\overline{\mathbf{V}_{D}}%
^{\|\cdot\|_{D,2}} = H^{1,p}(I_1) \otimes_{\|\cdot\|_{(0,1),p}}
H^{1,p}(I_2), $ where $\|\cdot\|_{(0,1),p}$ is the norm given by 
\begin{equation*}
\|f\|_{(0,1),p}:= \left(\|f\|_p^p + \left\|\frac{\partial f}{\partial x_2}%
\right\|_p^p\right)^{1/p}
\end{equation*}
for $1 \le p < \infty.$
\end{example}

In this context two questions about the choice of a norm $\|\cdot\|_{\alpha}$
for each algebraic tensor space $\mathbf{V}_{\alpha} = \left._a
\bigotimes_{\beta \in S(\alpha)} \mathbf{V}_{\beta}\right.,$ where $\alpha
\in T_D \setminus \mathcal{L}(T_D)$ appears:

\begin{enumerate}
\item What is the good choice for these norms to show that $\mathcal{FT}_{\le 
\mathfrak{r}}(\mathbf{V}_D)$ is proximinal?

\item What is the good choice for these norms to show that $\mathcal{FT}_{%
\mathfrak{r}}(\mathbf{V}_D)$ is an immersed submanifold?
\end{enumerate}

To see this we need to introduce the topological tensor spaces in the 
tree-based format.

\subsection{Topological tensor spaces in the tree-based Format}

First, we recall the definition of some topological tensor spaces and we
will give some examples.

\begin{definition}
\label{Banach tensor product space}We say that $\mathbf{V}_{\left\Vert
\cdot\right\Vert }$ is a \emph{Banach tensor space} if there exists an
algebraic tensor space $\mathbf{V}$ and a norm $\left\Vert \cdot\right\Vert $
on $\mathbf{V}$ such that $\mathbf{V}_{\left\Vert \cdot\right\Vert }$ is the
completion of $\mathbf{V}$ with respect to the norm $\left\Vert
\cdot\right\Vert $, i.e.
\begin{equation*}
\mathbf{V}_{\left\Vert \cdot\right\Vert }:=\left. _{\left\Vert \cdot
\right\Vert }\bigotimes_{j=1}^{d}V_{j}\right. =\overline{\left.
_{a}\bigotimes\nolimits_{j=1}^{d}V_{j}\right. }^{\left\Vert \cdot\right\Vert
}.
\end{equation*}
If $\mathbf{V}_{\left\Vert \cdot\right\Vert }$ is a Hilbert space, we say
that $\mathbf{V}_{\left\Vert \cdot\right\Vert }$ is a \emph{Hilbert tensor
space}.
\end{definition}

Next, we give some examples of Banach and Hilbert tensor spaces.

\begin{example}
\label{Bsp HNp}For $I_{j}\subset \mathbb{R}$ $\left( 1\leq j\leq d\right) $
and $1\leq p<\infty ,$ the Sobolev space $H^{N,p}(I_{j})$ consists of all
univariate functions $f$ from $L^{p}(I_{j})$ with bounded norm\footnote{%
It suffices to have in \eqref{(SobolevNormp a} the terms $n=0$ and $n=N.$
The derivatives are to be understood as weak derivatives.}%
\begin{subequations}
\label{(SobolevNormp}%
\begin{equation*}
\left\Vert f\right\Vert _{N,p;I_{j}}:=\bigg(\sum_{n=0}^{N}\int_{I_{j}}\left%
\vert \partial ^{n}f\right\vert ^{p}\mathrm{d}x\bigg)^{1/p},
\label{(SobolevNormp a}
\end{equation*}%
whereas the space $H^{N,p}(\mathbf{I})$ of $d$-variate functions on $\mathbf{%
I}=I_{1}\times I_{2}\times \ldots \times I_{d}\subset \mathbb{R}^{d}$ is
endowed with the norm%
\begin{equation*}
\left\Vert f\right\Vert _{N,p}:=\Big(\sum_{0\leq \left\vert \mathbf{n}%
\right\vert \leq N}\int_{\mathbf{I}}\left\vert \partial ^{\mathbf{n}%
}f\right\vert ^{p}\mathrm{d}\mathbf{x}\Big)^{1/p}  \label{(SobolevNormp b}
\end{equation*}%
\end{subequations}%
with $\mathbf{n}\in \mathbb{N}_{0}^{d}$ being a multi-index of length $%
\left\vert \mathbf{n}\right\vert :=\sum_{j=1}^{d}n_{j}$. For $p>1$ it is
well known that $H^{N,p}(I_{j})$ and $H^{N,p}(\mathbf{I})$ are reflexive and
separable Banach spaces. Moreover, for $p=2,$ the Sobolev spaces $%
H^{N}(I_{j}):=H^{N,2}(I_{j})$ and $H^{N}(\mathbf{I}):=H^{N,2}(\mathbf{I})$
are Hilbert spaces. As a first example,%
\begin{equation*}
H^{N,p}(\mathbf{I})=\left. _{\left\Vert \cdot \right\Vert
_{N,p}}\bigotimes_{j=1}^{d}H^{N,p}(I_{j})\right.
\end{equation*}%
is a Banach tensor space. Examples of Hilbert tensor spaces are%
\begin{equation*}
L^{2}(\mathbf{I})=\left. _{\left\Vert \cdot \right\Vert
_{0,2}}\bigotimes_{j=1}^{d}L^{2}(I_{j})\right. \text{\quad and\quad }H^{N}(%
\mathbf{I})=\left. _{\left\Vert \cdot \right\Vert
_{N,2}}\bigotimes_{j=1}^{d}H^{N}(I_{j})\right. \text{ for }N\in \mathbb{N}.
\end{equation*}
\end{example}

In the definition of a tensor Banach space $\left._{\|\cdot\|} \bigotimes_{j
\in D} V_j \right.$ we have not fixed, whether $V_j,$ for $j \in D,$ are
complete or not. This leads us to introduce the following definition.

\begin{definition}
Let $D$ be a finite index set and $T_{D}$ be a dimension partition tree. Let 
$(V_{j},\Vert \cdot \Vert _{j})$ be a normed space such that $V_{j_{\Vert
\cdot \Vert _{j}}}$ is a Banach space obtained as the completion of $V_{j},$
for $j\in D,$ and consider a representation $\{\mathbf{V}_{\alpha
}\}_{\alpha \in T_{D}\setminus \{D\}}$ of the tensor space $\mathbf{V}%
_{D}=\left. _{a}\bigotimes_{j\in D}V_{j}\right. $ where for each $\alpha \in
T_{D}\setminus \mathcal{L}(T_{D})$ we have a tensor space $\mathbf{V}%
_{\alpha }=\left. _{a}\bigotimes_{\beta \in S(\alpha )}\mathbf{V}_{\beta
}\right. .$ If for each $\alpha \in T_{D}\setminus \mathcal{L}(T_{D})$ there
exists a norm $\Vert \cdot \Vert _{\alpha }$ defined on $\mathbf{V}_{\alpha
} $ such that $\mathbf{V}_{\alpha _{\Vert \cdot \Vert _{\alpha }}}=\left.
_{\Vert \cdot \Vert _{\alpha }}\bigotimes_{\beta \in S(\alpha )}V_{\beta
}\right. $ is a tensor Banach space, we say that $\{\mathbf{V}_{\alpha
_{\Vert \cdot \Vert _{\alpha }}}\}_{\alpha \in T_{D}\setminus \{D\}}$ is a
representation of the tensor Banach space $\mathbf{V}_{D_{\Vert \cdot \Vert
_{D}}}=\left. _{\Vert \cdot \Vert _{D}}\bigotimes_{j\in D}V_{j}\right. $ in
the \emph{topological tree-based format}.
\end{definition}

Since $\mathbf{V}_{\alpha} = \left._a \bigotimes_{j\in \alpha} V_{j}\right.$%
, 
\begin{equation*}
\mathbf{V}_{\alpha_{\|\cdot\|_{\alpha}}} = \left._{\|\cdot\|_{\alpha}}
\bigotimes_{\alpha \in S(D)} \mathbf{V}_{\alpha}\right. =
\left._{\|\cdot\|_{\alpha}} \bigotimes_{j\in \alpha} V_{j}\right.
\end{equation*}
holds for all $\alpha \in T_D \setminus \mathcal{L}(T_D).$

\begin{example}
Figure \ref{figZp} gives an example of a representation in the topological
tree-based format for an anisotropic Sobolev space.
\end{example}

\begin{figure}[tbp]
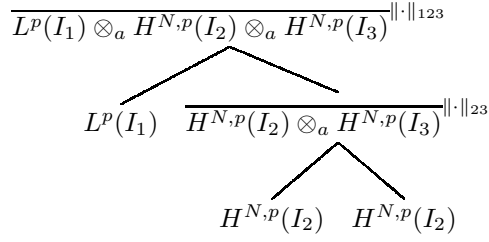

\centering
\synttree[$\overline{L^p(I_1) \otimes _a H^{N,p}(I_2) \otimes_{a} H^{N,p}(I_3)}^{\|\cdot\|_{123}}$
[$L^p(I_1)$]
[$\overline{H^{N,p}(I_2) \otimes_{a} H^{N,p}(I_3)}^{\|\cdot \|_{23}}$[$H^{N,p}(I_2)$][$H^{N,p}(I_2)$]
]]
\caption{A representation in the topological tree-based format for the
tensor Banach space $\overline{L^p(I_1) \otimes _a H^{N,p}(I_2) \otimes_{a}
H^{N,p}(I_3)}^{\|\cdot\|_{123}}.$ Here $\|\cdot\|_{23}$ and $\|\cdot\|_{123}$
are given norms.}
\label{figZp}
\end{figure}

\begin{figure}[tbp]
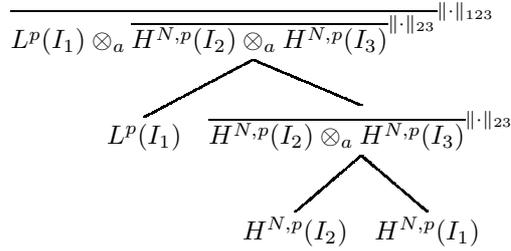

\centering
\synttree[$\overline{L^p(I_1) \otimes _a \overline{H^{N,p}(I_2) \otimes_a H^{N,p}(I_3)}^{\|\cdot \|_{23}}}^{\|\cdot\|_{123}}$
[$L^p(I_1)$]
[$\overline{H^{N,p}(I_2) \otimes_a H^{N,p}(I_3)}^{\|\cdot \|_{23}}$[$H^{N,p}(I_2)$][$H^{N,p}(I_1)$]
]]
\caption{A representation for the tensor Banach space $\overline{%
L^{p}(I_{1})\otimes _{a}\overline{H^{N,p}(I_{2})\otimes _{a}H^{N,p}(I_{3})}%
^{\Vert \cdot \Vert _{23}}}^{\Vert \cdot \Vert _{123}},$ using a tree. Here $%
\Vert \cdot \Vert _{23}$ and $\Vert \cdot \Vert _{123}$ are given norms.}
\label{figZp1}
\end{figure}

\begin{remark}
Observe that a tree as given in Figure \ref{figZp1} is not included in the
definition of the topological tree-based format. Moreover, for a tensor $%
\mathbf{v}\in L^{p}(I_{1})\otimes _{a}(H^{N,p}(I_{2})\otimes _{\Vert \cdot
\Vert _{23}}H^{N,p}(I_{3})),$ we have $U_{23}^{\min }(\mathbf{v})\subset
H^{N,p}(I_{2})\otimes _{\Vert \cdot \Vert _{23}}H^{N,p}(I_{3}).$ However, in
the topological tree-based representation of Figure \ref{figZp}, for a given 
$\mathbf{v}\in L^{p}(I_{1})\otimes _{a}H^{N,p}(I_{2})\otimes
_{a}H^{N,p}(I_{3})$ we have $U_{23}^{\min }(\mathbf{v})\subset
H^{N,p}(I_{2})\otimes _{a}H^{N,p}(I_{3}),$ and hence $U_{23}^{\min }(\mathbf{%
v})\subset U_{2}^{\min }(\mathbf{v})\otimes _{a}U_{3}^{\min }(\mathbf{v}).$
\end{remark}

The difference between the tensor spaces involved in Figure \ref{figZp} and
Figure \ref{figZp1} is the following. For all $\beta \in T_D \setminus 
\mathcal{L}(T_D)$, if $\|\cdot\|_{\beta}$ is also a norm on the tensor space 
$\left._{a} \bigotimes_{\eta \in S(\beta)} \mathbf{V}_{\eta_{\|\cdot\|_{%
\eta}}}\right.,$ we have 
\begin{equation*}
\left._{\|\cdot\|_{\beta}} \bigotimes_{\eta \in S(\beta)} \mathbf{V}%
_{\eta_{\|\cdot\|_{\eta}}}\right. \supset \mathbf{V}_{\beta_{\|\cdot\|_{%
\beta}}} = \left._{\|\cdot\|_{\beta}} \bigotimes_{\eta \in S(\beta)} \mathbf{%
V}_{\eta}\right. = \left._{\|\cdot\|_{\beta}} \bigotimes_{j\in \beta}
V_{j}\right..
\end{equation*}
A desirable property for the tensor product is that if $\|\cdot\|_{\alpha}$
is also a norm on the tensor space $\left._{a} \bigotimes_{\beta \in
S(\alpha)} \mathbf{V}_{\beta_{\|\cdot\|_{\beta}}}\right.,$ then 
\begin{align}  \label{desirable_property}
\left._{\|\cdot\|_{\alpha}} \bigotimes_{\beta \in S(\alpha)} \mathbf{V}%
_{\beta_{\|\cdot\|_{\beta}}}\right. = \left._{\|\cdot\|_{\alpha}}
\bigotimes_{\beta \in S(\alpha)} \mathbf{V}_{\beta}\right. =
\left._{\|\cdot\|_{\alpha}} \bigotimes_{j\in \alpha} V_{j}\right.
\end{align}
must be true for all $\alpha \in T_D \setminus \mathcal{L}(T_D).$ To precise
these ideas, we introduce the following definitions and results.

\bigskip

Let $\left\Vert \cdot\right\Vert _{j},$ $1\leq j\leq d,$ be the norms of the
vector spaces $V_{j}$ appearing in $\mathbf{V}=\left. _{a}\bigotimes
\nolimits_{j=1}^{d}V_{j}\right. .$ By $\left\Vert \cdot\right\Vert $ we
denote the norm on the tensor space $\mathbf{V}$. Note that $%
\left\Vert\cdot\right\Vert $ is not determined by $\left\Vert
\cdot\right\Vert _{j},$ for $j \in D,$ but there are relations which are
`reasonable'. Any norm $\left\Vert \cdot\right\Vert $ on $\left.
_{a}\bigotimes_{j=1}^{d}V_{j}\right. $ satisfying%
\begin{equation}
\Big{\|}%
\bigotimes\nolimits_{j=1}^{d}v_j%
\Big{\|}%
=\prod\nolimits_{j=1}^{d}\Vert v_j\Vert_{j}\qquad\text{for all }v_j\in V_{j}%
\text{ }\left( 1\leq j\leq d\right)  \label{(rcn a}
\end{equation}
is called a \emph{crossnorm}. As usual, the dual norm of $\left\Vert
\cdot\right\Vert $ is denoted by $\left\Vert \cdot\right\Vert ^{\ast}$. If $%
\left\Vert \cdot\right\Vert $ is a crossnorm and also $\left\Vert
\cdot\right\Vert ^{\ast}$ is a crossnorm on $\left.
_{a}\bigotimes_{j=1}^{d}V_{j}^{\ast}\right. $, i.e.,%
\begin{equation}
\Big{\|}%
\bigotimes\nolimits_{j=1}^{d}\varphi^{(j)}%
\Big{\|}%
^{\ast}=\prod\nolimits_{j=1}^{d}\Vert\varphi^{(j)}\Vert_{j}^{\ast}\qquad%
\text{for all }\varphi^{(j)}\in V_{j}^{\ast}\text{ }\left( 1\leq j\leq
d\right) ,  \label{(rcn b}
\end{equation}
then $\left\Vert \cdot\right\Vert $ is called a \emph{reasonable crossnorm}.

\begin{remark}
\label{tensor product continuity}Eq. \eqref{(rcn a} implies the inequality $%
\Vert\bigotimes\nolimits_{j=1}^{d}v_j\Vert\lesssim\prod\nolimits_{j=1}^{d}%
\Vert v_j\Vert_{j}$ which is equivalent to the continuity of the multilinear
tensor product mapping\footnote{%
Recall that a multilinear map $T$ from $%
\mathop{\mathchoice{\raise-0.22em\hbox{\huge $\times$}} {\raise-0.05em\hbox{\Large $\times$}}{\hbox{\large
$\times$}}{\times}}_{j=1}^{d} (V_{j},\|\cdot\|_j)$ equipped with the product
topology to a normed space $(W,\Vert\cdot\Vert)$ is continuous if and only
if $\Vert T\Vert <\infty$, with 
\begin{align*}
\| T \| := \sup_{\substack{ (v_1,\hdots,v_d)  \\ \|(v_1,\hdots,v_d)\|\le 1}}
\|T(v_{1},\ldots,v_{d})\| &= \sup_{\substack{ (v_1,\hdots,v_d)  \\ %
\|v_1\|_1\le 1,\hdots, \|v_d\|_d\le 1}} \|T(v_{1},\ldots,v_{d})\| = \sup_{{%
(v_1,\hdots,v_d)}} \frac{\|T (v_{1},\ldots,v_{d})\|}{\|v_1\|_1\hdots %
\|v_d\|_d}.
\end{align*}%
} between normed spaces: 
\begin{equation}
\bigotimes :\mathop{\mathchoice{\raise-0.22em\hbox{\huge $\times$}}
{\raise-0.05em\hbox{\Large $\times$}}{\hbox{\large $\times$}}{\times}}%
_{j=1}^{d}\left( V_{j},\left\Vert \cdot\right\Vert _{j}\right)
\longrightarrow%
\bigg(%
\left. _{a}\bigotimes_{j=1}^{d}V_{j}\right. ,\left\Vert \cdot\right\Vert 
\bigg)%
,  \label{(Tensorproduktabbildung}
\end{equation}
defined by $\bigotimes\left( (v_{1},\ldots,v_{d})\right)
=\bigotimes_{j=1}^{d}v_{j} $, the product space being equipped with the
product topology induced by the maximum norm $\|(v_1,\hdots,v_d)\| =
\max_{1\le j\le d} \|v_j\|_j$.\smallskip
\end{remark}

The following result is a consequence of Lemma~4.34 of \cite{Hackbusch}.

\begin{lemma}
\label{lemma434} Let $(V_j,\|\cdot\|_j)$ be normed spaces for $1 \le j \le
d. $ Assume that $\|\cdot\|$ is a norm on the tensor space $\left.
_{a}\bigotimes_{j=1}^{d}V_{j_{\|\cdot\|_j}}\right.$ such that the tensor
product map 
\begin{equation}
\bigotimes :\mathop{\mathchoice{\raise-0.22em\hbox{\huge $\times$}}
{\raise-0.05em\hbox{\Large $\times$}}{\hbox{\large $\times$}}{\times}}
_{j=1}^{d}\left( V_{j_{\|\cdot\|_j}},\left\Vert \cdot\right\Vert _{j}\right)
\longrightarrow 
\bigg(
\left. _{a}\bigotimes_{j=1}^{d}V_{j_{\|\cdot\|_j}}\right. ,\left\Vert
\cdot\right\Vert 
\bigg)
\label{(Tensorproduktabbildung1}
\end{equation}
is continuous. Then \eqref{(Tensorproduktabbildung} is also continuous and 
\begin{align*}
\left. _{\|\cdot\|}\bigotimes_{j=1}^d V_{j_{\|\cdot\|_j}}\right. = \left.
_{\|\cdot\|}\bigotimes_{j=1}^d V_{j}\right.
\end{align*}
holds.
\end{lemma}

\begin{definition}
Assume that for each $\alpha \in T_D \setminus \mathcal{L}(T_D)$ there
exists a norm $\|\cdot\|_{\alpha}$ defined on $\left._{a} \bigotimes_{\beta
\in S(\alpha)} V_{\beta_{\|\cdot\|_{\beta}}} \right..$ We will say that the
tensor product map $\bigotimes$ is $T_D$-continuous if the map 
\begin{equation*}
\bigotimes: 
\mathop{\mathchoice{\raise-0.22em\hbox{\huge
$\times$}} {\raise-0.05em\hbox{\Large $\times$}}{\hbox{\large
$\times$}}{\times}}_{\beta \in
S(\alpha)}(V_{\beta_{\|\cdot\|_{\beta}}},\|\cdot\|_{\beta}) \rightarrow
\left(\left._a \bigotimes_{\beta \in S(\alpha)}
V_{\beta_{\|\cdot\|_{\beta}}} \right., \|\cdot\|_{\alpha} \right)
\end{equation*}
is continuous for each $\alpha \in T_D \setminus \mathcal{L}(T_D).$
\end{definition}

The next result gives the conditions to have \eqref{desirable_property}.

\begin{theorem}
\label{ext_Banach} Assume that we have a representation $\{\mathbf{V}%
_{\alpha _{\Vert \cdot \Vert _{\alpha }}}\}_{\alpha \in T_{D}\setminus
\{D\}} $ in the topological tree-based format of the tensor Banach space $%
\mathbf{V}_{D_{\Vert \cdot \Vert _{D}}}=\left. _{\Vert \cdot \Vert
_{D}}\bigotimes_{\alpha \in S(D)}\mathbf{V}_{\alpha }\right. ,$ such that
for each $\alpha \in T_{D}\setminus \mathcal{L}(T_{D})$, the norm $\Vert
\cdot \Vert _{\alpha }$ is also defined on $\left. _{a}\bigotimes_{\beta \in
S(\alpha )}V_{\beta _{\Vert \cdot \Vert _{\beta }}}\right. $ and the tensor
product map $\bigotimes $ is $T_{D}$-continuous. Then 
\begin{equation*}
\left. _{\Vert \cdot \Vert _{\alpha }}\bigotimes_{\beta \in S(\alpha )}%
\mathbf{V}_{\beta _{\Vert \cdot \Vert _{\beta }}}\right. =\left. _{\Vert
\cdot \Vert _{\alpha }}\bigotimes_{\beta \in S(\alpha )}\mathbf{V}_{\beta
}\right. =\left. _{\Vert \cdot \Vert _{\alpha }}\bigotimes_{j\in \alpha
}V_{j}\right. ,
\end{equation*}%
holds for all $\alpha \in T_{D}\setminus \mathcal{L}(T_{D}).$
\end{theorem}

\begin{proof}
From Lemma \ref{lemma434}, if the tensor product map 
\begin{equation*}
\bigotimes: 
\mathop{\mathchoice{\raise-0.22em\hbox{\huge $\times$}} {\raise-0.05em\hbox{\Large $\times$}}{\hbox{\large
$\times$}}{\times}}_{\beta \in S(\alpha)} (\mathbf{V}_{\beta_{\|\cdot\|_{%
\beta}}}, \|\cdot\|_{\beta}) \longrightarrow (\left._{a} \bigotimes_{\beta
\in S(\alpha)} \mathbf{V}_{\beta_{\|\cdot\|_{\beta}}}\right.,\|\cdot\|_{%
\alpha})
\end{equation*}
is continuous, then 
\begin{equation*}
\left._{\|\cdot\|_{\alpha}} \bigotimes_{\beta \in S(\alpha)} \mathbf{V}%
_{\beta_{\|\cdot\|_{\beta}}}\right.= \left._{\|\cdot\|_{\alpha}}
\bigotimes_{\beta \in S(\alpha)} V_{\beta}\right.,
\end{equation*}
holds. Since $\mathbf{V}_{\alpha} = \left._{a} \bigotimes_{\beta \in
S(\alpha)} \mathbf{V}_{\beta}\right. = \left._{a} \bigotimes_{j \in \alpha}
V_{j}\right.,$ the theorem follows.
\end{proof}

\begin{example}
Assume that the tensor product maps 
\begin{equation*}
\bigotimes:(L^p(I_1),\|\cdot\|_{0,p;I_1}) \times (H^{N,p}(I_2)
\otimes_{\|\cdot\|_{23}} H^{N,p}(I_3), \|\cdot\|_{23}) \rightarrow (L^p(I_1)
\otimes_a (H^{N,p}(I_2) \otimes_{\|\cdot\|_{23}} H^{N,p}(I_3)),
\|\cdot\|_{123} )
\end{equation*}
and 
\begin{equation*}
\bigotimes:(H^{N,p}(I_2), \|\cdot\|_{N,p;I_2}) \times (H^{N,p}(I_3),
\|\cdot\|_{N,p;I_3}) \rightarrow (H^{N,p}(I_2) \otimes_{a} H^{N,p}(I_3),
\|\cdot\|_{23} )
\end{equation*}
are continuous. Then the trees of Figure \ref{figZp} and Figure \ref{figZp1}
are the same.
\end{example}

The next result is a consequence of the well-known fact that every
continuous multilinear map between normed spaces is also Fr\'{e}chet
differentiable (see (2.1.22) in \cite{Berger}).

\begin{proposition}
\label{Diff_otimes} Let $(V_j,\|\cdot\|_j)$ be normed spaces for $1 \le j
\le d.$ Assume that $\|\cdot\|$ is a norm onto the tensor space $\left.
_{a}\bigotimes_{j=1}^{d}V_{j_{\|\cdot\|_j}}\right.$ such that the tensor
product map \eqref{(Tensorproduktabbildung1} is continuous. Then it is also
Fr\'{e}chet differentiable and its differential is given by 
\begin{equation*}
D\left( \bigotimes (v_{1},\ldots ,v_{d})\right) (w_{1},\ldots
,w_{d})=\sum_{j=1}^{d}v_{1}\otimes \ldots \otimes v_{j-1}\otimes
w_{j}\otimes v_{j+1}\otimes \cdots v_{d}.
\end{equation*}
\end{proposition}

\subsubsection{On the best 
approximation in $\mathcal{FT}_{\le \mathfrak{r}}(\mathbf{V}_D)$}

Now we discuss about the best approximation problem in  $\mathcal{FT}_{\le \mathfrak{r%
}}(\mathbf{V}_D)$. For this, we need a stronger
condition than the $T_{D}$-continuity of the tensor product. Grothendieck 
\cite{Grothendiek1953} named the following norm $\left\Vert \cdot
\right\Vert _{\vee }$ the \emph{injective norm}.

\begin{definition}
Let $V_{i}$ be a Banach space with norm $\left\Vert \cdot\right\Vert _{i}$
for $1\leq i\leq d.$ Then for $\mathbf{v}\in\mathbf{V}=\left.
_{a}\bigotimes_{j=1}^{d}V_{j}\right. $ define $\left\Vert \cdot\right\Vert
_{\vee(V_1,\ldots,V_d)}$ by%
\begin{equation}
\left\Vert \mathbf{v}\right\Vert _{\vee(V_1,\ldots,V_d)}:=\sup\left\{ \frac{%
\left\vert \left(
\varphi_{1}\otimes\varphi_{2}\otimes\ldots\otimes\varphi_{d}\right) (\mathbf{%
v})\right\vert }{\prod_{j=1}^{d}\Vert\varphi_{j}\Vert_{j}^{\ast}}%
:0\neq\varphi_{j}\in V_{j}^{\ast},1\leq j\leq d\right\} .
\label{(Norm ind*(V1,...,Vd)}
\end{equation}
\end{definition}

It is well known that the injective norm is a reasonable crossnorm (see
Lemma 1.6 in \cite{Light} and \eqref{(rcn a}-\eqref{(rcn b}). Further
properties are given by the next proposition (see Lemma 4.96 and 4.2.4 in 
\cite{Hackbusch}).

\begin{proposition}
\label{injective bounded below} Let $V_{i}$ be a Banach space with norm $%
\left\Vert \cdot\right\Vert _{i}$ for $1\leq i\leq d,$ and $\|\cdot\|$ be a
norm on $\mathbf{V}:= \left._a \bigotimes_{j=1}^d V_j\right..$ The following
statements hold.

\begin{itemize}
\item[(a)] For each $1 \le j \le d$ introduce the tensor Banach space $%
\mathbf{X}_j:=\left._{\|\cdot\|_{\vee(V_1,\ldots,V_{j-1},V_{j+1},%
\ldots,V_d)}} \bigotimes_{k \neq j}V_k \right..$ Then 
\begin{equation}
\|\cdot\|_{\vee(V_1,\ldots,V_d)} = \|\cdot\|_{\vee\left(V_j, \mathbf{X}_j
\right)}
\end{equation}
holds for $1 \le j \le d.$

\item[(b)] The injective norm is the weakest reasonable crossnorm on $%
\mathbf{V},$ i.e., if $\left\Vert \cdot\right\Vert $ is a reasonable
crossnorm on $\mathbf{V},$ then 
\begin{equation}
\left\Vert \cdot\right\Vert \;\gtrsim\ \left\Vert \cdot\right\Vert
_{\vee(V_1,\ldots,V_d)}.  \label{(Norm staerker als ind*}
\end{equation}

\item[(c)] For any norm $\left\Vert \cdot\right\Vert $ on $\mathbf{V}$
satisfying $\left\Vert \cdot\right\Vert
_{\vee(V_1,\ldots,V_d)}\lesssim\left\Vert \cdot\right\Vert ,$ the map (\ref%
{(Tensorproduktabbildung}) is continuous, and hence Fr\'echet differentiable.
\end{itemize}
\end{proposition}

In Corollary 4.4 in \cite{FALHACK} the following result, which is re-stated
here using the notations of the present paper, is proved as a consequence of
a similar result showed for tensors in Tucker format with bounded rank.

\begin{theorem}
\label{classical_best_approx} Let $\mathbf{V}_{D}=\left.
_{a}\bigotimes_{j\in D}V_{j}\right. $ and let $\{\mathbf{V}_{\alpha
_{j}\left. _{\Vert \cdot \Vert _{\alpha _{j}}}\right. }:2\leq j\leq d\}\cup
\{V_{j_{\Vert \cdot \Vert _{j}}}:1\leq j\leq d\}$ for $d\geq 3,$ be a
representation of a reflexive Banach tensor space $\mathbf{V}_{D_{\Vert
\cdot \Vert _{D}}}=\left. _{\Vert \cdot \Vert _{D}}\bigotimes_{j\in
D}V_{j}\right. ,$ in topological tree-based format such that

\begin{enumerate}
\item[(a)] $\|\cdot\|_D \gtrsim
\|\cdot\|_{\vee(V_{1_{\|\cdot\|_j}},\ldots,V_{d_{\|\cdot\|_d}})},$

\item[(b)] $\mathbf{V}_{\alpha_d} = V_{d-1} \otimes_a V_d,$ and $\mathbf{V}%
_{\alpha_{j}} = V_{j-1} \otimes_a \mathbf{V}_{\alpha_{j+1}},$ for $2 \le j
\le d-1,$ and

\item[(c)] $\|\cdot\|_{\alpha_j} :=
\|\cdot\|_{\vee(V_{\left.j-1\right._{\|\cdot\|_{j-1}}},\ldots,V_{d_{\|\cdot%
\|_d}})}$ for $2 \le j \le d.$
\end{enumerate}

Then for each $\mathbf{v} \in \mathbf{V}_{D_{\|\cdot\|_D}}$ there exists $%
\mathbf{u}_{best} \in \mathcal{FT}_{\le \mathfrak{r}}(\mathbf{V}_D)$ such
that 
\begin{align*}
\|\mathbf{v}-\mathbf{u}_{best}\|_D = \min_{\mathbf{u} \in \mathcal{FT}_{\le 
\mathfrak{r}}(\mathbf{V}_D)}\|\mathbf{v}-\mathbf{u}\|_D.
\end{align*}
\end{theorem}

It seems clear that tensor Banach spaces as we show in Example \ref{figZp1}
are not included in this framework. So a natural question is if for a
representation in the topological tree-based format of a reflexive Banach
space the statement of Theorem \ref%
{classical_best_approx} is also true. To prove this, we will reformulate
some of the results given in \cite{FALHACK}. In the aforementioned paper,
the milestone to prove the existence of a best approximation is the
extension of the definition of minimal subspace for tensors $\mathbf{v}\in 
\mathbf{V}_{D_{\Vert \cdot \Vert _{D}}}\setminus \mathbf{V}_{D}.$ To do this
the authors use a similar result to the following lemma (see Lemma 3.8 in 
\cite{FALHACK}).

\begin{lemma}
\label{extension_Umin} Let $V_{j_{\|\cdot\|_j}}$ be a Banach space for $j\in
D,$ where $D$ is a finite index set, and $\alpha_1,\ldots,\alpha_m \subset
2^D \setminus \{D,\emptyset\},$ be such that $\alpha_i \cap \alpha_j =
\emptyset$ for all $i \neq j$ and $D = \bigcup_{i=1}^m \alpha_i.$ Assume
that if $\#\alpha_i \ge 2$ for some $1 \le i \le m,$ then $\mathbf{V}%
_{\left.\alpha_i\right._{\|\cdot\|_{\alpha_i}}}$ is a tensor Banach space.
Consider the tensor space 
\begin{equation*}
\mathbf{V}_D:= \left._a \bigotimes_{i=1}^m \mathbf{V}_{\left.\alpha_i%
\right._{\|\cdot\|_{\alpha_i}}} \right.
\end{equation*}
endowed with the injective norm $\|\cdot\|_{\vee(\mathbf{V}%
_{\left.\alpha_1\right._{\|\cdot\|_{\alpha_1}}},\ldots,\mathbf{V}%
_{\left.\alpha_m\right._{\|\cdot\|_{\alpha_m}}})}.$ Fix $1 \le k \le m,$
then given $\boldsymbol{\varphi}_{[\alpha_k]}\in \left._a \bigotimes_{i \neq
k} \mathbf{V}_{\left.\alpha_i\right._{\|\cdot\|_{\alpha_i}}}^*\right.$ the
map $id_{\alpha_k} \otimes \boldsymbol{\varphi}_{[\alpha_k]}$ belongs to $%
\mathcal{L}\left(\mathbf{V}_D, \mathbf{V}_{\left.\alpha_k\right._{\|\cdot%
\|_{\alpha_k}}} \right).$ Moreover, $\overline{id_{\alpha_k} \otimes 
\boldsymbol{\varphi}_{[\alpha_k]}} \in \mathcal{L}(\overline{\mathbf{V}_D}%
^{\|\cdot\|,},\mathbf{V}_{\left.\alpha_k\right._{\|\cdot\|_{\alpha_k}}})$
for any norm satisfying 
\begin{equation*}
\|\cdot\| \gtrsim \|\cdot\|_{\vee(\mathbf{V}_{\left.\alpha_1\right._{\|\cdot%
\|_{\alpha_1}}},\ldots,\mathbf{V}_{\left.\alpha_m\right._{\|\cdot\|_{%
\alpha_m}}})}.
\end{equation*}
\end{lemma}

Let $\{\mathbf{V}_{\alpha _{\Vert \cdot \Vert _{\alpha }}}\}_{\alpha \in
T_{D}\setminus \{D\}}$ be a representation of the Banach tensor space $%
\mathbf{V}_{D_{\Vert \cdot \Vert _{D}}}=\left. _{\Vert \cdot \Vert
_{D}}\bigotimes_{j\in D}V_{j}\right. ,$ in the topological tree-based format
and assume that the tensor product map $\bigotimes$ is $T_D$-continuous.
From Theorem \ref{ext_Banach}, we may assume that we have a tensor Banach
space 
\begin{equation*}
\mathbf{V}_{\alpha_{\|\cdot\|_{\alpha}}} = \left._{\|\cdot\|_{\alpha}}
\bigotimes_{\beta\in S(\alpha)} V_{\beta_{\|\cdot\|_{\beta}}} \right.
\end{equation*}
for each $\alpha \in T_D\setminus \mathcal{L}(T_D),$ and a Banach space $%
V_{j_{\|\cdot\|_j}}$ for $j \in \mathcal{L}(T_D).$ Let $\alpha \in
T_D\setminus \mathcal{L}(T_D).$ To simplify the notation we write for $A,B
\subset S(\alpha) $ 
\begin{equation*}
\|\cdot\|_{\vee(A)}:= \|\cdot\|_{\vee(\{\mathbf{V}_{\delta_{\|\cdot\|_{%
\delta}}}:\delta \in A\})},
\end{equation*}
and 
\begin{equation*}
\|\cdot\|_{\vee(A,\vee(B))}:= \|\cdot\|_{\vee(\{\mathbf{V}%
_{\delta_{\|\cdot\|_{\delta}}}:\delta \in A\}, \mathbf{X}_B)}
\end{equation*}
where 
\begin{equation*}
\mathbf{X}_B:= \left._{\|\cdot\|_{\vee(B)}} \bigotimes_{\beta \in B} \mathbf{%
V}_{\beta_{\|\cdot\|_{\beta}}}\right..
\end{equation*}
From Proposition~\ref{injective bounded below}(a), we can write 
\begin{align*}
\|\cdot\|_{\vee(S(\alpha))} = \|\cdot\|_{\vee(\beta,\vee(S(\alpha)\setminus
\beta))}  \label{tree_injective_norm}
\end{align*}
for each $\beta \in S(\alpha).$ From now on, we assume that 
\begin{align}
\|\cdot\|_{\alpha} \gtrsim \|\cdot\|_{\vee(S(\alpha))} \text{ for each }
\alpha \in T_D\setminus \mathcal{L}(T_D),
\end{align}
holds. Recall that Proposition~\ref{injective bounded below}(c) implies that
the tensor product map $\bigotimes$ is $T_D$-continuous. Since $%
\|\cdot\|_{\alpha} \gtrsim \|\cdot\|_{\vee(\beta,\vee(S(\alpha)\setminus
\beta))}$ holds for each $\beta\in S(\alpha)$, the tensor product map 
\begin{align*}
\bigotimes:(\mathbf{V}_{\beta_{\|\cdot\|_{\beta}}},\|\cdot\|_{\beta}) \times
\left( \left._{\|\cdot\|_{\vee(S(\alpha)\setminus \beta)}}
\bigotimes_{\delta \in S(\alpha)\setminus \{\beta\}}\mathbf{V}_{\delta_{\|\cdot%
\|_{\delta}}} \right. , \|\cdot\|_{\vee(S(\alpha)\setminus \beta)}\right)
\rightarrow (\mathbf{V}_{\alpha_{\|\cdot\|_{\alpha}}},\|\cdot\|_{\alpha})
\end{align*}
is also continuous for each $\beta \in S(\alpha).$ Moreover, by Theorem \ref%
{ext_Banach}, 
\begin{equation*}
\mathbf{V}_{\alpha _{\Vert \cdot \Vert _{\alpha }}}=\left. _{\Vert \cdot
\Vert _{\alpha }}\bigotimes_{\beta \in S(\alpha )}V_{\beta _{\Vert \cdot
\Vert _{\beta }}}\right. =\left. _{\Vert \cdot \Vert _{\alpha
}}\bigotimes_{\beta \in S(\alpha )}V_{\beta }\right. =\left. _{\Vert \cdot
\Vert _{\alpha }}\bigotimes_{j\in \alpha }V_{j}\right. ,
\end{equation*}%
holds for each $\alpha \in T_{D}\setminus \mathcal{L}(T_{D}).$ Observe, that 
$\mathbf{V}_{\alpha _{\Vert \cdot \Vert _{\alpha }}}^{\ast }\subset \mathbf{V%
}_{\alpha }^{\ast }$ for all $\alpha \in S(D).$ Take $\mathbf{V}_{D}=\left.
_{a}\bigotimes_{j\in D}V_{j}\right. .$ Since $\Vert \cdot \Vert _{D}\gtrsim
\Vert \cdot \Vert _{\vee (S(D ))},$ from Lemma \ref{extension_Umin} and
Proposition \ref{(Ualpha in Tensor Uj coro}(b), we can extend for $\mathbf{v}%
\in \mathbf{V}_{D_{\Vert \cdot \Vert _{D}}}\setminus \mathbf{V}_{D},$ the
definition of minimal subspace for each $\alpha \in S(D)$ as 
\begin{equation*}
U_{\alpha }^{\min }(\mathbf{v}):=\left\{ \overline{(id_{\alpha }\otimes 
\boldsymbol{\varphi }_{[\alpha ]})}(\mathbf{v}):\boldsymbol{\varphi }%
_{[\alpha ]}\in \left. _{a}\bigotimes_{\substack{ \beta \in S(D) \setminus
\{\alpha \}}}\mathbf{V}_{\beta }^{\ast }\right. \right\} .
\end{equation*}%
Observe that $\overline{(id_{\alpha }\otimes \boldsymbol{\varphi }_{[\alpha
]})}\in \mathcal{L}(\mathbf{V}_{D_{\Vert \cdot \Vert _{D}}},\mathbf{V}%
_{\alpha _{\Vert \cdot \Vert _{\alpha }}}).$ Recall that if $\mathbf{v}\in 
\mathbf{V}_{D}$ and $\alpha \notin \mathcal{L}(T_{D}),$ from Proposition \ref%
{inclusin_Umin}, we have $U_{\alpha }^{\min }(\mathbf{v})\subset \left.
_{a}\bigotimes_{\beta \in S(\alpha )}U_{\beta }^{\min }(\mathbf{v})\right.
\subset \left. _{a}\bigotimes_{\beta \in S(\alpha )}\mathbf{V}_{\beta
}\right. .$ Moreover, by Proposition \ref{(Ualpha in Tensor Uj coro}(b), for 
$\beta \in S(\alpha )$ we have 
\begin{align*}
U_{\beta }^{\min }(\mathbf{v})& = \mathrm{span}\, \left\{ (id_{\beta
}\otimes \boldsymbol{\varphi }_{[\beta ]})(\mathbf{v}_{\alpha }): \mathbf{v}%
_{\alpha}\in U_{\alpha }^{\min }(\mathbf{v}) \text{ and } \boldsymbol{%
\varphi }_{[\beta ]}\in \left._{a}\bigotimes_{\substack{ \delta\in
S(\alpha)\setminus \{\beta\}}} \mathbf{V}_{\delta }^{\ast } \right. \right\}
\\
& =\mathrm{span}\, \left\{ (id_{\beta }\otimes \boldsymbol{\varphi }%
_{[\beta]})\circ (id_{\alpha }\otimes \boldsymbol{\varphi }_{[\alpha ]})(%
\mathbf{v}):\boldsymbol{\varphi }_{[\alpha ]}\in \left. _{a}\bigotimes 
_{\substack{ \mu \in S(D) \setminus \{\alpha\} }}\mathbf{V}_{\mu }^{\ast
}\right. \text{ and }\boldsymbol{\varphi }_{[\beta ]}\in \left.
_{a}\bigotimes_{\substack{ \delta\in S(\alpha) \setminus \{\beta\}}} \mathbf{%
V}_{\delta }^{\ast }\right. \right\} .
\end{align*}%
Thus, $(id_{\alpha }\otimes \boldsymbol{\varphi }_{[\alpha ]})(\mathbf{v}%
)\in U_{\alpha }^{\min }(\mathbf{v})\subset \mathbf{V}_{\alpha }\subset 
\mathbf{V}_{\alpha _{\Vert \cdot \Vert _{\alpha }}},$ and hence 
\begin{equation*}
(id_{\beta }\otimes \boldsymbol{\varphi }_{[\beta ]})\circ (id_{\alpha
}\otimes \boldsymbol{\varphi }_{[\alpha ]})(\mathbf{v})\in U_{\beta }^{\min
}(\mathbf{v})\subset \mathbf{V}_{\beta }\subset \mathbf{V}_{\beta _{\Vert
\cdot \Vert _{\beta }}},
\end{equation*}%
when $\#\beta \geq 2.$ However, if $\mathbf{v}\in \mathbf{V}_{D_{\Vert \cdot
\Vert _{D}}}\setminus \mathbf{V}_{D}$ then $\overline{(id_{\alpha }\otimes 
\boldsymbol{\varphi }_{[\alpha ]})}(\mathbf{v})\in U_{\alpha }^{\min }(%
\mathbf{v})\subset \mathbf{V}_{\alpha _{\Vert \cdot \Vert _{\alpha }}}.$
Since $\Vert \cdot \Vert _{\alpha }\gtrsim \Vert \cdot \Vert _{\vee
(S(\alpha ))}$ also by Lemma \ref{extension_Umin} we have $\overline{%
id_{\beta }\otimes \boldsymbol{\varphi }_{[\beta ]}}\in \mathcal{L}(\mathbf{V%
}_{\alpha _{\Vert \cdot \Vert _{\alpha }}},\mathbf{V}_{\beta _{\Vert \cdot
\Vert _{\beta }}}).$ In consequence, a natural extension of the definition
of minimal subspace $U_{\beta }^{\min }(\mathbf{v}),$ for $\mathbf{v}\in 
\mathbf{V}_{D_{\Vert \cdot \Vert _{D}}}\setminus \mathbf{V}_{D},$ is given
by 
\begin{equation*}
U_{\beta }^{\min }(\mathbf{v}):=\mathrm{span}\, \left\{ \overline{%
(id_{\beta}\otimes \boldsymbol{\varphi }_{[\beta ]})}\circ \overline{%
(id_{\alpha }\otimes \boldsymbol{\varphi }_{[\alpha ]})}(\mathbf{v}):%
\boldsymbol{\varphi }_{[\alpha ]}\in \left. _{a}\bigotimes_{\substack{ \mu
\in S(D) \setminus \{\alpha\}}}\mathbf{V}_{\mu }^{\ast }\right. \text{ and } 
\boldsymbol{\varphi }_{[\beta ]}\in \left._{a}\bigotimes_{\substack{ %
\delta\in S(\alpha) \setminus \{\beta\}}} \mathbf{V}_{\delta }^{\ast
}\right. \right\}.
\end{equation*}%
To simplify the notation, we can write 
\begin{equation*}
\overline{(id_{\beta }\otimes \boldsymbol{\varphi }_{[\beta ,\alpha ]})}(%
\mathbf{v}):=\overline{(id_{\beta }\otimes \boldsymbol{\varphi }_{[\beta ]})}%
\circ \overline{(id_{\alpha }\otimes \boldsymbol{\varphi }_{[\alpha ]})}(%
\mathbf{v})
\end{equation*}%
where $\boldsymbol{\varphi }_{[\beta ,\alpha ]}:=\boldsymbol{\varphi }%
_{[\alpha ]}\otimes \boldsymbol{\varphi }_{[\beta ]}\in \left( \left.
_{a}\bigotimes_{\substack{ \mu \in S(D) \setminus \{\alpha\} }}\mathbf{V}%
_{\mu }^{\ast }\right. \right) \otimes _{a}\left( \left. _{a}\bigotimes 
_{\substack{ \delta \in S(\alpha ) \setminus \{\beta\} }}\mathbf{V}_{\delta
}^{\ast }\right. \right) $ and $\overline{(id_{\beta }\otimes \boldsymbol{%
\varphi }_{[\beta ,\alpha ]})}\in \mathcal{L}(\mathbf{V}_{D_{\Vert \cdot
\Vert _{D}}},\mathbf{V}_{\beta _{\Vert \cdot \Vert _{\beta }}}).$ Proceeding
inductively, from the root to the leaves, we define the minimal subspace $%
U_{j}^{\min }(\mathbf{v})$ for each $j\in \mathcal{L}(T_{D})$ such that
there exists $\eta \in T_{D}\setminus \{D\}$ with $j\in S(\eta )$ as 
\begin{equation*}
U_{j}^{\min }(\mathbf{v}):=\mathrm{span}\,\left\{ \overline{(id_{j}\otimes 
\boldsymbol{\varphi }_{[j,\eta ,\ldots ,\beta ,\alpha ]})}(\mathbf{v}):%
\boldsymbol{\varphi }_{[j,\eta,\ldots ,\beta ,\alpha ]}\in \mathbf{W}%
_{j}\right\} ,
\end{equation*}%
where 
\begin{equation*}
\mathbf{W}_{j}:=\left( \left. _{a}\bigotimes_{\substack{ \mu \in S(D)
\setminus \{\alpha\} }} \mathbf{V}_{\mu }^{\ast }\right. \right) \otimes
_{a}\left( \left. _{a}\bigotimes_{\substack{ \delta \in S(\alpha ) \setminus
\{\beta\} }}\mathbf{V}_{\delta }^{\ast }\right. \right) \otimes _{a}\cdots
\otimes _{a}\left( \left. _{a}\bigotimes_{\substack{ k\in S(\eta )\setminus
\{j\} }}V_{k}^{\ast }\right. \right) .
\end{equation*}%
With this extension the following result  can be shown (see Lemma 3.13 in 
\cite{FALHACK}).

\begin{lemma}
Let $\{\mathbf{V}_{\alpha _{\Vert \cdot \Vert _{\alpha }}}\}_{\alpha \in
T_{D}\setminus \{D\}}$ be a representation of the Banach tensor space $%
\mathbf{V}_{D_{\Vert \cdot \Vert _{D}}}=\left. _{\Vert \cdot \Vert
_{D}}\bigotimes_{j\in D}V_{j}\right. ,$ in the topological tree-based format
and assume that \eqref{tree_injective_norm} holds. Let $\{\mathbf{v}%
_{n}\}_{n\geq 0}\subset \mathbf{V}_{D_{\Vert \cdot \Vert _{D}}}$ with $%
\mathbf{v}_{n}\rightharpoonup \mathbf{v},$ and $\mu \in T_{D}\setminus
(\{D\}\cup \mathcal{L}(T_{D}))$. Then for each $\gamma \in S(\mu )$ we have 
\begin{equation*}
\overline{(id_{\gamma }\otimes \boldsymbol{\varphi }_{[\gamma ,\mu ,\cdots
,\beta ,\alpha ]})}(\mathbf{v}_{n})\rightharpoonup \overline{(id_{\gamma
}\otimes \boldsymbol{\varphi }_{[\gamma ,\mu ,\cdots ,\beta ,\alpha ]})}(%
\mathbf{v})\text{ in }\mathbf{V}_{\gamma _{\Vert \cdot \Vert _{\gamma }}},
\end{equation*}%
for all $\boldsymbol{\varphi }_{[\gamma ,\mu ,\cdots ,\beta ,\alpha ]}\in
\left( \left. _{a}\bigotimes_{\substack{ \mu \in S(D)\setminus \{\alpha \}}}%
\mathbf{V}_{\mu }^{\ast }\right. \right) \otimes _{a}\left( \left.
_{a}\bigotimes_{\substack{ \delta \in S(\alpha )\setminus \{\beta \}}}%
\mathbf{V}_{\delta }^{\ast }\right. \right) \otimes _{a}\cdots \otimes
_{a}\left( \left. _{a}\bigotimes_{\substack{ \eta \in S(\mu )\setminus
\{\gamma \}}}V_{\eta }^{\ast }\right. \right) .$
\end{lemma}

Then in a similar way as Theorem 3.15 in \cite{FALHACK} the following
theorem can be shown.

\begin{theorem}
Let $\{\mathbf{V}_{\alpha _{\Vert \cdot \Vert _{\alpha }}}\}_{\alpha \in
T_{D}\setminus \{D\}}$ be a representation of the Banach tensor space $%
\mathbf{V}_{D_{\Vert \cdot \Vert _{D}}}=\left. _{\Vert \cdot \Vert
_{D}}\bigotimes_{j\in D}V_{j}\right. ,$ in the topological tree-based format
and assume that \eqref{tree_injective_norm} holds. Let $\{\mathbf{v}%
_{n}\}_{n\geq 0}\subset \mathbf{V}_{D_{\Vert \cdot \Vert _{D}}}$ with $%
\mathbf{v}_{n}\rightharpoonup \mathbf{v},$ then 
\begin{equation*}
\dim \overline{U_{\alpha }^{\min }(\mathbf{v})}^{\Vert \cdot \Vert _{\alpha
}}=\dim U_{\alpha }^{\min }(\mathbf{v})\leq \liminf_{n\rightarrow \infty
}\dim U_{\alpha }^{\min }(\mathbf{v}_{n}),
\end{equation*}%
for all $\alpha \in T_{D}\setminus \{D\}.$
\end{theorem}

Now, following the proof of Theorem 4.1 in \cite{FALHACK} we obtain the
final theorem.

\begin{theorem}
\label{best_approximation_FT} Let $\mathbf{V}_{D}=\left.
_{a}\bigotimes_{j\in D}V_{j}\right. $ and let $\{\mathbf{V}_{\alpha _{\Vert
\cdot \Vert _{\alpha }}}\}_{\alpha \in T_{D}\setminus \{D\}}$ be a
representation of a reflexive Banach tensor space $\mathbf{V}_{D_{\Vert
\cdot \Vert _{D}}}=\left. _{\Vert \cdot \Vert _{D}}\bigotimes_{j\in
D}V_{j}\right. ,$ in the topological tree-based format and assume that \eqref%
{tree_injective_norm} holds. Then the set $\mathcal{FT}_{\leq \mathfrak{r}}(%
\mathbf{V}_{D})$ is weakly closed in $\mathbf{V}_{D_{\Vert \cdot \Vert _{D}}}
$ and hence for each $\mathbf{v}\in \mathbf{V}_{D_{\Vert \cdot \Vert _{D}}}$
there exists $\mathbf{u}_{best}\in \mathcal{FT}_{\leq \mathfrak{r}}(\mathbf{V%
}_{D})$ such that 
\begin{equation*}
\Vert \mathbf{v}-\mathbf{u}_{best}\Vert _{D}=\min_{\mathbf{u}\in \mathcal{FT}%
_{\leq \mathfrak{r}}(\mathbf{V}_{D})}\Vert \mathbf{v}-\mathbf{u}\Vert _{D}.
\end{equation*}
\end{theorem}

\subsection{Is $\mathcal{FT}_{\mathfrak{r}}(\mathbf{V}_D)$ an immersed
submanifold?}

Assume that the tensor product map $\bigotimes$ is $T_D$-continuous and that
we have a natural ambient space for $\mathcal{FT}_{\mathfrak{r}}(\mathbf{V}%
_{D})$ given by a Banach tensor space $\overline{\mathbf{V}_D}^{\|\cdot\|_D}
= \mathbf{V}_{D_{\|\cdot\|_D}}.$ Since the natural inclusion 
\begin{equation*}
\mathfrak{i}:\mathcal{FT}_{\mathfrak{r}}(\mathbf{V}_{D})\longrightarrow 
\mathbf{V}_{D_{\Vert \cdot \Vert _{D}}},
\end{equation*}
given by $\mathfrak{i}(\mathbf{v})=\mathbf{v},$ is an injective map we will
study $\mathfrak{i}$ as a function between Banach manifolds. To this end we
recall the definition of an immersion between manifolds.

\begin{definition}
Let $F:X\rightarrow Y$ be a morphism between Banach manifolds and let $x\in
X.$ We shall say that $F$ is an \emph{immersion at $x,$} if there exists an
open neighbourhood $X_{x}$ of $x$ in $X$ such that the restriction of $F$ to 
$X_{x}$ induces an isomorphism from $X_{x}$ onto a submanifold of $Y.$ We
say that $F$ is an \emph{immersion} if it is an immersion at each point of $%
X.$
\end{definition}

Our next step is to recall the definition of the differential as a morphism
which gives a linear map between the tangent spaces of the manifolds
involved with the morphism.

\begin{definition}
Let $X$ and $Y$ be two Banach manifolds. Let $F:X\rightarrow Y$ be a $%
\mathcal{C}^{r}$ morphism, i.e., 
\begin{equation*}
\psi \circ F\circ \varphi ^{-1}:\varphi (U)\rightarrow \psi (W)
\end{equation*}%
is a $\mathcal{C}^{r}$-Fr\'{e}chet differentiable map, where $(U,\varphi )$
is a chart in $X$ at $x$ and $(W,\psi )$ is a chart in $Y$ at $F(x)$. For $%
x\in X,$ we define 
\begin{equation*}
\mathrm{T}_{x}F:\mathbb{T}_{x}(X)\longrightarrow \mathbb{T}_{F(x)}(Y),\quad
v\mapsto \lbrack (\psi \circ F\circ \varphi ^{-1})^{\prime }(\varphi (x))]v.
\end{equation*}
\end{definition}

For Banach manifolds we have the following criterion for immersions (see
Theorem 3.5.7 in \cite{MRA}).

\begin{proposition}
\label{prop_inmersion} Let $X,Y$ be Banach manifolds of class $\mathcal{C}%
^{p}$ $(p\geq 1).$ Let $F:X\rightarrow Y$ be a $\mathcal{C}^{p}$ morphism
and $x\in X.$ Then $F$ is an immersion at $x$ if and only if $\mathrm{T}%
_{x}F $ is injective and $\mathrm{T}_{x}F(\mathbb{T}_{x}(X)) \in \mathbb{G}(%
\mathbb{T}_{F(x)}(Y)).$
\end{proposition}

A concept related to an immersion between Banach manifolds is the
following definition.

\begin{definition}
Assume that $X$ and $Y$ are Banach manifolds and let $f:X\longrightarrow Y$
be a $\mathcal{C}^{r}$ morphism. If $f$ is an injective immersion, then $%
f(X) $ is called an \emph{immersed submanifold of $Y$.}
\end{definition}

Recall that there exists injective immersions which are not isomorphisms
onto manifolds. It allows us to introduce the following definition.

\begin{definition}
An injective immersion $f:X\longrightarrow Y$ which is a homeomorphism onto $%
f(X)$ with the relative topology induced from $Y$ is called an \emph{%
embedding}. Moreover, if $f:X\longrightarrow Y$ is an embedding, then $f(X)$
is called an \emph{embedded submanifold of $Y.$}
\end{definition}

A classical example of an immersed submanifold which is not an embedded
submanifold is given by the map $f:(3\pi /4,7\pi /4)\longrightarrow \mathbb{R%
}^{2},$ written in polar coordinates by $r=\cos 2\theta .$ It can be see
that $f$ is an injective immersion however $f(3\pi /4,7\pi /4)$ is not an
open set in $\mathbb{R}^{2},$ because any neighbourhood of $0$ in $\mathbb{R}%
^{2}$ intersects $f(3\pi /4,7\pi /4)$ in a set with "corners" which is not
homeomorphic to an open interval (see Figure~\ref{eight}). Before to give an
example with tensors we need the following lemma.

\begin{lemma}
\label{projection_manifold} For each $\alpha \in T_D \setminus \{D\},$ the
set $\mathcal{L}(U_{\alpha}^{\min}(\mathbf{v}),W_{\alpha}^{\min}(\mathbf{v}%
)) $ is a complemented subspace of $\mathcal{L}( \mathbf{V}_{\alpha_{\Vert
\cdot \Vert _{\alpha }}} ,\mathbf{V}_{\alpha_{\Vert \cdot \Vert _{\alpha
}}}).$ Hence for each $\mathbf{v} \in \mathbf{V}_D$ and $\beta \notin 
\mathcal{L}(T_D)$ the set $%
\mathop{\mathchoice{\raise-0.22em\hbox{\huge $\times$}} {\raise-0.05em\hbox{\Large $\times$}}{\hbox{\large
$\times$}}{\times}}_{\alpha \in S(\beta)}\mathcal{L}(U_{\alpha }^{\min }(%
\mathbf{v}),W_{\alpha }^{\min }(\mathbf{v})))$ is a complemented subspace of
the Banach space $%
\mathop{\mathchoice{\raise-0.22em\hbox{\huge $\times$}} {\raise-0.05em\hbox{\Large $\times$}}{\hbox{\large
$\times$}}{\times}}_{\alpha \in S(\beta)}\mathcal{L}(\mathbf{V}%
_{\alpha_{\|\cdot\|_{\alpha}}},\mathbf{V}_{\alpha_{\|\cdot\|_{\alpha}}}).$
\end{lemma}

\begin{proof}
Observe that the map 
\begin{equation*}
\Pi_{\alpha}: \mathcal{L}\left(\mathbf{V}_{\alpha_{\Vert \cdot \Vert
_{\alpha }}} ,\mathbf{V}_{\alpha_{\Vert \cdot \Vert _{\alpha }}}\right)
\rightarrow \mathcal{L}\left( \mathbf{V}_{\alpha_{\Vert \cdot \Vert _{\alpha
}}} ,\mathbf{V}_{\alpha_{\Vert \cdot \Vert _{\alpha }}} \right)
\end{equation*}
defined by 
\begin{equation*}
\Pi_{\alpha}(L_{\alpha})=P_{W_{\alpha}^{\min}(\mathbf{v})\oplus
U_{\alpha}^{\min}(\mathbf{v})}L_{\alpha}P_{U_{\alpha}^{\min}(\mathbf{v})
\oplus W_{\alpha}^{\min}(\mathbf{v})}
\end{equation*}
is a projection onto $\mathcal{L}(U_{\alpha}^{\min}(\mathbf{v}%
),W_{\alpha}^{\min}(\mathbf{v})).$
\end{proof}

\bigskip

\begin{example}
Consider the morphism 
\begin{equation*}
f :\mathcal{U}(\mathbf{v}) \subset \mathcal{FT}_{\mathfrak{r}}(\mathbf{V}%
_{D}) \longrightarrow 
\mathop{\mathchoice{\raise-0.22em\hbox{\huge $\times$}} {\raise-0.05em\hbox{\Large $\times$}}{\hbox{\large
$\times$}}{\times}}_{\alpha \in \mathcal{L}(T_D)}\mathcal{L}%
(V_{\alpha_{\|\cdot\|_{\alpha}}},V_{\alpha_{\|\cdot\|_{\alpha}}}) \times 
\mathbb{R}^{\mathfrak{r}}
\end{equation*}
defined locally for each $\mathbf{v} \in \mathcal{FT}_{\mathfrak{r}}(\mathbf{%
V}_{D})$ by $f(\mathbf{w})= (\Theta_{\mathbf{v}} \circ \chi_{\mathfrak{r}}(%
\mathbf{v}))(\mathbf{w}) = (\mathfrak{L},\mathfrak{C}).$ Then in local
coordinates we have that $f$ is the identity map. Clearly, $f$ is injective
and 
\begin{equation*}
\mathrm{T}_{\mathbf{v}}f(%
\mathop{\mathchoice{\raise-0.22em\hbox{\huge $\times$}} {\raise-0.05em\hbox{\Large $\times$}}{\hbox{\large
 $\times$}}{\times}}_{\alpha \in \mathcal{L}(T_D)}\mathcal{L}(U_{\alpha
}^{\min }( \mathbf{v}),W_{\alpha }^{\min }( \mathbf{v})) \times \mathbb{R}^{%
\mathfrak{r}}) = 
\mathop{\mathchoice{\raise-0.22em\hbox{\huge $\times$}} {\raise-0.05em\hbox{\Large $\times$}}{\hbox{\large
 $\times$}}{\times}}_{\alpha \in \mathcal{L}(T_D)}\mathcal{L}(U_{\alpha
}^{\min }( \mathbf{v}),W_{\alpha }^{\min }( \mathbf{v})) \times \mathbb{R}^{%
\mathfrak{r}}.
\end{equation*}
From Lemma~\ref{projection_manifold} we have that 
\begin{equation*}
\mathop{\mathchoice{\raise-0.22em\hbox{\huge $\times$}} {\raise-0.05em\hbox{\Large $\times$}}{\hbox{\large
$\times$}}{\times}}_{\alpha \in \mathcal{L}(T_D)}\mathcal{L}(U_{\alpha
}^{\min }(\mathbf{v}),W_{\alpha }^{\min }(\mathbf{v})) \in \mathbb{G}(%
\mathop{\mathchoice{\raise-0.22em\hbox{\huge $\times$}} {\raise-0.05em\hbox{\Large $\times$}}{\hbox{\large
$\times$}}{\times}}_{\alpha \in \mathcal{L}(T_D)}\mathcal{L}%
(V_{\alpha_{\|\cdot\|_{\alpha}}},V_{\alpha_{\|\cdot\|_{\alpha}}}))
\end{equation*}
and hence 
\begin{equation*}
\mathop{\mathchoice{\raise-0.22em\hbox{\huge $\times$}} {\raise-0.05em\hbox{\Large $\times$}}{\hbox{\large
$\times$}}{\times}}_{\alpha \in \mathcal{L}(T_D)}\mathcal{L}(U_{\alpha
}^{\min }(\mathbf{v}),W_{\alpha }^{\min }(\mathbf{v})) \times \mathbb{R}^{%
\mathfrak{r}} \in \mathbb{G}(%
\mathop{\mathchoice{\raise-0.22em\hbox{\huge $\times$}} {\raise-0.05em\hbox{\Large $\times$}}{\hbox{\large
$\times$}}{\times}}_{\alpha \in \mathcal{L}(T_D)}\mathcal{L}%
(V_{\alpha_{\|\cdot\|_{\alpha}}},V_{\alpha_{\|\cdot\|_{\alpha}}}) \times 
\mathbb{R}^{\mathfrak{r}}).
\end{equation*}
Then by Proposition~\ref{prop_inmersion} $f$ is an immersion. Moreover, $f(%
\mathcal{U}(\mathbf{v}))$ with the topology induced by 
\begin{equation*}
\mathop{\mathchoice{\raise-0.22em\hbox{\huge $\times$}} {\raise-0.05em\hbox{\Large $\times$}}{\hbox{\large
 $\times$}}{\times}}_{\alpha \in \mathcal{L}(T_D)}\mathcal{L}%
(V_{\alpha_{\|\cdot\|_{\alpha}}},V_{\alpha_{\|\cdot\|_{\alpha}}}) \times 
\mathbb{R}^{\mathfrak{r}}
\end{equation*}
is homeomorphic to $\mathcal{U}(\mathbf{v})$ when we consider in $\mathcal{U}%
(\mathbf{v})$ the initial topology induced by $f.$ We point out that we can
consider $\{\mathcal{U}(\mathbf{v}): \mathbf{\mathbf{v}} \in \mathcal{FT}_{%
\mathfrak{r}}(\mathbf{V}_{D})\}$ as a basis for a topology in $\mathcal{FT}_{%
\mathfrak{r}}(\mathbf{V}_{D}).$ Then, $f$ is an embedding and $f(\mathcal{FT}%
_{\mathfrak{r}}(\mathbf{V}_{D}))$ is an embedded submanifold of $%
\mathop{\mathchoice{\raise-0.22em\hbox{\huge $\times$}} {\raise-0.05em\hbox{\Large $\times$}}{\hbox{\large
 $\times$}}{\times}}_{\alpha \in \mathcal{L}(T_D)}\mathcal{L}(\mathbf{V}%
_{\alpha_{\|\cdot\|_{\alpha}}},\mathbf{V}_{\alpha_{\|\cdot\|_{\alpha}}})
\times \mathbb{R}^{\mathfrak{r}}.$
\end{example}

From the above example we have that even the manifold $\mathcal{FT}_{%
\mathfrak{r}}(\mathbf{V}_{D})$ is a subset of $\mathbf{V}_{D_{\Vert \cdot
\Vert _{D}}}$ its geometric structure is fully compatible with the topology
of the Banach space $%
\mathop{\mathchoice{\raise-0.22em\hbox{\huge $\times$}} {\raise-0.05em\hbox{\Large $\times$}}{\hbox{\large
 $\times$}}{\times}}_{\alpha \in S(\beta )}\mathcal{L}(\mathbf{V}_{\alpha
_{\Vert \cdot \Vert _{\alpha }}},\mathbf{V}_{\alpha _{\Vert \cdot \Vert
_{\alpha }}})\times \mathbb{R}^{\mathfrak{r}}.$ Moreover, it is not
difficult to see that the same argument runs for the manifold $\mathcal{FT}%
_{\leq \mathfrak{r}}(\mathbf{V}_{D}).$

\bigskip

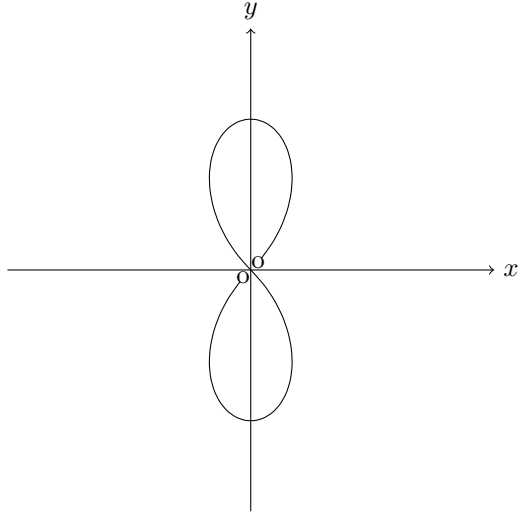
\begin{figure}[tbp]
\begin{center}
\begin{tikzpicture}
\draw[->] (-3.2,0) -- (3.2,0) node[right] {$x$};
\draw[->] (0,-3.2) -- (0,3.2) node[above] {$y$};
\node (a) at (-0.1,-0.1) {o};
\node (a) at (0.1,0.1) {o};
\draw[scale=2,domain=2.356:3.875,smooth,variable=\t]
plot ({cos(2*\t r)*sin(\t r)},{cos(2*\t r)*cos(\t r)})
plot ({-cos(2*\t r)*sin(\t r)},{-cos(2*\t r)*cos(\t r)});
\end{tikzpicture}
\end{center}
\caption{The set $f(3\protect\pi/4,7\protect\pi/4)$ in $\mathbb{R}^2.$ The
"o" means that the lines approach without touch.}\label{eight}
\end{figure}

In consequence, to prove that the standard inclusion map $\mathfrak{i}$ is
an immersion we shall prove, under the appropriate conditions, that if $%
\mathfrak{i}$ is a differentiable morphism then for each $\mathbf{v} \in 
\mathcal{FT}_{\mathfrak{r}}(\mathbf{V}_{D})$ the linear map $\mathrm{T}_{%
\mathbf{v}}\mathfrak{i}$ is injective and $\mathrm{T}_{\mathbf{v}}\mathfrak{i%
}(\mathbb{T}_{\mathbf{v}}(\mathcal{FT}_{\mathfrak{r}}(\mathbf{V}_{D})))$
belongs to $\mathbb{G}(\mathbf{V}_{D_{\|\cdot\|_D}}).$

\subsubsection{The derivative as a morphism of the standard inclusion map}

To describe $\mathfrak{i}$ as a morphism, we proceed as follows. Given $%
\mathbf{v}\in \mathcal{FT}_{\mathfrak{r}}(\mathbf{V}_{D}),$ we consider $%
\mathcal{U}(\mathbf{v}),$ a local neighbourhood of $\mathbf{v},$ and then 
\begin{equation*}
(\mathfrak{i}\circ \Theta _{\mathbf{v}}^{-1}\circ \chi _{\mathfrak{r}}^{-1}(%
\mathbf{v})):%
\mathop{\mathchoice{\raise-0.22em\hbox{\huge $\times$}} {\raise-0.05em\hbox{\Large $\times$}}{\hbox{\large
$\times$}}{\times}}_{\alpha \in \mathcal{L}(T_{D})}\mathcal{L}(U_{\alpha
}^{\min }(\mathbf{v}),W_{\alpha }^{\min }(\mathbf{v}))\times \mathbb{R}%
_{\ast }^{\mathfrak{r}}\rightarrow \mathbf{V}_{\Vert \cdot \Vert _{D}}
\end{equation*}%
is given by 
\begin{equation*}
\left( \mathfrak{L},\mathfrak{C}\right) \mapsto \sum_{\substack{ 1\leq
i_{k}\leq r_{k}  \\ k\in \mathcal{L}(T_{D})}}\left( \sum_{\substack{ 1\leq
i_{\alpha }\leq r_{\alpha }  \\ \alpha \in T_{D}\setminus \{D\}  \\ \alpha
\notin \mathcal{L}(T_{D})}}C_{(i_{\alpha })_{\alpha \in S(D)}}^{(D)}\prod 
_{\substack{ \mu \in T_{D}\setminus \{D\}  \\ S(\mu )\neq \emptyset }}%
C_{i_{\mu },(i_{\beta })_{\beta \in S(\mu )}}^{(\mu )}\right)
\bigotimes_{k\in \mathcal{L}(T_{D})}(u_{i_{k}}^{(k)}+L_{k}(u_{i_{k}}^{(k)})),
\end{equation*}%
that is, 
\begin{equation*}
(\mathfrak{i}\circ \Theta _{\mathbf{v}}^{-1}\circ \chi _{\mathfrak{r}}^{-1}(%
\mathbf{v}))\left( \mathfrak{L},\mathfrak{C}\right) =\mathbf{w}=\sum 
_{\substack{ 1\leq i_{\alpha }\leq r_{\alpha }  \\ \alpha \in S(D)}}%
C_{(i_{\alpha })_{\alpha \in S(D)}}^{(D)}\bigotimes_{\alpha \in S(D)}\mathbf{%
w}_{i_{\alpha }}^{(\alpha )}
\end{equation*}%
where for each $\mu \in T_{D}\setminus \{D\}$ we write 
\begin{equation*}
\mathbf{w}_{i_{\mu }}^{(\mu )}=\left\{ 
\begin{array}{lc}
(id+L_{\mu })(\mathbf{u}_{i_{\mu }}^{(\mu }) & \text{ if }\mu \in \mathcal{L}%
(T_{D}) \\ 
&  \\ 
\sum_{\substack{ 1\leq i_{\beta }\leq r_{\beta }  \\ \beta \in S(\mu )}}%
C_{i_{\mu },(i_{\beta })_{\beta \in S(\mu )}}^{(\mu )}\bigotimes_{\beta \in
S(\mu )}\mathbf{w}_{i_{\beta }}^{(\beta )} & \text{ otherwise, }%
\end{array}%
\right.
\end{equation*}%
for $1\leq i_{\mu }\leq r_{\mu }.$

\bigskip

Assume that $(\mathfrak{i}\circ \Theta _{\mathbf{v}}^{-1} \circ \chi_{%
\mathfrak{r}}^{-1}(\mathbf{v}))$ is Fr\'{e}chet differentiable, then 
\begin{equation*}
\mathrm{T}_{\mathbf{v}}\mathfrak{i}:%
\mathop{\mathchoice{\raise-0.22em\hbox{\huge $\times$}} {\raise-0.05em\hbox{\Large $\times$}}{\hbox{\large
$\times$}}{\times}}_{\alpha \in \mathcal{L}(T_D)}\mathcal{L}(U_{\alpha
}^{\min }(\mathbf{v}),W_{\alpha }^{\min }(\mathbf{v})) \times \mathbb{R}%
_{\ast }^{\mathfrak{r}}\longrightarrow \mathbf{V}_{\Vert \cdot \Vert _{D}}
\end{equation*}
is given by 
\begin{equation*}
\mathrm{T}_{\mathbf{v}}\mathfrak{i}(\dot{\mathfrak{L}},\dot{\mathfrak{C}})=[(%
\mathfrak{i}\circ \Theta _{\mathbf{v}}^{-1}\circ \chi_{\mathfrak{r}}^{-1}(%
\mathbf{v}))^{\prime }((\Theta _{\mathbf{v}} \circ \chi_{\mathfrak{r}}(%
\mathbf{v}))(\mathbf{v}))](\dot{\mathfrak{L}},\dot{\mathfrak{C}}) = [(%
\mathfrak{i}\circ \Theta _{\mathbf{v}}^{-1}\circ \chi_{\mathfrak{r}}^{-1}(%
\mathbf{v}))^{\prime }(\mathfrak{0},\mathfrak{C})](\dot{\mathfrak{L}},\dot{%
\mathfrak{C}}),
\end{equation*}%
where $(\Theta _{\mathbf{v}} \circ \chi_{\mathfrak{r}}(\mathbf{v}))(\mathbf{v%
}) = (\mathfrak{0},\mathfrak{C}),$ because $\Psi_{\mathbf{v}}((U_k^{\min}(%
\mathbf{v}))_{k \in \mathcal{L}(T_D)}) = (0)_{k \in \mathcal{L}(T_D)} = 
\mathfrak{0}.$

\bigskip

The next lemma describes the tangent map $\mathrm{T}_{\mathbf{v}}\mathfrak{i}%
.$

\begin{proposition}
\label{characterization_tangent_map} Assume that the tensor product map $%
\bigotimes$ is $T_D$-continuous. Let $\mathbf{v}\in \mathcal{FT}_{\mathfrak{r%
}}(\mathbf{V}_{D})$ be such that $\Theta _{\mathbf{v}}(\mathbf{v})=(%
\mathfrak{0},\mathfrak{C}(\mathbf{v})),$ where $\mathfrak{C}(\mathbf{v}%
)=(C^{(\alpha )})_{\alpha \in T_{D}\setminus \mathcal{L}(T_{D})}\in \mathbb{R%
}^{\mathfrak{r}},$ $\mathfrak{0}=(0)_{\alpha \in \mathcal{L}(T_D)} \in 
\mathop{\mathchoice{\raise-0.22em\hbox{\huge $\times$}} {\raise-0.05em\hbox{\Large $\times$}}{\hbox{\large
$\times$}}{\times}}_{\alpha \in \mathcal{L}(T_D)}\mathcal{L}(U_{\alpha
}^{\min }(\mathbf{v}),W_{\alpha }^{\min }(\mathbf{v}))$ and 
\begin{equation*}
U_{\alpha }^{\min }(\mathbf{v})=\mathrm{span}\,\{\mathbf{u}_{i_{\alpha
}}^{(\alpha )}:1\leq i_{\alpha }\leq r_{\alpha }\}
\end{equation*}%
for $\alpha \in T_{D}\setminus \{D\}.$ Then the following statements hold.

\begin{itemize}
\item[(a)] The map $(\mathfrak{i} \circ \Theta_{\mathbf{v}}^{-1} \circ \chi_{%
\mathfrak{r}}(\mathbf{v}))$ from $%
\mathop{\mathchoice{\raise-0.22em\hbox{\huge $\times$}} {\raise-0.05em\hbox{\Large $\times$}}{\hbox{\large
$\times$}}{\times}}_{\alpha \in \mathcal{L}(T_D)}\mathcal{L}(U_{\alpha
}^{\min }(\mathbf{v}),W_{\alpha }^{\min }(\mathbf{v})) \times \mathbb{R}^{%
\mathfrak{r}}_*$ to $\mathbf{V}_{D_{\|\cdot\|_D}}$ is Fr\'echet
differentiable, and hence 
\begin{equation*}
\mathrm{T}_{\mathbf{v}}\mathfrak{i} \in \mathcal{L}\left(\mathbb{T}_{\mathbf{%
v}}(\mathcal{FT}_{\mathfrak{r}}(\mathbf{V}_{D})), \mathbf{V}%
_{D_{\|\cdot\|_D}} \right).
\end{equation*}

\item[(b)] Assume $(\dot{\mathfrak{L}},\dot{\mathfrak{C}})\in \mathbb{T}_{%
\mathbf{v}}(\mathcal{FT}_{\mathfrak{r}}(\mathbf{V}_{D})),$ where $\dot{%
\mathfrak{C}}=(\dot{C}^{(\alpha )})_{\alpha \in T_{D}\setminus \mathcal{L}%
(T_{D})}\in \mathbb{R}^{\mathfrak{r}}$ and $\dot{\mathfrak{L}}=(\dot{L}%
_{\alpha })_{\alpha \in \mathcal{L}(T_{D})}$ is in $ 
\mathop{\mathchoice{\raise-0.22em\hbox{\huge $\times$}} {\raise-0.05em\hbox{\Large $\times$}}{\hbox{\large
$\times$}}{\times}}_{\alpha \in \mathcal{L}(T_{D})}\mathcal{L}(U_{\alpha
}^{\min }(\mathbf{v}),W_{\alpha }^{\min }(\mathbf{v})).$ Then $\dot{\mathbf{w%
}}=\mathrm{T}_{\mathbf{v}}\mathfrak{i}(\dot{\mathfrak{C}},\dot{\mathfrak{L}}%
) $ if and only 
\begin{equation}
\dot{\mathbf{w}}=\sum_{\substack{ 1\leq i_{\alpha }\leq r_{\alpha }  \\ %
\alpha \in S(D)}}\dot{C}_{(i_{\alpha })_{\alpha \in
S(D)}}^{(D)}\bigotimes_{\alpha \in S(D)}\mathbf{u}_{i_{\alpha }}^{(\alpha
)}+\sum_{\alpha \in S(D)}\sum_{\substack{ 1\leq i_{\alpha }\leq r_{\alpha }}}%
\left( \dot{\mathbf{u}}_{i_{\alpha }}^{(\alpha )}\otimes \mathbf{U}%
_{i_{\alpha }}^{(\alpha )}\right) ,  \label{kinematic1}
\end{equation}%
where 
\begin{equation}
\mathbf{U}_{i_{\alpha }}^{(\alpha )}=\sum_{\substack{ 1\leq i_{\beta }\leq
r_{\beta }  \\ \beta \in S(D)  \\ \beta \neq \alpha }}C_{(i_{\beta })_{\beta
\in S(D)}}^{(D)}\bigotimes_{\beta \in S(D)}\mathbf{u}_{i_{\beta }}^{(\beta
)},  \label{Ualpha ialpha}
\end{equation}%
and for each $\gamma \in T_{D}\setminus \{D\}$ we have 
\begin{equation*}
\dot{\mathbf{u}}_{i_{\gamma }}^{(\gamma )}=\left\{ 
\begin{array}{lcc}
\dot{L}_{\mu }(\mathbf{u}_{i_{\gamma }}^{(\gamma )}) & \text{ if } & \gamma
\in \mathcal{L}(T_{D}) \\ 
&  &  \\ 
\sum_{\substack{ 1\leq i_{\beta }\leq r_{\beta }  \\ \beta \in S(\gamma )}}%
\dot{C}_{i_{\gamma },(i_{\beta })_{\beta \in S(\gamma )}}^{(\gamma
)}\bigotimes_{\beta \in S(\gamma )}\mathbf{u}_{i_{\beta }}^{(\beta
)}+\sum_{\beta \in S(\gamma )}\sum_{\substack{ 1\leq i_{\beta }\leq r_{\beta
}}}\left( \dot{\mathbf{u}}_{i_{\beta }}^{(\beta )}\otimes \mathbf{U}%
_{i_{\gamma },i_{\beta }}^{(\beta )}\right) &  & \text{ otherwise,}%
\end{array}%
\right.
\end{equation*}
where 
\begin{equation}
\mathbf{U}_{i_{\gamma },i_{\beta }}^{(\beta )}=\sum_{\substack{ 1\leq
i_{\delta }\leq r_{\delta }  \\ \delta \in S(\gamma )  \\ \delta \neq \beta }}%
C_{i_{\mu },(i_{\delta })_{\delta \in S(\gamma )}}^{(\gamma )}\bigotimes 
_{\substack{ \delta \neq \beta  \\ \delta \in S(\gamma )}}\mathbf{u}%
_{i_{\delta }}^{(\delta )},  \label{Ubeta igamma ibeta}
\end{equation}%
for $1\leq i_{\gamma }\leq r_{\gamma }$ and $1\leq i_{\beta }\leq r_{\beta
}. $
\end{itemize}
\end{proposition}

\begin{proof}
To prove statement (a), observe that for each $\mathbf{u}_{\alpha }\in
U_{\alpha }^{\min }(\mathbf{v}),$ $\alpha \in \mathcal{L}(T_{D}),$ the map 
\begin{equation*}
\Phi _{\mathbf{u}_{\alpha }}:\mathcal{L}(U_{\alpha }^{\min }(\mathbf{v}%
),W_{\alpha }^{\min }(\mathbf{v}))\rightarrow W_{\alpha }^{\min }(\mathbf{v}%
),\quad L_{\alpha }\mapsto L_{\alpha }(\mathbf{u}_{\alpha }),
\end{equation*}%
is linear and continuous, and hence Fr\'{e}chet differentiable. Clearly, its
differential is given by 
\begin{equation*}
\lbrack \Phi _{\mathbf{u}_{\alpha }}^{\prime }(L_{\alpha })](H_{\alpha
})=H_{\alpha }(\mathbf{u}_{\alpha }).
\end{equation*}%
Also, if the tensor product map $\bigotimes $ is $T_{D}$-continuous, by
Proposition~\ref{Diff_otimes}, the tensor product map 
\begin{equation*}
\bigotimes :%
\mathop{\mathchoice{\raise-0.22em\hbox{\huge
 $\times$}} {\raise-0.05em\hbox{\Large $\times$}}{\hbox{\large
 $\times$}}{\times}}_{\beta \in S(\gamma )}(V_{\beta _{\Vert \cdot \Vert
_{\beta }}},\Vert \cdot \Vert _{\beta })\rightarrow \left( \left.
_{a}\bigotimes_{\beta \in S(\gamma )}V_{\beta _{\Vert \cdot \Vert _{\beta
}}}\right. ,\Vert \cdot \Vert _{\gamma }\right) ,
\end{equation*}%
for $\gamma \in T_{D}\setminus \mathcal{L}(T_{D}),$ is also Fr\'{e}chet
differentiable. Then, by the chain rule, the map $\Theta _{\mathbf{v}}^{-1}$
is Fr\'{e}chet differentiable. Since $\mathrm{T}_{\mathbf{v}}\mathfrak{i}(%
\dot{\mathfrak{C}},\dot{\mathfrak{L}})=[(\mathfrak{i}\circ \Theta _{\mathbf{v%
}}^{-1}\circ \chi _{\mathfrak{r}}^{-1}(\mathbf{v}))^{\prime }(\mathfrak{C},%
\mathfrak{0})](\dot{\mathfrak{C}},\dot{\mathfrak{L}}),$ (a) follows. Using
the chain rule, we obtain (b).
\end{proof}

\bigskip

Let $\mathbf{v}\in \mathcal{FT}_{\mathfrak{r}}(\mathbf{V}_{D})\subset 
\mathbf{V}_{D_{\Vert \cdot \Vert _{D}}}$ be such that 
\begin{equation*}
\mathbf{v}=\sum_{\substack{ 1\leq i_{\alpha }\leq r_{\alpha }  \\ \alpha \in
S(D)}}C_{(i_{\alpha })_{\alpha \in S(D)}}^{(D)}\bigotimes_{\alpha \in S(D)}%
\mathbf{u}_{i_{\alpha }}^{(\alpha )},
\end{equation*}%
where for each $\mu \in T_{D}\setminus (\{D\}\cup \mathcal{L}(T_{D}))$ we
have 
\begin{equation*}
\mathbf{u}_{i_{\mu }}^{(\mu )}=\sum_{\substack{ 1\leq i_{\beta }\leq
r_{\beta }  \\ \beta \in S(\mu )}}C_{i_{\mu },(i_{\beta })_{\beta \in
S(\mu )}}^{(\mu )}\bigotimes_{\beta \in S(\mu )}\mathbf{u}_{i_{\beta
}}^{(\beta )}
\end{equation*}%
for $1\leq i_{\mu }\leq r_{\mu }.$ Recall that for $\alpha \in S(D)$ we have 
\begin{equation*}
U_{S(D)\setminus \{\alpha \}}^{\min }(\mathbf{v})=%
\mathrm{span}\,\{\mathbf{U}_{i_{\alpha }}^{(\alpha )}:1\leq i_{\alpha }\leq
r_{\alpha }\},
\end{equation*}%
and for $\mu \in T_{D}\setminus (\{D\}\cup \mathcal{L}(T_{D}))$ we know that 
$U_{\beta }^{\min }(\mathbf{u}_{i_{\mu }}^{(\mu )})=U_{\beta }^{\min }(%
\mathbf{v})$ and 
\begin{equation*}
U_{S(\mu )\setminus \{\beta \}}^{\min }(\mathbf{u}_{i_{\mu }}^{(\mu )})=%
\mathrm{span}\,\{\mathbf{U}_{i_{\mu },i_{\beta }}^{(\beta )}:1\leq i_{\beta
}\leq r_{\beta }\}
\end{equation*}%
for $1\leq i_{\mu }\leq r_{\mu }$ and $\beta \in S(\mu ).$ Hence 
\begin{equation*}
W_{\beta }^{\min }(\mathbf{v})=W_{\beta }^{\min }(\mathbf{u}_{i_{\mu
}}^{(\mu )})\text{ for }1\leq i_{\mu }\leq r_{\mu }\text{ and }\beta \in
S(\mu ).
\end{equation*}%
In the next proposition we prove that $\mathrm{T}_{\mathbf{v}}\mathfrak{i}$
injective when we consider $\mathbf{v}$ in the manifold $\mathcal{M}_{%
\mathbf{r}}(\mathbf{V}_{D}).$ It allows us to characterise the tangent space
for Tucker tensors inside the tensor space $\mathbf{V}_{D_{\Vert \cdot \Vert
_{D}}}.$

\begin{proposition}
\label{rank_one_tangent_space} Assume that $S(D)=\mathcal{L}(T_D)$ and the
tensor product map $\bigotimes$ is $T_D$-continuous. Let $\mathbf{v}\in 
\mathcal{M}_{\mathbf{r}}(\mathbf{V}_{D}),$ then the linear map $\mathrm{T}_{%
\mathbf{v}}\mathfrak{i}$ is injective and 
\begin{equation*}
\mathrm{T}_{\mathbf{v}}\mathfrak{i}(\mathbb{T}_{\mathbf{v}}(\mathcal{M}_{%
\mathbf{r}}(\mathbf{V}_{D}))) = \left._a \bigotimes_{\alpha \in S(D)}
U_{\alpha}^{\min}(\mathbf{v})\right. \oplus \left( \bigoplus_{\alpha \in
S(D)} W_{\alpha}^{\min}(\mathbf{v})\otimes_a U_{S(D)\setminus
\{\alpha\}}^{\min}(\mathbf{v}) \right)
\end{equation*}
is linearly isomorphic to $\mathbb{T}_{\mathbf{v}}(\mathcal{M}_{\mathbf{r}}(%
\mathbf{V}_{D})).$
\end{proposition}

\begin{proof}
First, observe that if $\mathbf{v}\in \mathcal{M}_{\mathbf{r}}(\mathbf{V}%
_{D})$ and $\dot{\mathbf{w}}=\mathrm{T}_{\mathbf{v}}\mathfrak{i}(\dot{%
\mathfrak{C}},\dot{\mathfrak{L}}),$ then by Proposition~\ref%
{characterization_tangent_map}(b)%
\begin{equation*}
\dot{\mathbf{w}}=\sum_{\substack{ 1\leq i_{\alpha }\leq r_{\alpha }  \\ %
\alpha \in S(D)}}\dot{C}_{(i_{\alpha })_{\alpha \in
S(D)}}^{(D)}\bigotimes_{\alpha \in S(D)}\mathbf{u}_{i_{\alpha }}^{(\alpha
)}+\sum_{\alpha \in S(D)}\sum_{\substack{ 1\leq i_{\alpha }\leq r_{\alpha }}}%
\left( \dot{\mathbf{u}}_{i_{\alpha }}^{(\alpha )}\otimes \mathbf{U}%
_{i_{\alpha }}^{(\alpha )}\right) ,
\end{equation*}%
where 
\begin{equation*}
\mathbf{U}_{i_{\alpha }}^{(\alpha )}=\sum_{\substack{ 1\leq i_{\beta }\leq
r_{\beta }  \\ \beta \in S(D)  \\ \beta \neq \alpha }}C_{(i_{\beta })_{\beta
\in S(D)}}^{(D)}\bigotimes_{\beta \in S(D)}\mathbf{u}_{i_{\beta }}^{(\beta
)}\in U_{S(D)\setminus \{\alpha \}}^{\min }(\mathbf{v}),
\end{equation*}%
and $\dot{\mathbf{u}}_{i_{\alpha }}^{(\alpha )}=\dot{L}(\mathbf{u}%
_{i_{\alpha }}^{(\alpha ))})\in W_{\alpha }^{\min }(\mathbf{v})$ for all $%
\alpha \in \mathcal{L}(T_{D}).$ Hence $\mathrm{T}_{\mathbf{v}}\mathfrak{i}(%
\mathbb{T}_{\mathbf{v}}(\mathcal{M}_{\mathbf{r}}(\mathbf{V}_{D})))\subset 
\mathbf{Z}^{(D)}(\mathbf{v})$ where 
\begin{equation*}
\mathbf{Z}^{(D)}(\mathbf{v})=\left. _{a}\bigotimes_{\alpha \in
S(D)}U_{\alpha }^{\min }(\mathbf{v})\right. \oplus \left( \bigoplus_{\alpha
\in S(D)}W_{\alpha }^{\min }(\mathbf{v})\otimes _{a}U_{S(D)\setminus
\{\alpha \}}^{\min }(\mathbf{v})\right) .
\end{equation*}%
Next, we claim that $\mathbf{Z}^{(D)}(\mathbf{v})\subset \mathrm{T}_{\mathbf{%
v}}\mathfrak{i}(\mathbb{T}_{\mathbf{v}}(\mathcal{M}_{\mathbf{r}}(\mathbf{V}%
_{D}))).$ To prove the claim take $\mathbf{w}\in \mathbf{Z}^{(D)}(\mathbf{v}%
).$ Then we can write 
\begin{equation*}
\mathbf{w}=\sum_{\substack{ 1\leq i_{\alpha }\leq r_{\alpha }  \\ \alpha \in
S(D)}}(\dot{C}^{(D)})_{(i_{\alpha })_{\alpha \in S(D)}}\bigotimes_{\alpha
\in S(D)}\mathbf{u}_{i_{\alpha }}^{(\alpha )}+\sum_{\alpha \in S(D)}\sum 
_{\substack{ 1\leq i_{\alpha }\leq r_{\alpha }}}\left( \mathbf{w}_{i_{\alpha
}}^{(\alpha )}\otimes \mathbf{U}_{i_{\alpha }}^{(\alpha )}\right) ,
\end{equation*}%
where $\mathbf{w}_{i_{\alpha }}^{(\alpha )}=W_{\alpha }^{\min }(\mathbf{v})$
for $1\leq i_{\alpha }\leq r_{\alpha }$ and $\alpha \in S(D).$ Recall that 
\begin{equation*}
U_{S(D)\setminus \{\alpha \}}^{\min }(\mathbf{v})=\mathrm{span}\,\{\mathbf{U}%
_{i_{\alpha }}^{(\alpha )}:1\leq i_{\alpha }\leq r_{\alpha }\}.
\end{equation*}%
Now, define $\dot{L}_{\alpha }\in \mathcal{L}(U_{\alpha }^{\min }(\mathbf{v}%
),W_{\alpha }^{\min }(\mathbf{v}))$ by $\dot{L}_{\alpha }(\mathbf{u}%
_{i_{\alpha }}^{(\alpha )}):=\mathbf{w}_{i_{\alpha }}^{(\alpha )}$ for $%
1\leq i_{\alpha }\leq r_{\alpha }$ and $\alpha \in S(D).$ Then the claim
follows from $\mathbf{w}=\mathrm{T}_{\mathbf{v}}\mathfrak{i}((\dot{L}%
_{\alpha })_{\alpha \in S(D)},\dot{C}^{(D)}).$ To conclude the proof of the
proposition we need to show that the map $\mathrm{T}_{\mathbf{v}}\mathfrak{i}
$ is an injective linear operator. To prove this consider that 
\begin{equation*}
\mathrm{T}_{\mathbf{v}}\mathfrak{i}\left( (\dot{L}_{\beta })_{\beta \in 
\mathcal{L}(T_{D})},\dot{C}^{(D)}\right) =\mathbf{0},
\end{equation*}%
that is, 
\begin{equation*}
\mathbf{0}=\sum_{\substack{ 1\leq i_{\alpha }\leq r_{\alpha }  \\ \alpha \in
S(D)}}(\dot{C}^{(D)})_{(i_{\alpha })_{\alpha \in S(D)}}\bigotimes_{\alpha
\in S(D)}\mathbf{u}_{i_{\alpha }}^{(\alpha )}+\sum_{\alpha \in S(D)}\sum 
_{\substack{ 1\leq i_{\alpha }\leq r_{\alpha }}}\left( \dot{\mathbf{u}}%
_{i_{\alpha }}^{(\alpha )}\otimes \mathbf{U}_{i_{\alpha }}^{(\alpha
)}\right) .
\end{equation*}%
Thus, 
\begin{align*}
\sum_{\substack{ 1\leq i_{\alpha }\leq r_{\alpha }  \\ \alpha \in S(D)}}(%
\dot{C}^{(D)})_{(i_{\alpha })_{\alpha \in S(D)}}\bigotimes_{\alpha \in S(D)}%
\mathbf{u}_{i_{\alpha }}^{(\alpha )}& =\mathbf{0}, \\
\sum_{\substack{ 1\leq i_{\alpha }\leq r_{\alpha }}}\left( \dot{\mathbf{u}}%
_{i_{\alpha }}^{(\alpha )}\otimes \mathbf{U}_{i_{\alpha }}^{(\alpha
)}\right) & =\mathbf{0}\text{ for }\alpha \in S(D),
\end{align*}%
and hence $\dot{C}^{(D)}=\mathfrak{0},$ because $\left\{ \bigotimes_{\alpha
\in S(D)}\mathbf{u}_{i_{\alpha }}^{(\alpha )}\right\} $ is a basis of $%
\left. _{a}\bigotimes_{\alpha \in S(D)}U_{\alpha }^{\min }(\mathbf{v}%
)\right. ,$ and $\dot{L}_{\alpha }(\mathbf{u}_{i_{\alpha }}^{(\alpha
)})\otimes \mathbf{U}_{i_{\alpha }}^{(\alpha )}=\mathbf{0}$ for $1\leq
i_{\alpha }\leq r_{\alpha },$ because the $\{\mathbf{U}_{i_{\alpha
}}^{(\alpha )}:1\leq i_{\alpha }\leq r_{\alpha }\}$ are linearly independent
for $\alpha \in S(D).$ Then $\dot{L}_{\alpha }=0$ for all $\alpha \in S(D)$.
We conclude that

\begin{equation*}
\left( (\dot{L}_{\beta })_{\beta \in \mathcal{L}(T_{D})},\dot{C}%
^{(D)}\right) =((0)_{\beta \in \mathcal{L}(T_{D})},\mathfrak{0})
\end{equation*}%
and, in consequence, $\mathrm{T}_{\mathbf{v}}\mathfrak{i}$ is injective.
\end{proof}

\bigskip

Our next step is to show, by using the above proposition, that if the tensor
product map $\bigotimes$ is $T_D$-continuous then the linear map $\mathrm{T}%
_{\mathbf{v}}\mathfrak{i}$ is always injective for all $\mathbf{v}\in 
\mathcal{FT}_{\mathfrak{r}}(\mathbf{V}_D).$

\begin{proposition}
Assume that the tensor product map $\bigotimes $ is $T_{D}$-continuous. Let $%
\mathbf{v}\in \mathcal{FT}_{\mathfrak{r}}(\mathbf{V}_{D}),$ then the linear
map $\mathrm{T}_{\mathbf{v}}\mathfrak{i}:%
\mathop{\mathchoice{\raise-0.22em\hbox{\huge
$\times$}} {\raise-0.05em\hbox{\Large $\times$}}{\hbox{\large
$\times$}}{\times}}_{\beta \in \mathcal{L}(T_{D})}\mathcal{L}(U_{\beta
}^{\min }(\mathbf{v}),W_{\beta }^{\min }(\mathbf{v}))\times \mathbb{R}^{%
\mathfrak{r}}\rightarrow \mathbf{V}_{D_{\Vert \cdot \Vert _{D}}}$ is
injective.
\end{proposition}

\begin{proof}
From Proposition~\ref{rank_one_tangent_space} the statement holds when $S(D)=%
\mathcal{L}(T_{D}).$ Thus assume that $S(D)\neq \mathcal{L}(T_{D}).$ Then we
can write the standard inclusion map $\mathfrak{i}:\mathcal{FT}_{\mathfrak{r}%
}(\mathbf{V}_{D})\longrightarrow \mathbf{V}_{D_{\Vert \cdot \Vert _{D}}}$ as 
$\mathfrak{i}_{D}\circ \mathfrak{i}_{\mathfrak{r},D}$ where 
\begin{equation*}
\mathfrak{i}_{\mathfrak{r},D}:\mathcal{FT}_{\mathfrak{r}}(\mathbf{V}%
_{D})\longrightarrow \mathcal{M}_{(r_{\beta })_{\beta \in S(D)}}\left(
\left. _{a}\bigotimes_{\beta \in S(D)}\mathbf{V}_{\beta }\right. \right)
,\quad \mathbf{v}\mapsto \mathbf{v}
\end{equation*}%
is a standard inclusion map and

\begin{equation*}
\mathfrak{i}_{D}:\mathcal{M}_{(r_{\beta })_{\beta \in S(D)}}\left( \left.
_{a}\bigotimes_{\beta \in S(D)}\mathbf{V}_{\beta }\right. \right)
\longrightarrow \mathbf{V}_{D_{\Vert \cdot \Vert _{D}}}
\end{equation*}
is given by 
\begin{equation*}
\mathbf{v}=\mathfrak{i}_{D}(\mathbf{v})=\sum_{\substack{ 1\leq i_{\beta
}\leq r_{\beta }  \\ \beta \in S(\alpha )}}C_{(i_{\beta })_{\beta \in
S(\gamma )}}^{(D)}\bigotimes_{\beta \in S(D)}\mathbf{u}_{i_{\beta }}^{(\beta
)}.
\end{equation*}%
Using the chain rule, we have 
\begin{equation*}
\mathrm{T}_{\mathbf{v}}\mathfrak{i}=\mathrm{T}_{\mathbf{v}}\mathfrak{i}%
_{D}\circ \mathrm{T}_{\mathbf{v}}\mathfrak{i}_{\mathfrak{r},D},
\end{equation*}%
where 
\begin{equation*}
\mathrm{T}_{\mathbf{v}}\mathfrak{i}_{D}:%
\mathop{\mathchoice{\raise-0.22em\hbox{\huge
$\times$}} {\raise-0.05em\hbox{\Large $\times$}}{\hbox{\large
$\times$}}{\times}}_{\beta \in S(D)}\mathcal{L}(U_{\beta }^{\min }(\mathbf{v}%
),W_{\beta }^{\min }(\mathbf{v}))\times \mathbb{R}^{%
\mathop{\mathchoice{\raise-0.22em\hbox{\huge
$\times$}} {\raise-0.05em\hbox{\Large $\times$}}{\hbox{\large
$\times$}}{\times}}_{\beta \in S(D)}r_{\beta }}\rightarrow \mathbf{V}%
_{D_{\Vert \cdot \Vert _{D}}},
\end{equation*}%
is given by 
\begin{equation*}
\mathrm{T}_{\mathbf{v}}\mathfrak{i}_{D}((\dot{L}_{\alpha })_{\alpha \in
S(D)},\dot{C}^{(D)})=\sum_{\substack{ 1\leq i_{\alpha }\leq r_{\alpha }  \\ %
\alpha \in S(D)}}\dot{C}_{(i_{\alpha })_{\alpha \in
S(D)}}^{(D)}\bigotimes_{\alpha \in S(D)}\mathbf{u}_{i_{\alpha }}^{(\alpha
)}+\sum_{\alpha \in S(D)}\sum_{\substack{ 1\leq i_{\alpha }\leq r_{\alpha }}}%
\left( \dot{L}_{\alpha }(\mathbf{u}_{i_{\alpha }}^{(\alpha )})\otimes 
\mathbf{U}_{i_{\alpha }}^{(\alpha )}\right) ,
\end{equation*}%
and 
\begin{equation*}
\mathrm{T}_{\mathbf{v}}\mathfrak{i}_{\mathfrak{r},D}:%
\mathop{\mathchoice{\raise-0.22em\hbox{\huge
$\times$}} {\raise-0.05em\hbox{\Large $\times$}}{\hbox{\large
$\times$}}{\times}}_{\beta \in \mathcal{L}(T_{D})}\mathcal{L}(U_{\beta
}^{\min }(\mathbf{v}),W_{\beta }^{\min }(\mathbf{v}))\times \mathbb{R}^{%
\mathfrak{r}}\rightarrow 
\mathop{\mathchoice{\raise-0.22em\hbox{\huge
$\times$}} {\raise-0.05em\hbox{\Large $\times$}}{\hbox{\large
$\times$}}{\times}}_{\beta \in S(D)}\mathcal{L}(U_{\beta }^{\min }(\mathbf{v}%
),W_{\beta }^{\min }(\mathbf{v}))\times \mathbb{R}^{%
\mathop{\mathchoice{\raise-0.22em\hbox{\huge
$\times$}} {\raise-0.05em\hbox{\Large $\times$}}{\hbox{\large
$\times$}}{\times}}_{\beta \in S(D)}r_{\beta }}
\end{equation*}%
is given by 
\begin{equation*}
\mathrm{T}_{\mathbf{v}}\mathfrak{i}_{\mathfrak{r},D}((\dot{L}_{\beta
})_{\beta \in \mathcal{L}(T_{D})},(\dot{C}^{(\alpha )})_{\alpha \in
T_{D}\setminus \mathcal{L}(T_{D})})=((\dot{S}_{\beta })_{\beta \in S(D)},%
\dot{C}^{(D)}),
\end{equation*}%
where $\dot{S}_{\gamma }=\dot{L}_{\gamma }$ if $\gamma \in \mathcal{L}%
(T_{D}),$ otherwise 
\begin{equation*}
\dot{S}_{\gamma }(\mathbf{u}_{i_{\gamma }}^{(\gamma )})=\sum_{\substack{ %
1\leq i_{\beta }\leq r_{\beta }  \\ \beta \in S(\gamma )}}\dot{C}_{i_{\gamma
},(i_{\beta })_{\beta \in S(\gamma )}}^{(\gamma )}\bigotimes_{\beta \in
S(\gamma )}\mathbf{u}_{i_{\beta }}^{(\beta )}+\sum_{\beta \in S(\gamma
)}\sum_{1\leq i_{\beta }\leq r_{\beta }}\dot{\mathbf{u}}_{i_{\gamma
}}^{(\gamma )}\otimes \mathbf{U}_{i_{\gamma },i_{\beta }}^{(\beta )}
\end{equation*}%
and where for each $\gamma \in T_{D}\setminus \{D\}$ we have 
\begin{equation*}
\dot{\mathbf{u}}_{i_{\gamma }}^{(\gamma )}=\left\{ 
\begin{array}{lcc}
\dot{L}_{\gamma }(\mathbf{u}_{i_{\gamma }}^{(\gamma )}) & \text{ if } & 
\gamma \in \mathcal{L}(T_{D}) \\ 
&  &  \\ 
\sum_{\substack{ 1\leq i_{\beta }\leq r_{\beta }  \\ \beta \in S(\gamma )}}%
\dot{C}_{i_{\gamma },(i_{\beta })_{\beta \in S(\gamma )}}^{(\gamma
)}\bigotimes_{\beta \in S(\gamma )}\mathbf{u}_{i_{\beta }}^{(\beta
)}+\sum_{\beta \in S(\gamma )}\sum_{\substack{ 1\leq i_{\beta }\leq r_{\beta
}}}\left( \dot{\mathbf{u}}_{i_{\beta }}^{(\beta )}\otimes \mathbf{U}%
_{i_{\gamma },i_{\beta }}^{(\beta )}\right) &  & \text{ otherwise. }%
\end{array}%
\right.
\end{equation*}%
Let $\dot{\mathbf{w}}=T_{\mathbf{v}}\mathfrak{i}((\dot{L}_{\beta })_{\beta
\in \mathcal{L}(T_{D})},(\dot{C}^{(\alpha )})_{\alpha \in T_{D}\setminus 
\mathcal{L}(T_{D})})=\mathbf{0}.$ Since $\mathrm{T}_{\mathbf{v}}\mathfrak{i}=%
\mathrm{T}_{\mathbf{v}}\mathfrak{i}_{D}\circ \mathrm{T}_{\mathbf{v}}%
\mathfrak{i}_{\mathfrak{r},D}$ and, by Proposition~\ref%
{rank_one_tangent_space}, the linear map $\mathrm{T}_{\mathbf{v}}\mathfrak{i}%
_{D}$ is injective, then 
\begin{equation*}
\mathrm{T}_{\mathbf{v}}\mathfrak{i}_{\mathfrak{r},D}((\dot{L}_{\beta
})_{\beta \in \mathcal{L}(T_{D})},(\dot{C}^{(\alpha )})_{\alpha \in
T_{D}\setminus \mathcal{L}(T_{D})})=((0)_{\beta \in \mathcal{L}(T_{D})},0).
\end{equation*}%
In particular $\dot{C}^{(D)}=0$ and by Proposition~\ref%
{characterization_tangent_map}(b), we have 
\begin{equation}
\dot{\mathbf{w}}=\mathbf{0}=\sum_{\alpha \in S(D)}\sum_{\substack{ 1\leq
i_{\alpha }\leq r_{\alpha }}}\left( \dot{\mathbf{u}}_{i_{\alpha }}^{(\alpha
)}\otimes \mathbf{U}_{i_{\alpha }}^{(\alpha )}\right) ,  \label{lemma:eq1}
\end{equation}%
where 
\begin{equation*}
\mathbf{U}_{i_{\alpha }}^{(\alpha )}=\sum_{\substack{ 1\leq i_{\beta }\leq
r_{\beta }  \\ \beta \in S(D)  \\ \beta \neq \alpha }}C_{(i_{\beta })_{\beta
\in S(D)}}^{(D)}\bigotimes_{\beta \in S(D)}\mathbf{u}_{i_{\beta }}^{(\beta
)},
\end{equation*}%
and for each $\gamma \in T_{D}\setminus \{D\}$ we have 
\begin{equation*}
\dot{\mathbf{u}}_{i_{\gamma }}^{(\gamma )}=\left\{ 
\begin{array}{lcc}
\dot{L}_{\mu }(\mathbf{u}_{i_{\gamma }}^{(\gamma )})=\dot{S}_{\mu }(\mathbf{u%
}_{i_{\gamma }}^{(\gamma )})=\mathbf{0} & \text{ if } & \gamma \in \mathcal{L%
}(T_{D}) \\ 
&  &  \\ 
\sum_{\substack{ 1\leq i_{\beta }\leq r_{\beta }  \\ \beta \in S(\gamma )}}%
\dot{C}_{i_{\gamma },(i_{\beta })_{\beta \in S(\gamma )}}^{(\gamma
)}\bigotimes_{\beta \in S(\gamma )}\mathbf{u}_{i_{\beta }}^{(\beta
)}+\sum_{\beta \in S(\gamma )}\sum_{\substack{ 1\leq i_{\beta }\leq r_{\beta
}}}\left( \dot{\mathbf{u}}_{i_{\beta }}^{(\beta )}\otimes \mathbf{U}%
_{i_{\gamma },i_{\beta }}^{(\beta )}\right) &  & \text{ otherwise,}%
\end{array}%
\right.
\end{equation*}
where 
\begin{equation*}
\mathbf{U}_{i_{\gamma },i_{\beta }}^{(\beta )}=\sum_{\substack{ 1\leq
i_{\delta }\leq r_{\delta }  \\ \delta \in S(\mu )  \\ \delta \neq \beta }}%
C_{i_{\mu },(i_{\delta })_{\delta \in S(\gamma )}}^{(\gamma )}\bigotimes 
_{\substack{ \delta \neq \beta  \\ \delta \in S(\gamma )}}\mathbf{u}%
_{i_{\delta }}^{(\delta )},
\end{equation*}%
for $1\leq i_{\gamma }\leq r_{\gamma }$ and $1\leq i_{\beta }\leq r_{\beta
}. $ We remark that if $S(\gamma )\subset \mathcal{L}(T_{D})$ then 
\begin{equation*}
\dot{\mathbf{u}}_{i_{\gamma }}^{(\gamma )}=\sum_{\substack{ 1\leq i_{\beta
}\leq r_{\beta }  \\ \beta \in S(\gamma )}}\dot{C}_{i_{\gamma },(i_{\beta
})_{\beta \in S(\gamma )}}^{(\gamma )}\bigotimes_{\beta \in S(\gamma )}%
\mathbf{u}_{i_{\beta }}^{(\beta )}.
\end{equation*}%
From \eqref{lemma:eq1} and the fact that $\sum_{\alpha \in S(D)}\sum 
_{\substack{ 1\leq i_{\alpha }\leq r_{\alpha }}}\left( \dot{\mathbf{u}}%
_{i_{\alpha }}^{(\alpha )}\otimes \mathbf{U}_{i_{\alpha }}^{(\alpha
)}\right) \in \bigoplus_{\alpha \in S(D)}W_{\alpha }^{\min }(\mathbf{v}%
)\otimes _{a}U_{S(D)\setminus \{\alpha \}}^{\min }(\mathbf{v})$ we obtain
that 
\begin{equation*}
\sum_{\substack{ 1\leq i_{\alpha }\leq r_{\alpha }}}\left( \dot{\mathbf{u}}%
_{i_{\alpha }}^{(\alpha )}\otimes \mathbf{U}_{i_{\alpha }}^{(\alpha
)}\right) =0
\end{equation*}%
for each $\alpha \in S(D).$ Finally, $\dot{\mathbf{u}}_{i_{\alpha
}}^{(\alpha )}=\mathbf{0},$ because $\{\mathbf{U}_{i_{\alpha }}^{(\alpha
)}:1\leq i_{\alpha }\leq r_{\alpha }\}$ are linearly independent vectors for
each $\alpha \in S(D).$ In consequence, if $\alpha \in \mathcal{L}(T_{D})$
then nothing has to be done, otherwise we have that for all $\gamma \notin 
\mathcal{L}(T_{D})$ the equality 
\begin{equation*}
\mathbf{0}=\sum_{\substack{ 1\leq i_{\beta }\leq r_{\beta }  \\ \beta \in
S(\gamma )}}\dot{C}_{i_{\gamma },(i_{\beta })_{\beta \in S(\gamma
)}}^{(\gamma )}\bigotimes_{\beta \in S(\gamma )}\mathbf{u}_{i_{\beta
}}^{(\beta )}+\sum_{\beta \in S(\gamma )}\sum_{\substack{ 1\leq i_{\beta
}\leq r_{\beta }}}\left( \dot{\mathbf{u}}_{i_{\beta }}^{(\beta )}\otimes 
\mathbf{U}_{i_{\gamma },i_{\beta }}^{(\beta )}\right)
\end{equation*}%
holds for all $1\leq i_{\gamma }\leq r_{\gamma }.$ We remark that when $%
S(\gamma )\subset \mathcal{L}(T_{D})$ we have 
\begin{equation*}
\mathbf{0}=\sum_{\substack{ 1\leq i_{\beta }\leq r_{\beta }  \\ \beta \in
S(\gamma )}}\dot{C}_{i_{\gamma },(i_{\beta })_{\beta \in S(\gamma
)}}^{(\gamma )}\bigotimes_{\beta \in S(\gamma )}\mathbf{u}_{i_{\beta
}}^{(\beta )}
\end{equation*}%
and hence we obtain that $\dot{C}^{(\gamma )}=0.$ Proceeding from the leaves
to the root in the tree, we check that $\dot{C}^{(\gamma )}=0$ holds for all 
$\gamma \in T_{D}\setminus \mathcal{L}(T_{D})$ and the proposition follows.
\end{proof}

\bigskip

Now, we want to construct for each $\mathbf{v}\in \mathcal{FT}_{\mathfrak{r}%
}(\mathbf{V}_{D})\subset \mathbf{V}_{D_{\Vert \cdot \Vert _{D}}}$ a linear
subspace $\mathbf{Z}^{(D)}(\mathbf{v})\subset \mathbf{V}_{D_{\Vert \cdot
\Vert _{D}}}$ to prove that $\mathbf{Z}^{(D)}(\mathbf{v})=\mathrm{T}_{%
\mathbf{v}}\mathfrak{i}\left( \mathbb{T}_{\mathbf{v}}(\mathcal{FT}_{%
\mathfrak{r}}(\mathbf{V}_{D}))\right).$ To this end assume that 
\begin{equation*}
\mathbf{v}= (\Theta _{\mathbf{v}}^{-1}\circ \chi _{\mathfrak{r}}^{-1}(%
\mathbf{v}))(\mathfrak{0},\mathfrak{C}) = \sum_{\substack{ 1\leq i_{\alpha
}\leq r_{\alpha }  \\ \alpha \in S(D)}}C_{(i_{\alpha })_{\alpha \in
S(D)}}^{(D)}\bigotimes_{\alpha \in S(D)}\mathbf{u}_{i_{\alpha }}^{(\alpha )},
\end{equation*}%
where for each $\mu \in T_{D}\setminus (\{D\}\cup \mathcal{L}(T_{D}))$ we
have 
\begin{equation*}
\mathbf{u}_{i_{\mu }}^{(\mu )}=\sum_{\substack{ 1\leq i_{\beta }\leq
r_{\beta }  \\ \beta \in S(\alpha )}}C_{i_{\mu },(i_{\beta })_{\beta \in
S(\mu )}}^{(\mu )}\bigotimes_{\beta \in S(\mu )}\mathbf{u}_{i_{\beta
}}^{(\beta )}.
\end{equation*}%
Then to define $\mathbf{Z}^{(D)}(\mathbf{v})$ we proceed by the following
steps.

\bigskip

\noindent \textbf{Step 1:} For $\gamma \in T_{D}\setminus \mathcal{L}(T_{D})$
we observe that 
\begin{equation*}
\mathbf{u}_{i_{\gamma }}^{(\gamma )}=\sum_{\substack{ 1\leq i_{\beta }\leq
r_{\beta }  \\ \beta \in S(\gamma )}}C_{i_{\gamma },(i_{\beta })_{\beta \in
S(\alpha )}}^{(\alpha )}\bigotimes_{\beta \in S(\gamma )}\mathbf{u}%
_{i_{\beta }}^{(\beta )}\in \mathcal{M}_{(r_{\beta })_{\beta \in S(\gamma
)}}\left( \left. _{a}\bigotimes_{\beta \in S(\gamma )}\mathbf{V}_{\beta
}\right. \right)
\end{equation*}%
for $1\leq {i_{\gamma }}\leq r_{\gamma }$ and $\beta \in S(\gamma ).$ In
particular, $\mathbf{u}_{1}^{(D)}=\mathbf{v}.$ Let 
\begin{equation*}
\mathfrak{i}_{\gamma }:\mathcal{M}_{(r_{\beta })_{\beta \in S(\gamma
)}}\left( \left. _{a}\bigotimes_{\beta \in S(\gamma )}\mathbf{V}_{\beta
}\right. \right) \longrightarrow \mathbf{V}_{\gamma _{\Vert \cdot \Vert
_{\gamma }}},\quad \mathbf{u}_{\gamma }\mapsto \mathbf{u}_{\gamma },
\end{equation*}%
be the standard inclusion map. Thanks to the proof of Proposition~\ref%
{rank_one_tangent_space} we have a linear injective map 
\begin{equation*}
\mathrm{T}_{\mathbf{u}_{i_{\gamma }}^{(\gamma )}}\mathfrak{i}_{\gamma }:%
\mathop{\mathchoice{\raise-0.22em\hbox{\huge
$\times$}} {\raise-0.05em\hbox{\Large $\times$}}{\hbox{\large
$\times$}}{\times}}_{\beta \in S(\gamma )}\mathcal{L}(U_{\beta }^{\min }(%
\mathbf{v}),W_{\beta }^{\min }(\mathbf{v}))\times \mathbb{R}^{%
\mathop{\mathchoice{\raise-0.22em\hbox{\huge
$\times$}} {\raise-0.05em\hbox{\Large $\times$}}{\hbox{\large
$\times$}}{\times}}_{\beta \in S(\gamma )}r_{\beta }}\rightarrow \mathbf{V}%
_{\gamma _{\Vert \cdot \Vert _{\gamma }}}
\end{equation*}%
given by 
\begin{equation*}
\mathrm{T}_{\mathbf{u}_{i_{\gamma }}^{(\gamma )}}\mathfrak{i}_{\gamma }((%
\dot{L}_{\beta })_{\beta \in S(\gamma )},\dot{C}_{i_{\gamma }}^{(\gamma
)})=\sum_{\substack{ 1\leq i_{\beta }\leq r_{\beta }  \\ \beta \in S(\gamma
) }}\dot{C}_{i_{\gamma },(i_{\beta })_{\beta \in S(\gamma )}}^{(\gamma
)}\bigotimes_{\beta \in S(\gamma )}\mathbf{u}_{i_{\beta }}^{(\beta
)}+\sum_{\beta \in S(\gamma )}\sum_{1\leq i_{\beta }\leq r_{\beta }}\dot{L}%
_{\beta }(\mathbf{u}_{i_{\beta }}^{(\beta )})\otimes \mathbf{U}_{i_{\gamma
},i_{\beta }}^{(\beta )},
\end{equation*}%
where $\mathbf{U}_{i_{\gamma },i_{\beta }}^{(\beta )}=\sum_{\substack{ 1\leq
i_{\delta }\leq r_{\delta }  \\ \delta \in S(\gamma )  \\ \delta \neq \beta 
}}C_{i_{\gamma },(i_{\delta })_{\delta \in S(\gamma )}}^{(\alpha
)}\bigotimes_{\delta \in S(\gamma )}\mathbf{u}_{i_{\delta }}^{(\delta )}$
for $1\leq i_{\beta }\leq r_{\beta }$ and $\beta \in S(\gamma )$ and also a
linear subspace 
\begin{align*}
\mathbf{Z}^{(\gamma )}(\mathbf{u}_{j_{\gamma }}^{(\gamma )})& :=\mathrm{T}_{%
\mathbf{u}_{j_{\gamma }}^{(\gamma )}}\mathfrak{i}_{\gamma }\left( 
\mathop{\mathchoice{\raise-0.22em\hbox{\huge
$\times$}} {\raise-0.05em\hbox{\Large $\times$}}{\hbox{\large
$\times$}}{\times}}_{\beta \in S(\gamma )}\mathcal{L}(U_{\beta }^{\min }(%
\mathbf{v}),W_{\beta }^{\min }(\mathbf{v}))\times \mathbb{R}^{%
\mathop{\mathchoice{\raise-0.22em\hbox{\huge
$\times$}} {\raise-0.05em\hbox{\Large $\times$}}{\hbox{\large
$\times$}}{\times}}_{\beta \in S(\gamma )}r_{\beta }}\right) \\
& \cong 
\mathop{\mathchoice{\raise-0.22em\hbox{\huge
$\times$}} {\raise-0.05em\hbox{\Large $\times$}}{\hbox{\large
$\times$}}{\times}}_{\beta \in S(\gamma )}\mathcal{L}(U_{\beta }^{\min }(%
\mathbf{v}),W_{\beta }^{\min }(\mathbf{v}))\times \mathbb{R}^{%
\mathop{\mathchoice{\raise-0.22em\hbox{\huge
$\times$}} {\raise-0.05em\hbox{\Large $\times$}}{\hbox{\large
$\times$}}{\times}}_{\beta \in S(\gamma )}r_{\beta }}
\end{align*}%
for $1\leq j_{\gamma }\leq r_{\gamma }$ such that 
\begin{equation*}
\mathbf{Z}^{(\gamma )}(\mathbf{u}_{j_{\gamma }}^{(\gamma )})=\left.
_{a}\bigotimes_{\beta \in S(\gamma )}U_{\beta }^{\min }(\mathbf{v})\right.
\oplus \left( \bigoplus_{\beta \in S(\gamma )}W_{\beta }^{\min }(\mathbf{v}%
)\otimes _{a}\mathrm{span}\,\left\{ \mathbf{U}_{j_{\gamma },i_{\beta
}}^{(\beta )}:1\leq i_{\beta }\leq r_{\beta }\right\} \right)
\end{equation*}%
for $1\leq j_{\gamma }\leq r_{\gamma }.$ Since for each $\gamma \in
T_{D}\setminus \mathcal{L}(T_{D})$ we can write 
\begin{equation}
\mathfrak{i}_{\gamma }(\mathbf{u}_{i_{\gamma }}^{(\gamma )})=\sum_{\substack{
1\leq i_{\beta }\leq r_{\beta }  \\ \beta \in S(\alpha )}}C_{i_{\gamma
},(i_{\beta })_{\beta \in S(\gamma )}}^{(\gamma )}\bigotimes_{\beta \in
S(\gamma )}\mathbf{z}^{(\beta )}(\mathbf{u}_{i_{\beta }}^{(\beta )})
\label{frakiG}
\end{equation}
for $1\leq i_{\gamma }\leq r_{\gamma },$ where 
\begin{equation*}
\mathbf{z}^{(\beta )}(\mathbf{u}_{i_{\beta }}^{(\beta )}):=\left\{ 
\begin{array}{ll}
\mathbf{u}_{i_{\beta }}^{(\beta )} & \text{ if }\beta \in \mathcal{L}(T_{D})
\\ 
&  \\ 
\mathfrak{i}_{\beta }(\mathbf{u}_{i_{\beta }}^{(\beta )})=\sum_{\substack{ %
1\leq i_{\mu }\leq r_{\mu }  \\ \mu \in S(\beta )}}C_{i_{\beta },(i_{\mu
})_{\mu \in S(\beta )}}^{(\beta )}\bigotimes_{\mu \in S(\beta )}\mathbf{u}%
_{i_{\mu }}^{(\mu )} & \text{ otherwise, } \\ 
& 
\end{array}%
\right.
\end{equation*}%
represents that either $\mathbf{u}_{i_{\beta }}^{(\beta )}\in \mathbf{V}%
_{\beta _{\Vert \cdot \Vert _{\beta }}}$ if $\beta \in \mathcal{L}(T_{D})$
or $\mathbf{u}_{i_{\beta }}^{(\beta )}\in \mathcal{M}_{(r_{\gamma })_{\gamma
\in S(\beta )}}(\mathbf{V}_{\beta }),$ otherwise. We remark that in any case 
$\mathbf{z}^{(\beta )}(\mathbf{u}_{i_{\beta }}^{(\beta )})=\mathbf{u}%
_{i_{\beta }}^{(\beta )}.$ In particular, for each $\mathbf{v}\in \mathcal{FT%
}_{\mathfrak{r}}(\mathbf{V}_{D})$ we have 
\begin{equation}
\mathfrak{i}_{D}(\mathbf{v})=\sum_{\substack{ 1\leq i_{\beta }\leq r_{\beta
}  \\ \beta \in S(D)}}C_{(i_{\beta })_{\beta \in
S(D)}}^{(D)}\bigotimes_{\beta \in S(D)}\mathbf{z}^{(\beta )}(\mathbf{u}%
_{i_{\beta }}^{(\beta )}).  \label{frakiD}
\end{equation}%
Assume that 
\begin{equation*}
\dot{\mathbf{w}}=\mathrm{T}_{\mathbf{v}}\mathfrak{i}((\dot{L}_{k})_{k\in 
\mathcal{L}(T_{D})},(\dot{C}^{(\alpha )})_{\alpha \in T_{D}\setminus 
\mathcal{L}(T_{D})})=\mathrm{T}_{\mathbf{v}}\mathfrak{i}_{D}((\dot{L}_{\beta
})_{\beta \in S(D)},\dot{C}^{(D)}),
\end{equation*}%
where $((\dot{L}_{\beta })_{\beta \in S(D)},\dot{C}^{(D)})=\mathrm{T}_{%
\mathbf{v}}\mathfrak{i}_{\mathfrak{r},D}((\dot{L}_{k})_{k\in \mathcal{L}%
(T_{D})},(\dot{C}^{(\alpha )})_{\alpha \in T_{D}\setminus \mathcal{L}%
(T_{D})}).$ Then, using the chain rule in \eqref{frakiD} and taking into
account \eqref{frakiG}, we have 
\begin{equation*}
\dot{\mathbf{w}}=\mathrm{T}_{\mathbf{v}}\mathfrak{i}_{D}((\dot{L}_{\beta
})_{\beta \in S(D)},\dot{C}^{(D)})=\sum_{\substack{ 1\leq i_{\beta }\leq
r_{\beta }  \\ \beta \in S(D)}}\dot{C}_{(i_{\beta })_{\beta \in
S(D)}}^{(D)}\bigotimes_{\beta \in S(D)}\mathbf{u}_{i_{\beta }}^{(\beta
)}+\sum_{\beta \in S(D)}\sum_{\substack{ 1\leq i_{\beta }\leq r_{\beta }}}%
\left( \dot{\mathbf{u}}_{i_{\beta }}^{(\beta )}\otimes \mathbf{U}_{i_{\beta
}}^{(\beta )}\right)
\end{equation*}%
where for all $\mu \in T_{D}\setminus \{D\}$ either $\dot{\mathbf{u}}%
_{i_{\mu }}^{(\mu )}=\dot{L}_{\mu }(\mathbf{u}_{i_{\mu }}^{(\mu )})$ if $\mu
\in \mathcal{L}(T_{D})$ or there exists a unique 
\begin{equation*}
(\dot{L}_{\gamma })_{\substack{ \gamma \in S(\mu )  \\ \gamma \notin 
\mathcal{L}(T_{D})}}\in 
\mathop{\mathchoice{\raise-0.22em\hbox{\huge
$\times$}} {\raise-0.05em\hbox{\Large $\times$}}{\hbox{\large
$\times$}}{\times}}_{\substack{ \gamma \in S(\mu )  \\ \gamma \notin 
\mathcal{L}(T_{D})}}\mathcal{L}(U_{\gamma }^{\min }(\mathbf{v}),W_{\gamma
}^{\min }(\mathbf{v}))
\end{equation*}%
such that 
\begin{align*}
\dot{\mathbf{u}}_{i_{\mu }}^{(\mu )}& =\mathrm{T}_{\mathbf{u}_{i_{\mu
}}^{(\mu )}}\mathfrak{i}_{\mu }((\dot{L}_{\gamma })_{\gamma \in S(\mu )},%
\dot{C}_{i_{\mu }}^{(\mu )}) \\
& =\sum_{\substack{ 1\leq i_{\mu }\leq r_{\mu }  \\ \mu \in S(D)}}\dot{C}%
_{i_{\mu },(i_{\gamma })_{\gamma \in S(\mu )}}^{(D)}\bigotimes_{\gamma \in
S(\mu )}\mathbf{u}_{i_{\gamma }}^{(\gamma )}+\sum_{\gamma \in S(\mu )}\sum 
_{\substack{ 1\leq i_{\gamma }\leq r_{\gamma }}}\left( \dot{L}_{\gamma }(%
\mathbf{u}_{i_{\gamma }}^{(\gamma )})\otimes \mathbf{U}_{i_{\mu },i_{\gamma
}}^{(\gamma )}\right) \\
& =\sum_{\substack{ 1\leq i_{\mu }\leq r_{\mu }  \\ \mu \in S(D)}}\dot{C}%
_{i_{\mu },(i_{\gamma })_{\gamma \in S(\mu )}}^{(D)}\bigotimes_{\gamma \in
S(\mu )}\mathbf{u}_{i_{\gamma }}^{(\gamma )}+\sum_{\gamma \in S(\mu )}\sum 
_{\substack{ 1\leq i_{\gamma }\leq r_{\gamma }}}\left( \dot{\mathbf{u}}%
_{i_{\gamma }}^{(\gamma )}\otimes \mathbf{U}_{i_{\mu },i_{\gamma }}^{(\gamma
)}\right) ,
\end{align*}%
where the last equality is given by Lemma~\ref{characterization_tangent_map}%
(b). In consequence, we obtain that 
\begin{equation*}
\dot{\mathbf{u}}_{i_{\gamma }}^{(\gamma )}\in W_{\gamma }^{\min }(\mathbf{v})%
\text{ for all }\gamma \in T_{D}\setminus \{D\}.
\end{equation*}

\bigskip

\noindent \textbf{Step 2:} Now, for each $\gamma \in T_{D}\setminus \{D\}$
we define a linear subspace $\mathcal{H}_{\gamma }(\mathbf{v})\subset
W_{\gamma }^{\min }(\mathbf{v})^{r_{\gamma }}$ as follows. Let $\mathcal{H}%
_{\gamma }(\mathbf{v}):=W_{\gamma }^{\min }(\mathbf{v})^{r_{\gamma }}$ if $%
\gamma \in \mathcal{L}(T_{D}).$ For $\gamma \notin \mathcal{L}(T_{D})$ we
construct $\mathcal{H}_{\gamma }(\mathbf{v})$ in the following way. Let 
\begin{equation*}
\Upsilon _{\gamma ,\mathbf{v}}:\mathbb{R}^{r_{\gamma }\times 
\mathop{\mathchoice{\raise-0.22em\hbox{\huge
$\times$}} {\raise-0.05em\hbox{\Large $\times$}}{\hbox{\large
$\times$}}{\times}}_{\beta \in S(\gamma )}r_{\beta }}\times 
\mathop{\mathchoice{\raise-0.22em\hbox{\huge
$\times$}} {\raise-0.05em\hbox{\Large $\times$}}{\hbox{\large
$\times$}}{\times}}_{\beta \in S(\gamma )}\mathcal{H}_{\beta }(\mathbf{v}%
)\longrightarrow W_{\gamma }^{\min }(\mathbf{v})^{r_{\gamma }}
\end{equation*}%
be a linear map defined by 
\begin{equation*}
\Upsilon _{\gamma ,\mathbf{v}}(\dot{C}^{(\gamma )},((\mathbf{w}_{i_{\beta
}}^{(\beta )})_{i_{\beta }=1}^{r_{\beta }})_{\beta \in S(\gamma )}):=(%
\mathbf{w}_{i_{\gamma }}^{(\gamma )})_{i_{\gamma }=1}^{r_{\gamma }},
\end{equation*}%
where 
\begin{equation*}
\mathbf{w}_{i_{\gamma }}^{(\gamma )}:=\sum_{\substack{ 1\leq i_{\beta }\leq
r_{\beta }  \\ \beta \in S(\gamma )}}\dot{C}_{i_{\gamma },(i_{\beta
})_{\beta \in S(\gamma )}}^{(\gamma )}\bigotimes_{\beta \in S(\gamma )}%
\mathbf{u}_{i_{\beta }}^{(\beta )}+\sum_{\beta \in S(\gamma )}\sum_{1\leq
i_{\beta }\leq r_{\beta }}\mathbf{w}_{i_{\beta }}^{(\beta )}\otimes \mathbf{U%
}_{i_{\gamma },i_{\beta }}^{(\beta )}
\end{equation*}
for $1\leq i_{\gamma }\leq r_{\gamma }.$ Let $\pi _{i_{\gamma }}:\mathbb{R}%
^{r_{\gamma }\times 
\mathop{\mathchoice{\raise-0.22em\hbox{\huge
$\times$}} {\raise-0.05em\hbox{\Large $\times$}}{\hbox{\large
$\times$}}{\times}}_{\beta \in S(\gamma )}r_{\beta }}\rightarrow \mathbb{R}^{%
\mathop{\mathchoice{\raise-0.22em\hbox{\huge
$\times$}} {\raise-0.05em\hbox{\Large $\times$}}{\hbox{\large
$\times$}}{\times}}_{\beta \in S(\gamma )}r_{\beta }}$ be given by $\pi
_{i_{\gamma }}(\dot{C}^{(\gamma )})=\dot{C}_{i_{\gamma }}^{(\gamma )},$ for $%
1\leq i_{\gamma }\leq r_{\gamma }.$ Observe that if we define $\dot{L}%
_{\gamma }(\mathbf{u}_{i_{\gamma }}^{(\gamma )}):=\mathbf{w}_{i_{\gamma
}}^{(\gamma )}$ for $1\leq i_{\gamma }\leq r_{\gamma }$ and $\dot{L}_{\beta
}(\mathbf{u}_{i_{\beta }}):=\mathbf{w}_{i_{\beta }}^{(\beta )}$ for $1\leq
i_{\beta }\leq r_{\beta }$ and $\beta \in S(\gamma ),$ then 
\begin{equation*}
\mathbf{w}_{i_{\gamma }}^{(\gamma )}=\mathrm{T}_{{\mathbf{u}_{i_{\gamma
}}^{(\gamma )}}}\mathfrak{i}_{\gamma }(\pi _{i_{\gamma }}(\dot{C}^{(\gamma
)}),(\dot{L}_{\beta })_{\beta \in S(\gamma )})\in \mathbf{Z}^{(\gamma )}(%
\mathbf{u}_{i_{\gamma }}^{(\gamma )})
\end{equation*}%
for $1\leq i_{\gamma }\leq r_{\gamma },$ and hence by Proposition~\ref%
{characterization_tangent_map} the map $\Upsilon _{\gamma ,\mathbf{v}}$ is
injective. Finally, we define the linear subspace 
\begin{equation*}
\mathcal{H}_{\gamma }(\mathbf{v}):=\Upsilon _{\gamma ,\mathbf{v}}\left( 
\mathbb{R}^{r_{\gamma }\times 
\mathop{\mathchoice{\raise-0.22em\hbox{\huge
$\times$}} {\raise-0.05em\hbox{\Large $\times$}}{\hbox{\large
$\times$}}{\times}}_{\beta \in S(\gamma )}r_{\beta }}\times 
\mathop{\mathchoice{\raise-0.22em\hbox{\huge
$\times$}} {\raise-0.05em\hbox{\Large $\times$}}{\hbox{\large
$\times$}}{\times}}_{\beta \in S(\gamma )}\mathcal{H}_{\beta }(\mathbf{v}%
)\right) .
\end{equation*}%
For $\delta \in T_{D}\setminus \{D\}$ let $\Pi _{i_{\delta }}:W_{\delta
}^{\min }(\mathbf{v})^{r_{\delta }}\rightarrow W_{\delta }^{\min }(\mathbf{v}%
)$ be given by $\Pi _{i_{\delta }}((\mathbf{w}_{k_{\delta }}^{(\delta
)})_{k_{\delta }=1}^{r_{\delta }}):=\mathbf{w}_{i_{\delta }}^{(\delta )}$
for $1\leq i_{\delta }\leq r_{\delta }.$ Observe, that for each $\beta \in
S(\gamma ),$ we can identify $(\mathbf{w}_{i_{\beta }}^{(\beta )})_{i_{\beta
}=1}^{r_{\beta }}\in \mathcal{H}_{\beta }(\mathbf{v})$ with 
\begin{equation*}
\sum_{1\leq i_{\beta }\leq r_{\beta }}\mathbf{w}_{i_{\beta }}^{(\beta
)}\otimes \mathbf{U}_{i_{\gamma },i_{\beta }}^{(\beta )}=\sum_{1\leq
i_{\beta }\leq r_{\beta }}\Pi _{i_{\beta }}((\mathbf{w}_{k_{\beta }}^{(\beta
)})_{k_{\beta }=1}^{r_{\beta }})\otimes \mathbf{U}_{i_{\gamma },i_{\beta
}}^{(\beta )}
\end{equation*}%
for $1\leq i_{\gamma }\leq r_{\gamma }.$ It allows us to construct an
injective linear map 
\begin{equation*}
f_{\beta ,i_{\gamma }}:\mathcal{H}_{\beta }(\mathbf{v})\longrightarrow
V_{\gamma _{\Vert \cdot \Vert _{\gamma }}},\quad (\mathbf{w}_{i_{\beta
}}^{(\beta )})_{i_{\beta }=1}^{r_{\beta }}\mapsto \sum_{1\leq i_{\beta }\leq
r_{\beta }}\mathbf{w}_{i_{\beta }}^{(\beta )}\otimes \mathbf{U}_{i_{\gamma
},i_{\beta }}^{(\beta )},
\end{equation*}%
for $1\leq i_{\gamma }\leq r_{\gamma }.$ Hence $f_{\beta ,i_{\gamma }}(%
\mathcal{H}_{\beta }(\mathbf{v}))$ is a linear subspace of $V_{\gamma
_{\Vert \cdot \Vert _{\gamma }}}$ linearly isomorphic to $\mathcal{H}_{\beta
}(\mathbf{v})$ for $1\leq i_{\gamma }\leq i_{\gamma }.$ Thus, 
\begin{equation*}
\Pi _{i_{\gamma }}(\mathcal{H}_{\gamma }(\mathbf{v}))=\left\{ 
\begin{array}{ll}
\left. _{a}\bigotimes_{\beta \in S(\gamma )}U_{\beta }^{\min }(\mathbf{v}%
)\right. \oplus \left( \bigoplus_{\beta \in S(\gamma )}f_{\beta ,i_{\gamma
}}(\mathcal{H}_{\beta }(\mathbf{v}))\right) & \text{ if }\gamma \notin 
\mathcal{L}(T_{D}), \\ 
W_{\gamma }^{\min }(\mathbf{v}) & \text{ if }\gamma \in \mathcal{L}(T_{D}),%
\end{array}%
\right.
\end{equation*}%
where 
\begin{equation*}
f_{\beta ,i_{\gamma }} (\mathcal{H}_{\beta }(\mathbf{v}))= \left\{ 
\begin{array}{ll}
\bigoplus_{i_{\beta }=1}^{r_{\beta }} \Pi_{i_{\beta }}(\mathcal{H}_{\beta }(%
\mathbf{v})) \otimes_{a}\mathrm{span}\{ \mathbf{U}_{i_{\gamma
},i_{\beta}}^{(\beta )} \} & \text{ if } \beta \notin \mathcal{L}(T_{D}) \\ 
\bigoplus_{i_{\beta }=1}^{r_{\beta }}W_{\beta }^{\min }(\mathbf{v})\otimes
_{a}\mathrm{span}\{\mathbf{U}_{i_{\gamma },i_{\beta }}^{(\beta )}\} & \text{
if }\beta \in \mathcal{L}(T_{D})%
\end{array}
\right.
\end{equation*}%
for $1\leq i_{\gamma }\leq r_{\gamma }.$

\bigskip

\noindent \textbf{Step 3:} Finally, we construct a linear subspace $\mathbf{Z%
}^{(D)}(\mathbf{v}) \subset \mathbf{V}_{D_{\|\cdot\|_D}}$ by using a linear
injective map 
\begin{equation*}
\Upsilon_{D,\mathbf{v}}:\mathbb{R}^{ 
\mathop{\mathchoice{\raise-0.22em\hbox{\huge
$\times$}} {\raise-0.05em\hbox{\Large $\times$}}{\hbox{\large
$\times$}}{\times}}_{\alpha \in S(D)} r_{\alpha}} \times 
\mathop{\mathchoice{\raise-0.22em\hbox{\huge
$\times$}} {\raise-0.05em\hbox{\Large $\times$}}{\hbox{\large
$\times$}}{\times}}_{\alpha \in S(D)} \mathcal{H}_{\alpha}(\mathbf{v})
\longrightarrow \mathbf{V}_{D_{\|\cdot\|_D}}
\end{equation*}
defined by 
\begin{equation*}
\Upsilon_{\gamma,\mathbf{v}}(\dot{C}^{(D)},((\mathbf{w}_{i_{\alpha}}^{(%
\alpha)})_{i_{\alpha}=1}^{r_{\alpha}})_{\alpha \in S(D)}):= \mathbf{w}
\end{equation*}
where 
\begin{equation*}
\mathbf{w}:=\sum_{\substack{ 1\leq i_{\alpha }\leq r_{\alpha }  \\ \alpha
\in S(D )}}\dot{C}_{(i_{\alpha })_{\alpha \in S(D )}}^{(D
)}\bigotimes_{\alpha \in S(D )}\mathbf{u}_{i_{\alpha }}^{(\alpha )}+
\sum_{\alpha \in S(D)} \sum_{1 \le i_{\alpha}\le r_{\alpha}} \mathbf{w}%
_{i_{\alpha}}^{(\alpha)} \otimes \mathbf{U}_{i_{\alpha}}^{(\alpha)}.
\end{equation*}%
Then $\mathbf{Z}^{(D)}(\mathbf{v}):= \Upsilon_{D,\mathbf{v}}\left(\mathbb{R}%
^{ 
\mathop{\mathchoice{\raise-0.22em\hbox{\huge
$\times$}} {\raise-0.05em\hbox{\Large $\times$}}{\hbox{\large
$\times$}}{\times}}_{\alpha \in S(D)} r_{\alpha}} \times 
\mathop{\mathchoice{\raise-0.22em\hbox{\huge
$\times$}} {\raise-0.05em\hbox{\Large $\times$}}{\hbox{\large
$\times$}}{\times}}_{\alpha \in S(D)} \mathcal{H}_{\alpha}(\mathbf{v}%
)\right) $ and from Step 1 we have that 
\begin{equation*}
\mathrm{T}_{\mathbf{v}}\mathfrak{i}(\mathbb{T}_{\mathbf{v}}(\mathcal{FT}_{%
\mathfrak{r}}(\mathbf{V}_{D}))) \subset \mathbf{Z}^{(D)}(\mathbf{v})
\end{equation*}
holds. Moreover, we can introduce for each $\alpha \in S(D)$ a linear
injective map 
\begin{equation*}
f_{D,\alpha}: \mathcal{H}_{\alpha}(\mathbf{v}) \rightarrow \mathbf{V}%
_{D_{\|\cdot\|_D}}, \quad (\mathbf{w}_{i_{\alpha}})_{i_{\alpha}=1}^{r_{%
\alpha}} \mapsto \sum_{1 \le i_{\alpha}\le r_{\alpha}} \mathbf{w}%
_{i_{\alpha}}^{(\alpha)} \otimes \mathbf{U}_{i_{\alpha}}^{(\alpha)}.
\end{equation*}
Then $f_{D,\alpha}(\mathcal{H}_{\alpha}(\mathbf{v}))$ is a linear subspace in $%
\mathbf{V}_{D_{\|\cdot\|_D}}$ linearly isomorphic to $\mathcal{H}_{\alpha}(%
\mathbf{v}).$ It is not difficult to show that 
\begin{equation*}
f_{D,\alpha}(\mathcal{H}_{\alpha}(\mathbf{v}))= \left\{ 
\begin{array}{ll}
\bigoplus_{i_{\alpha}=1}^{r_{\alpha}} \Pi_{i_{\alpha}}(\mathcal{H}_{\alpha}(%
\mathbf{v})) \otimes_a \mathrm{span}\{\mathbf{U}_{i_{\alpha}}^{(\alpha)} \}
& \text{ if } \alpha \notin \mathcal{L}(T_D) \\ 
\bigoplus_{i_{\alpha}=1}^{r_{\alpha}} W_{\alpha}^{\min}(\mathbf{v})
\otimes_a \mathrm{span}\{\mathbf{U}_{i_{\alpha}}^{(\alpha)} \} & \text{ if }
\alpha \in \mathcal{L}(T_D)%
\end{array}
\right.
\end{equation*}
for $\alpha \in S(D).$ By construction, we have 
\begin{equation*}
\mathbf{Z}^{(D)}(\mathbf{v}) = \left._a \bigotimes_{\alpha \in S(D)}
U_{\alpha}^{\min}(\mathbf{v})\right. \oplus \left( \bigoplus_{\alpha \in
S(D)} f_{D,\alpha}(\mathcal{H}_{\alpha}(\mathbf{v})) \right).
\end{equation*}

\bigskip

\begin{proposition}
Assume that $S(D) \neq \mathcal{L}(T_D)$ and the tensor product map $%
\bigotimes$ is $T_D$-continuous. Let $\mathbf{v}\in \mathcal{FT}_{\mathfrak{r%
}}(\mathbf{V}_{D}),$ then $\mathrm{T}_{\mathbf{v}}\mathfrak{i}(\mathbb{T}_{%
\mathbf{v}}(\mathcal{FT}_{\mathfrak{r}}(\mathbf{V}_{D}))) = \mathbf{Z}^{(D)}(%
\mathbf{v})$ and hence it is linearly isomorphic to $\mathbb{T}_{\mathbf{v}}(%
\mathcal{FT}_{\mathfrak{r}}(\mathbf{V}_{D})).$
\end{proposition}

\begin{proof}
From Step 1 and the construction of $\mathbf{Z}^{(D)}(\mathbf{v}),$ the
inclusion $\mathrm{T}_{\mathbf{v}}\mathfrak{i}(\mathbb{T}_{\mathbf{v}}(%
\mathcal{FT}_{\mathfrak{r}}(\mathbf{V}_{D}))) \subset \mathbf{Z}^{(D)}(%
\mathbf{v})$ holds. Now, take $\mathbf{w} \in \mathbf{Z}^{(D)}(\mathbf{v}).$
Then we can write 
\begin{equation*}
\mathbf{w} = \sum_{\substack{ 1\leq i_{\alpha }\leq r_{\alpha }  \\ \alpha
\in S(D)}}(\dot{C}^{(D)})_{(i_{\alpha })_{\alpha \in
S(D)}}\bigotimes_{\alpha \in S(D)}\mathbf{u}_{i_{\alpha }}^{(\alpha
)}+\sum_{\alpha \in S(D)}\sum_{\substack{ 1\leq i_{\alpha }\leq r_{\alpha }}}
\left(\mathbf{w}_{i_{\alpha }}^{(\alpha )}\otimes \mathbf{U}_{i_{\alpha
}}^{(\alpha )}\right),
\end{equation*}
where $\dot{C}^{(D)} \in \mathbb{R}^{ 
\mathop{\mathchoice{\raise-0.22em\hbox{\huge
$\times$}} {\raise-0.05em\hbox{\Large $\times$}}{\hbox{\large
$\times$}}{\times}}_{\alpha \in S(D)} r_{\alpha}}$ and $\mathbf{w}%
_{i_{\alpha }}^{(\alpha )} \in W_{\alpha}^{\min}(\mathbf{v}) $ for $1\leq
i_{\alpha }\leq r_{\alpha }.$ Then we can define $\dot{L}_{\alpha} \in 
\mathcal{L}( U_{\alpha}^{\min}(\mathbf{v}) , W_{\alpha}^{\min}(\mathbf{v}) )$
by $\dot{L}_{\alpha}(\mathbf{u}_{i_{\alpha }}^{(\alpha )}):= \mathbf{w}%
_{i_{\alpha }}^{(\alpha )}$ for $1\leq i_{\alpha }\leq r_{\alpha },$ and we
have 
\begin{equation*}
(\dot{C}^{(D)},(\dot{L}_{\alpha})_{\alpha\in S(D)}) \in \mathbb{R}^{ 
\mathop{\mathchoice{\raise-0.22em\hbox{\huge
$\times$}} {\raise-0.05em\hbox{\Large $\times$}}{\hbox{\large
$\times$}}{\times}}_{\alpha \in S(D)} r_{\alpha}} \times 
\mathop{\mathchoice{\raise-0.22em\hbox{\huge
$\times$}} {\raise-0.05em\hbox{\Large $\times$}}{\hbox{\large
$\times$}}{\times}}_{\alpha \in S(D)}\mathcal{L}(U_{\alpha}^{\min}(\mathbf{v}%
),W_{\alpha}^{\min}(\mathbf{v})).
\end{equation*}
Moreover, $\sum_{\substack{ 1\leq i_{\alpha }\leq r_{\alpha }}} \mathbf{w}%
_{i_{\alpha }}^{(\alpha )}\otimes \mathbf{U}_{i_{\alpha }}^{(\alpha )} \in
f_{D,\alpha}(\mathcal{H}_{\alpha}(\mathbf{v}))$ for $\alpha \in S(D).$ If $%
\alpha \notin \mathcal{L}(T_D)$, then $(\mathbf{w}_{i_{\alpha}}^{(%
\alpha)})_{i_{\alpha}=1}^{r_{\alpha}} \in \mathcal{H}_{\alpha}(\mathbf{v}) =
\Upsilon_{\alpha,\mathbf{v}}\left(\mathbb{R}^{r_{\alpha} \times 
\mathop{\mathchoice{\raise-0.22em\hbox{\huge
$\times$}} {\raise-0.05em\hbox{\Large $\times$}}{\hbox{\large
$\times$}}{\times}}_{\beta \in S(\gamma)} r_{\beta}} \times 
\mathop{\mathchoice{\raise-0.22em\hbox{\huge
$\times$}} {\raise-0.05em\hbox{\Large $\times$}}{\hbox{\large
$\times$}}{\times}}_{\beta \in S(\alpha)} \mathcal{H}_{\beta}(\mathbf{v})
\right).$ Hence there exists 
\begin{equation*}
(\dot{C}^{(\alpha)},((\mathbf{w}_{i_{\beta}}^{(\beta)})_{i_{\beta}=1}^{r_{%
\beta}})_{\beta \in S(\alpha)}) \in \mathbb{R}^{r_{\alpha} \times 
\mathop{\mathchoice{\raise-0.22em\hbox{\huge
$\times$}} {\raise-0.05em\hbox{\Large $\times$}}{\hbox{\large
$\times$}}{\times}}_{\beta \in S(\alpha)} r_{\beta}} \times 
\mathop{\mathchoice{\raise-0.22em\hbox{\huge
$\times$}} {\raise-0.05em\hbox{\Large $\times$}}{\hbox{\large
$\times$}}{\times}}_{\beta \in S(\alpha)} \mathcal{H}_{\beta}(\mathbf{v})
\end{equation*}
such that 
\begin{equation*}
\mathbf{w}_{i_{\alpha }}^{(\alpha )} = \sum_{\substack{ 1 \le i_{\beta} \le
r_{\beta}  \\ \beta \in S(\alpha)}} \dot{C}^{(\alpha)}_{i_{\alpha},(i_{%
\beta})_{\beta \in S(\alpha)}}\bigotimes_{\beta \in S(\alpha)}\mathbf{u}%
_{i_{\beta}}^{(\beta)} + \sum_{\beta \in S(\alpha)} \sum_{1 \le i_{\beta}\le
r_{\beta}} \mathbf{w}_{i_{\beta}}^{(\beta)} \otimes \mathbf{U}%
_{i_{\alpha},i_{\beta}}^{(\beta)}
\end{equation*}
for $1 \le i_{\alpha} \le r_{\alpha}.$ Define $\dot{L}_{\beta}(\mathbf{u}%
_{i_{\beta}}^{(\beta)}):= \mathbf{w}_{i_{\beta}}^{(\beta)}$ for $1 \le
i_{\beta} \le r_{\beta}$ and $\beta \in S(\alpha).$ Then 
\begin{equation*}
(\dot{C}^{(\alpha)},(\dot{L}_{\beta})_{\beta \in S(\alpha)}) \in \mathbb{R}%
^{r_{\gamma} \times 
\mathop{\mathchoice{\raise-0.22em\hbox{\huge
$\times$}} {\raise-0.05em\hbox{\Large $\times$}}{\hbox{\large
$\times$}}{\times}}_{\beta \in S(\gamma)} r_{\beta}} \times 
\mathop{\mathchoice{\raise-0.22em\hbox{\huge
$\times$}} {\raise-0.05em\hbox{\Large $\times$}}{\hbox{\large
$\times$}}{\times}}_{\beta \in S(\alpha)}\mathcal{L}(U_{\beta}^{\min}(%
\mathbf{v}),W_{\beta}^{\min}(\mathbf{v})).
\end{equation*}
Moreover, $\sum_{1 \le i_{\beta}\le r_{\beta}} \mathbf{w}_{i_{\beta}}^{(%
\beta)} \otimes \mathbf{U}_{i_{\alpha},i_{\beta}}^{(\beta)} \in
f_{\beta,i_{_{\alpha}}}(\mathcal{H}_{\beta}(\mathbf{v}))$ for $1 \le
i_{\alpha} \le r_{\alpha}.$ If $\beta \notin \mathcal{L}(T_D)$, then $(%
\mathbf{w}_{i_{\beta}}^{(\beta)})_{i_{\beta}=1}^{r_{\beta}} \in \mathcal{H}%
_{\beta}(\mathbf{v}) = \Upsilon_{\beta,\mathbf{v}}\left(\mathbb{R}%
^{r_{\beta} \times 
\mathop{\mathchoice{\raise-0.22em\hbox{\huge
$\times$}} {\raise-0.05em\hbox{\Large $\times$}}{\hbox{\large
$\times$}}{\times}}_{\gamma \in S(\beta)} r_{\gamma}} \times 
\mathop{\mathchoice{\raise-0.22em\hbox{\huge
$\times$}} {\raise-0.05em\hbox{\Large $\times$}}{\hbox{\large
$\times$}}{\times}}_{\gamma \in S(\beta)} \mathcal{H}_{\gamma}(\mathbf{v})
\right).$ Proceeding in a similar way from the root to the leaves, we
construct $(\dot{\mathfrak{L}},\dot{\mathfrak{C}})\in \mathbb{T}_{\mathbf{v}%
}(\mathcal{FT}_{\mathfrak{r}}(\mathbf{V}_{D})),$ where $\dot{\mathfrak{C}}=(%
\dot{C}^{(\alpha )})_{\alpha \in T_{D}\setminus \mathcal{L}(T_{D})} \in 
\mathbb{R}^{\mathfrak{r}}$ and $\dot{\mathfrak{L}}=(\dot{L}_{\alpha
})_{\alpha \in T_{D}\setminus \{D\}}\in \mathcal{L}_{T_D}(\mathbf{v})$ such
that $\mathbf{w}=\mathrm{T}_{\mathbf{v}}\mathfrak{i}(\dot{\mathfrak{C}},\dot{%
\mathfrak{L}}).$ Thus, we can conclude that $\mathbf{Z}^{(D)}(\mathbf{v})
\subset \mathrm{T}_{\mathbf{v}}\mathfrak{i}\left( \mathbb{T}_{\mathbf{v}}(%
\mathcal{FT}_{\mathfrak{r}}(\mathbf{V}_{D}))\right)$ and the equality
follows.
\end{proof}

\begin{figure}[tbp]
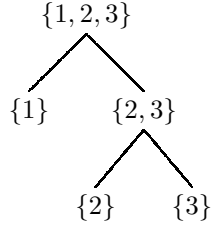

\centering
\synttree[$\{1,2,3\}$
[$\{1\}$]
[$\{2,3\}$[$\{2\}$][$\{3\}$]
]]
\caption{A binary tree $T_D.$}
\label{figTREE}
\end{figure}

\begin{example}
Consider the binary tree $T_D$ given in Figure~\ref{figTREE} and consider TB
ranks $\mathfrak{r}=(1,r_{1},r_{23},r_{2},r_{3}).$ Let $\mathbf{v} \in 
\mathcal{FT}_{\mathfrak{r}}(V_1 \otimes_a V_2 \otimes_a V_3)$ and assume
that the tensor product map $\bigotimes$ is $T_D$-continuous. Then 
\begin{align*}
\mathbf{Z}^{(123)}(\mathbf{v}) = \left(U_1^{\min}(\mathbf{v}) \otimes_a
U_{23}^{\min}(\mathbf{v})\right) \oplus f_{123,1}(\mathcal{H}_1(\mathbf{v}))
\oplus f_{123,23}(\mathcal{H}_{23}(\mathbf{v})),
\end{align*}
where 
\begin{align*}
f_{123,1}(\mathcal{H}_1(\mathbf{v})) & = \bigoplus_{i_{1}=1}^{r_1}
W_1^{\min}(\mathbf{v}) \otimes_a \mathrm{span}\,\{ \mathbf{U}_{i_1}^{(1)}\}
\subset V_{1_{\|\cdot\|_1}} \otimes_a \left(V_{2_{\|\cdot\|_2}} \otimes_a
V_{3_{\|\cdot\|_3}}\right) , \\
f_{123,23}(\mathcal{H}_{23}(\mathbf{v})) & = \bigoplus_{i_{23}=1}^{r_{23}} 
\mathrm{span}\,\{ \mathbf{U}_{i_{23}}^{(23)}\} \otimes_a \Pi_{i_{23}}(%
\mathcal{H}_{23}(\mathbf{v})) \subset V_{1_{\|\cdot\|_1}} \otimes_a
\left(V_{2_{\|\cdot\|_2}} \otimes_a V_{3_{\|\cdot\|_3}} \right),
\end{align*}
and 
\begin{align*}
\Pi_{i_{23}}(\mathcal{H}_{23}(\mathbf{v})) & = \left(U_2^{\min}(\mathbf{v})
\otimes_a U_{3}^{\min}(\mathbf{v})\right) \\
& \oplus \left(\bigoplus_{i_{2}=1}^{r_2} W_2^{\min}(\mathbf{v}) \otimes_a 
\mathrm{span}\,\{ \mathbf{U}_{i_{23},i_{2}}^{(2)}\} \right) \oplus
\left(\bigoplus_{i_{3}=1}^{r_3} \mathrm{span}\,\{ \mathbf{U}%
_{i_{23},i_{3}}^{(3)}\} \otimes_a W_3^{\min}(\mathbf{v}) \right),
\end{align*}
which is a linear subspace in $V_{2_{\|\cdot\|_2}} \otimes_a
V_{3_{\|\cdot\|_3}}.$
\end{example}

\subsubsection{Is the standard inclusion map an immersion?}

Finally, to show that $\mathfrak{i}$ is an immersion, and hence $\mathcal{FT}%
_{\mathfrak{r}}(\mathbf{V}_{D})$ is an immersed submanifold of $\mathbf{V}%
_{D_{\Vert \cdot \Vert _{D}}},$ we need to prove that $\mathrm{T}_{\mathbf{v}%
}\mathfrak{i}\left( \mathbb{T}_{\mathbf{v}}(\mathcal{FT}_{\mathfrak{r}}(%
\mathbf{V}_{D}))\right) \in \mathbb{G}(\mathbf{V}_{\Vert \cdot \Vert _{D}}).$
Let $\{\mathbf{V}_{\alpha _{\Vert \cdot \Vert _{\alpha }}}\}_{\alpha \in
T_{D}\setminus \{D\}}$ be a representation of the Banach tensor space $%
\mathbf{V}_{D_{\Vert \cdot \Vert _{D}}}=\left. _{\Vert \cdot \Vert
_{D}}\bigotimes_{j\in D}V_{j}\right. ,$ in the topological tree-based format
and take $\mathbf{V}_{D}:=\left. _{a}\bigotimes_{j\in D}V_{j}\right. .$ A
first useful result is the following lemma.

\bigskip

\begin{lemma}
\label{two_faces} Assume that \eqref{tree_injective_norm} holds. Let $\alpha
\in T_{D}\setminus \mathcal{L}(T_{D})$ and take $\beta \in S(\alpha ).$ If $%
W_{\beta }\in \mathbb{G}(\mathbf{V}_{\beta _{\Vert \cdot \Vert _{\beta }}})$
satisfies $\mathbf{V}_{\beta _{\Vert \cdot \Vert _{\beta }}}=U_{\beta
}\oplus W_{\beta }$ for some finite-dimensional subspace $U_{\beta }$ in $%
\mathbf{V}_{\beta _{\Vert \cdot \Vert _{\beta }}},$ then $W_{\beta }\otimes
_{a}U_{[\beta ]}\in \mathbb{G}(\mathbf{V}_{\alpha _{\Vert \cdot \Vert
_{\alpha }}})$ for every finite-dimensional subspace $U_{[\beta ]}\subset
\left. _{a}\bigotimes_{\delta \in S(\alpha )\setminus \beta }\mathbf{V}%
_{\delta _{\Vert \cdot \Vert _{\delta }}}\right. .$
\end{lemma}

\begin{proof}
First, observe that if $W_{\beta }$ is a finite-dimensional subspace, then $%
W_{\beta }\otimes _{a}U_{[\beta ]}$ is also finite dimensional, and hence
the lemma follows. Thus, assume that $W_{\beta }$ is an infinite-dimensional
closed subspace of $\mathbf{V}_{\beta _{\Vert \cdot \Vert _{\beta }}},$ and
to simplify the notation write 
\begin{equation*}
\mathbf{X}_{\beta }:=\left. _{\Vert \cdot \Vert _{\vee (S(\alpha )\setminus
\beta )}}\bigotimes_{\delta \in S(\alpha)\setminus \{\beta \}}\mathbf{V}_{\delta
_{\Vert \cdot \Vert _{\delta }}}\right. .
\end{equation*}%
If $U_{[\beta ]}\subset \mathbf{X}_{\beta }$ is a finite-dimensional
subspace, then there exists $W_{[\beta ]}\in \mathbb{G}(\mathbf{X}_{\beta })$
such that $\mathbf{X}_{\beta }=U_{[\beta ]}\oplus W_{[\beta ]}.$ Since the
tensor product map 
\begin{equation*}
\bigotimes :(\mathbf{V}_{\beta _{\Vert \cdot \Vert _{\beta }}},\Vert \cdot
\Vert _{\beta })\times \left( \mathbf{X}_{\beta },\Vert \cdot \Vert _{\vee
(S(\alpha )\setminus \beta )}\right) \rightarrow (\mathbf{V}_{\alpha _{\Vert
\cdot \Vert _{\alpha }}},\Vert \cdot \Vert _{\alpha })
\end{equation*}%
is continuous and by Lemma~3.18 in \cite{FALHACK}, for each elementary
tensor $\mathbf{v}_{\beta }\otimes \mathbf{v}_{[\beta ]}\in \mathbf{V}%
_{\beta _{\Vert \cdot \Vert _{\beta }}}\otimes _{a}\mathbf{X}_{\beta }$ we
have 
\begin{align*}
\Vert (id_{\beta }\otimes P_{_{U_{[\beta ]}\oplus W_{[\beta ]}}})(\mathbf{v}%
_{\beta }\otimes \mathbf{v}_{[\beta ]})\Vert _{\alpha }& \leq C\sqrt{\dim
U_{[\beta ]}}\,\Vert \mathbf{v}_{\beta }\Vert _{\beta }\Vert \mathbf{v}%
_{[\beta ]}\Vert _{\vee (S(\alpha )\setminus \beta )} \\
& =C\,\sqrt{\dim U_{[\beta ]}}\,\Vert \mathbf{v}_{\beta }\otimes \mathbf{v}%
_{[\beta ]}\Vert _{\vee (S(\alpha ))} \\
& \leq C^{\prime }\,\sqrt{\dim U_{[\beta ]}}\,\Vert \mathbf{v}_{\beta
}\otimes \mathbf{v}_{[\beta ]}\Vert _{\alpha }.
\end{align*}%
Thus, $(id_{\beta }\otimes P_{_{U_{[\beta ]}\oplus W_{[\beta ]}}})$ is
continuous over $\mathbf{V}_{\beta _{\Vert \cdot \Vert _{\beta }}}\otimes
_{a}\mathbf{X}_{\beta },$ and hence in $\mathbf{V}_{\alpha _{\Vert \cdot
\Vert _{\alpha }}}.$ Now, take into account the fact that 
\begin{equation*}
id_{\beta }=P_{_{U_{\beta }\oplus W_{\beta }}}+P_{_{W_{\beta }\oplus
U_{\beta }}},
\end{equation*}%
so that
\begin{equation*}
id_{\beta }\otimes P_{_{U_{[\beta ]}\oplus W_{[\beta ]}}}=P_{_{U_{\beta
}\oplus W_{\beta }}}\otimes P_{_{U_{[\beta ]}\oplus W_{[\beta
]}}}+P_{_{W_{\beta }\oplus U_{\beta }}}\otimes P_{_{U_{[\beta ]}\oplus
W_{[\beta ]}}}.
\end{equation*}%
Observe that $id_{\beta }\otimes P_{_{U_{[\beta ]}\oplus W_{[\beta ]}}}$ and 
$P_{_{U_{\beta }\oplus W_{\beta }}}\otimes P_{_{U_{[\beta ]}\oplus W_{[\beta
]}}}$ are continuous linear maps over $\mathbf{V}_{\beta _{\Vert \cdot \Vert
_{\beta }}}\otimes _{a}\mathbf{X}_{\beta },$ and then $P_{_{W_{\beta }\oplus
U_{\beta }}}\otimes P_{_{U_{[\beta ]}\oplus W_{[\beta ]}}}$ is a continuous
linear map over $\mathbf{V}_{\beta _{\Vert \cdot \Vert _{\beta }}}\otimes
_{a}\mathbf{X}_{\beta }.$ Thus, 
\begin{equation*}
\mathcal{P}_{\alpha }:=\overline{P_{_{W_{\beta }\oplus U_{\beta }}}\otimes
P_{_{U_{[\beta ]}\oplus W_{[\beta ]}}}}\in \mathcal{L}(\mathbf{V}_{\alpha
_{\Vert \cdot \Vert _{\alpha }}},\mathbf{V}_{\alpha _{\Vert \cdot \Vert
_{\alpha }}})
\end{equation*}%
and $\mathcal{P}_{\alpha }\circ \mathcal{P}_{\alpha }=\mathcal{P}_{\alpha }.$
Since $\mathcal{P}_{\alpha }(\mathbf{V}_{\alpha _{\Vert \cdot \Vert _{\alpha
}}})=W_{\beta }\otimes _{a}U_{[\beta ]},$ 
the lemma follows by Proposition~\ref{characterize_P}.
\end{proof}

\bigskip

\begin{lemma}
\label{direct_sum_G} Let $X$ be a Banach space and assume that $U,V \in 
\mathbb{G}(X).$ If $U\cap V=\{0\},$ then $U\oplus V\in \mathbb{G}(X).$
Moreover, $U\cap V\in \mathbb{G}(X)$ holds.
\end{lemma}

\begin{proof}
To prove the first statement assume that $U \cap V = \{0\}.$ Since $U,V \in 
\mathbb{G}(X)$ there exist $U^{\prime },V^{\prime }\in \mathbb{G}(X),$ such
that $X=U \oplus U^{\prime }= V \oplus V^{\prime }.$ Then $U=X \cap U = (V
\oplus V^{\prime }) \cap U = U \cap V^{\prime }$ and $V = X \cap V = (U
\oplus U^{\prime }) \cap V = V \cap U^{\prime }.$ In consequence, we can
write 
\begin{equation*}
U \oplus V \oplus (U^{\prime }\cap V^{\prime }) = (U \cap V^{\prime })
\oplus (V \cap U^{\prime }) \oplus (U^{\prime }\cap V^{\prime }) = (U \oplus
U^{\prime }) \cap ( V \oplus V^{\prime }) = X,
\end{equation*}
and the first statement follows. To prove the second one, observe that $X =
(U \cap V) \oplus (U \cap V^{\prime }) \oplus (V \cap U^{\prime }) \oplus
(U^{\prime }\cap V^{\prime }). $
\end{proof}

\bigskip

A very useful consequence of the above two lemmas is the following Theorem.

\begin{theorem}
\label{closed_linear_subspace} Let $\{\mathbf{V}_{\alpha _{\Vert \cdot \Vert
_{\alpha }}}\}_{\alpha \in T_{D}\setminus \{D\}}$ be a representation of a
tensor Banach space $\mathbf{V}_{D_{\Vert \cdot \Vert _{D}}}=\left. _{\Vert
\cdot \Vert _{D}}\bigotimes_{j\in D}V_{j}\right. $ in the topological 
tree-based format and assume that \eqref{tree_injective_norm} holds. Then $%
\mathbf{Z}^{(D)}(\mathbf{v})\in \mathbb{G}(\mathbf{V}_{D_{\Vert \cdot \Vert
_{D}}}),$ and hence $\mathcal{FT}_{\mathfrak{r}}(\mathbf{V}_{D})$ is an
immersed submanifold of $\mathbf{V}_{D_{\Vert \cdot \Vert _{D}}}.$
\end{theorem}

\begin{proof}
Since the tensor product map is $T_{D}$-continuous, Proposition~\ref%
{characterization_tangent_map} gives us the differentiability of $\mathrm{T}%
_{\mathbf{v}}i.$ Assume first that $S(D)=\mathcal{L}(T_{D}).$ From Corollary~%
\ref{rank_one_tangent_space} we have 
\begin{equation*}
\mathbf{Z}^{(D)}(\mathbf{v})=\left. _{a}\bigotimes_{\alpha \in
S(D)}U_{\alpha }^{\min }(\mathbf{v})\right. \oplus \left( \bigoplus_{\alpha
\in S(D)}W_{\alpha }^{\min }(\mathbf{v})\otimes _{a}U_{S(D)\setminus
\{\alpha \}}^{\min }(\mathbf{v})\right) .
\end{equation*}%
For each $\alpha \in S(D)$ we have $W_{\alpha }^{\min }(\mathbf{v})\in 
\mathbb{G}(\mathbf{V}_{\alpha _{\Vert \cdot \Vert _{\alpha }}})$ and $%
U_{S(D)\setminus \{\alpha \}}^{\min }(\mathbf{v})\subset \left.
_{a}\bigotimes_{\delta \in S(D)\setminus \{\alpha \}}\mathbf{V}_{\delta
_{\Vert \cdot \Vert _{\delta }}}\right. $ is a finite-dimensional subspace.
From Lemma~\ref{two_faces} we have $W_{\alpha }^{\min }(\mathbf{v})\otimes
_{a}U_{S(D)\setminus \{\alpha \}}^{\min }(\mathbf{v})\in \mathbb{G}(\mathbf{V%
}_{D_{\Vert \cdot \Vert _{D}}})$ for all $\alpha \in S(D).$ Since $\left.
_{a}\bigotimes_{\alpha \in S(D)}U_{\alpha }^{\min }(\mathbf{v})\right. \in 
\mathbb{G}(\mathbf{V}_{D_{\Vert \cdot \Vert _{D}}}),$ by Lemma~\ref%
{direct_sum_G}, we obtain that $\mathbf{Z}^{(D)}(\mathbf{v})\in \mathbb{G}(%
\mathbf{V}_{D_{\Vert \cdot \Vert _{D}}}).$

\bigskip

Now, assume that $S(D)\neq \mathcal{L}(T_{D}).$ Then 
\begin{equation*}
\mathbf{Z}^{(D)}(\mathbf{v})=\left. _{a}\bigotimes_{\alpha \in
S(D)}U_{\alpha }^{\min }(\mathbf{v})\right. \oplus \left( \bigoplus_{\alpha
\in S(D)}f_{D,\alpha }(\mathcal{H}_{\alpha }(\mathbf{v}))\right)
\end{equation*}
and 
\begin{equation*}
f_{D,\alpha }(\mathcal{H}_{\alpha }(\mathbf{v}))=\left\{ 
\begin{array}{ll}
\bigoplus_{i_{\alpha }=1}^{r_{\alpha }}\Pi _{i_{\alpha }}(\mathcal{H}%
_{\alpha }(\mathbf{v}))\otimes _{a}\mathrm{span}\{\mathbf{U}_{i_{\alpha
}}^{(\alpha )}\} & \text{ if }\alpha \notin \mathcal{L}(T_{D}) \\ 
\bigoplus_{i_{\alpha }=1}^{r_{\alpha }}W_{\alpha }^{\min }(\mathbf{v}%
)\otimes _{a}\mathrm{span}\{\mathbf{U}_{i_{\alpha }}^{(\alpha )}\} & \text{
if }\alpha \in \mathcal{L}(T_{D})%
\end{array}%
\right.
\end{equation*}%
for $\alpha \in S(D).$ For $\alpha \in \mathcal{L}(T_{D})$ we have $%
W_{\alpha }^{\min }(\mathbf{v})\in \mathbb{G}(\mathbf{V}_{\alpha _{\Vert
\cdot \Vert _{\alpha }}})$ and $\mathrm{span}\{\mathbf{U}_{i_{\alpha
}}^{(\alpha )}\}$ is a finite-dimensional subspace for $1\leq i_{\alpha
}\leq r_{\alpha }$, and from Lemma~\ref{two_faces}, $W_{\alpha }^{\min }(%
\mathbf{v})\otimes _{a}\mathrm{span}\{\mathbf{U}_{i_{\alpha }}^{(\alpha
)}\}\in \mathbb{G}(\mathbf{V}_{D_{\Vert \cdot \Vert _{D}}})$ for $1\leq
i_{\alpha }\leq r_{\alpha }.$ By Lemma~\ref{direct_sum_G}, $f_{D,\alpha }(%
\mathcal{H}_{\alpha }(\mathbf{v}))\in \mathbb{G}(\mathbf{V}_{D_{\Vert \cdot
\Vert _{D}}}).$ Otherwise, if $\alpha \notin \mathcal{L}(T_{D})$ then 
\begin{equation*}
f_{D,\alpha }(\mathcal{H}_{\alpha }(\mathbf{v}))=\bigoplus_{i_{\alpha
}=1}^{r_{\alpha }}\Pi _{i_{\alpha }}(\mathcal{H}_{\alpha }(\mathbf{v}%
))\otimes _{a}\mathrm{span}\{\mathbf{U}_{i_{\alpha }}^{(\alpha )}\},
\end{equation*}%
where 
\begin{equation*}
\Pi _{i_{\alpha }}(\mathcal{H}_{\alpha }(\mathbf{v}))=\left.
_{a}\bigotimes_{\beta \in S(\alpha )}U_{\beta }^{\min }(\mathbf{v})\right.
\oplus \left( \bigoplus_{\beta \in S(\alpha )}f_{\beta ,i_{\alpha }}(%
\mathcal{H}_{\beta }(\mathbf{v}))\right)
\end{equation*}%
for $1\leq i_{\alpha }\leq r_{\alpha }.$ Now, 
\begin{equation*}
f_{\beta ,i_{\alpha }}(\mathcal{H}_{\beta }(\mathbf{v}))=\left\{ 
\begin{array}{ll}
\bigoplus_{i_{\beta }=1}^{r_{\beta }}\Pi _{i_{\beta }}(\mathcal{H}_{\beta }(%
\mathbf{v}))\otimes _{a}\mathrm{span}\{\mathbf{U}_{i_{\alpha },i_{\beta
}}^{(\beta )}\} & \text{ if }\beta \notin \mathcal{L}(T_{D}) \\ 
\bigoplus_{i_{\beta }=1}^{r_{\beta }}W_{\beta }^{\min }(\mathbf{v})\otimes
_{a}\mathrm{span}\{\mathbf{U}_{i_{\alpha },i_{\beta }}^{(\beta )}\} & \text{
if }\beta \in \mathcal{L}(T_{D})%
\end{array}%
\right.
\end{equation*}%
for $1\leq i_{\alpha }\leq r_{\alpha }.$ Clearly, if $\beta \in \mathcal{L}%
(T_{D})$ then $f_{\beta ,i_{\alpha }}(\mathcal{H}_{\beta }(\mathbf{v}))\in 
\mathbb{G}(\mathbf{V}_{\alpha _{\Vert \cdot \Vert _{\alpha }}})$ for $1\leq
i_{\alpha }\leq r_{\alpha }.$ Then we can write, 
\begin{equation*}
\Pi _{i_{\alpha }}(\mathcal{H}_{\alpha }(\mathbf{v}))=\left.
_{a}\bigotimes_{\beta \in S(\alpha )}U_{\beta }^{\min }(\mathbf{v})\right.
\oplus \left( \bigoplus_{\substack{ \beta \in S(\alpha )  \\ \beta \in 
\mathcal{L}(T_{D})}}f_{\beta ,i_{\alpha }}(\mathcal{H}_{\beta }(\mathbf{v}%
))\right) \oplus \left( \bigoplus_{\substack{ \beta \in S(\alpha )  \\ \beta
\notin \mathcal{L}(T_{D})}}f_{\beta ,i_{\alpha }}(\mathcal{H}_{\beta }(%
\mathbf{v}))\right)
\end{equation*}%
for $1\leq i_{\alpha }\leq r_{\alpha }.$ Starting from the leaves, that is $%
\gamma \in \mathcal{L}(T_{D}),$ we have that  $\Pi _{i_{\gamma }}(%
\mathcal{H}_{\gamma }(\mathbf{v}))=W_{\gamma }^{\min }(\mathbf{v})\in 
\mathbb{G}(\mathbf{V}_{\gamma _{\Vert \cdot \Vert _{\gamma }}})$ for $1\leq
i_{\gamma }\leq r_{\gamma },$ and hence for $\delta \in T_{D}$ such that $%
\gamma \in S(\delta )$ we obtain $f_{\gamma ,i_{\delta }}(\mathcal{H}%
_{\gamma }(\mathbf{v}))\in \mathbb{G}(\mathbf{V}_{\delta _{\Vert \cdot \Vert
_{\delta }}})$ for $1\leq i_{\delta }\leq r_{\delta }.$ Proceeding
inductively from the leaves to the root, we obtain that $f_{\beta ,i_{\alpha
}}(\mathcal{H}_{\beta }(\mathbf{v}))\in \mathbb{G}(\mathbf{V}_{\alpha
_{\Vert \cdot \Vert _{\alpha }}}),$ for $\beta \in S(\alpha )$ with $\beta
\notin \mathcal{L}(T_{D})$ and $1\leq i_{\alpha }\leq r_{\alpha }.$ Lemma~%
\ref{direct_sum_G} says us that $\Pi _{i_{\alpha }}(\mathcal{H}_{\alpha }(%
\mathbf{v}))\in \mathbb{G}(\mathbf{V}_{\alpha _{\Vert \cdot \Vert _{\alpha
}}})$ for $1\leq i_{\alpha }\leq r_{\alpha }.$ From Lemma~\ref{two_faces}
and Lemma~\ref{direct_sum_G} we obtain that $f_{D,\alpha }(\mathcal{H}%
_{\alpha }(\mathbf{v}))\in \mathbb{G}(\mathbf{V}_{D_{\Vert \cdot \Vert
_{D}}}).$ Also by Lemma~\ref{direct_sum_G}, we have $\mathbf{Z}^{(D)}(%
\mathbf{v})\in \mathbb{G}(\mathbf{V}_{D_{\Vert \cdot \Vert _{D}}})$, that proves the theorem.
\end{proof}

\begin{example}
Let us recall the topological tensor spaces introduced in Example \ref{Bsp HNp}%
. Let $I_{j}\subset \mathbb{R}$ $\left( 1\leq j\leq d\right) $ and $1\leq
p<\infty.$ Given tree $T_D,$ let $\mathbf{I}_{\alpha}:= 
\mathop{\mathchoice{\raise-0.22em\hbox{\huge
$\times$}} {\raise-0.05em\hbox{\Large $\times$}}{\hbox{\large
$\times$}}{\times}}_{j\in \alpha} I_j$ for $\alpha \in T_D$. Hence $L^p(\mathbf{I}_{\alpha})$
is a tensor Banach space for all $\alpha \in T_D.$ In this example we denote
the usual norm of $L^p(\mathbf{I}_{\alpha})$ by $\|\cdot\|_{\alpha,p}.$
Since $\|\cdot\|_{\alpha,p}$ is a reasonable crossnorm (see Example 4.72 in 
\cite{Hackbusch}), then \eqref{tree_injective_norm} holds for all $\alpha
\in T_D.$ From Theorem \ref{closed_linear_subspace} we obtain that $\mathcal{%
FT}_{\mathfrak{r}}\left(\left._a \bigotimes_{j=1}^d L^p(I_j)\right.\right)$
is an immersed submanifold of $L^p(\mathbf{I}_D).$
\end{example}

\begin{example}
Now, we return to Example~\ref{example_BM1}. From Example 4.42 in \cite%
{Hackbusch} we know that the norm \newline
$\|\cdot\|_{(0,1),p}$ is a crossnorm on $H^{1,p}(I_1) \otimes_a
H^{1,p}(I_2), $ and hence it is not weaker than the injective norm. In
consequence, from Theorem~\ref{closed_linear_subspace}, we obtain that $%
\mathcal{FT}_{\mathfrak{r}}(H^{1,p}(I_1) \otimes_a H^{1,p}(I_2))$ is an
immersed submanifold in $H^{1,p}(I_1) \otimes_{\|\cdot\|_{(0,1),p}}
H^{1,p}(I_2).$
\end{example}

Since in a reflexive Banach space every closed linear subspace is proximinal
(see p. 61 in \cite{Floret}), we have the following corollary.

\begin{corollary}
\label{approximation_corollary} Let $\{\mathbf{V}_{\alpha _{\Vert \cdot
\Vert _{\alpha }}}\}_{\alpha \in T_{D}\setminus \{D\}}$ be a representation
of a reflexive tensor Banach space $\mathbf{V}_{D_{\Vert \cdot \Vert
_{D}}}=\left. _{\Vert \cdot \Vert _{D}}\bigotimes_{j\in D}V_{j}\right.$ in
the topological tree-based format and assume that \eqref{tree_injective_norm}%
 holds. Let $\mathbf{v}\in \mathcal{FT}_{\mathfrak{r}}(\mathbf{V}_{D}),$
then for each $\dot{\mathbf{u}}\in \mathbf{V}_{D_{\Vert \cdot \Vert _{D}}}$
there exists $\dot{\mathbf{v}}_{best}\in \mathbf{Z}^{(D)}(\mathbf{v})$ such
that 
\begin{equation*}
\Vert \dot{\mathbf{u}}-\dot{\mathbf{v}}_{best}\Vert =\min_{\dot{\mathbf{v}}%
\in \mathbf{Z}^{(D)}(\mathbf{v})}\Vert \dot{\mathbf{u}}-\dot{\mathbf{v}}%
\Vert .
\end{equation*}
\end{corollary}

Using the standard inclusion map $\mathfrak{i}:\mathcal{FT}_{\leq \mathfrak{r%
}}(\mathbf{V}_{D})\rightarrow \mathbf{V}_{D_{\Vert \cdot \Vert _{D}}}$ the
following result can be shown.

\begin{corollary}
Let $\{\mathbf{V}_{\alpha _{\Vert \cdot \Vert _{\alpha }}}\}_{\alpha \in
T_{D}\setminus \{D\}}$ be a representation of a tensor Banach space $\mathbf{%
V}_{D_{\Vert \cdot \Vert _{D}}}=\left. _{\Vert \cdot \Vert
_{D}}\bigotimes_{j\in D}V_{j}\right. ,$ in the topological tree-based format
and assume that \eqref{tree_injective_norm} holds. Then $\mathcal{FT}_{\leq 
\mathfrak{r}}(\mathbf{V}_{D})$ is an immersed submanifold of $\mathbf{V}%
_{D_{\Vert \cdot \Vert _{D}}}.$
\end{corollary}

\section{On the Dirac--Frenkel variational principle on tensor Banach spaces}

\label{Dirac_Frenkel}

\subsection{Model Reduction in tensor Banach spaces}

In this section we consider the abstract ordinary differential equation in a
reflexive tensor Banach space $\mathbf{V}_{D_{\Vert \cdot \Vert _{D}}},$
given by 
\begin{align}
\dot{\mathbf{u}}(t)& =\mathbf{F}(t,\mathbf{u}(t)),\text{ for }t\geq 0
\label{BODE1} \\
\mathbf{u}(0)& =\mathbf{u}_{0},  \label{BODE2}
\end{align}%
where we assume $\mathbf{u}_{0}\neq \mathbf{0}$ and $\mathbf{F}:[0,\infty
)\times \mathbf{V}_{D_{\Vert \cdot \Vert _{D}}}\longrightarrow \mathbf{V}%
_{D_{\Vert \cdot \Vert _{D}}}$ satisfying the usual conditions 
to have existence and unicity of solutions. Let $\{%
\mathbf{V}_{\alpha _{\Vert \cdot \Vert _{\alpha }}}\}_{\alpha \in
T_{D}\setminus \{D\}}$ be a representation of $\mathbf{V}_{D_{\Vert \cdot
\Vert _{D}}}=\left. _{\Vert \cdot \Vert _{D}}\bigotimes_{j\in D}V_{j}\right. 
$ in the topological tree-based format and assume that (\ref%
{tree_injective_norm}) holds. As usual we will consider $\mathbf{V}%
_{D}=\left. _{a}\bigotimes_{j\in D}V_{j}\right. .$ We want to approximate $%
\mathbf{u}(t),$ for $t\in I:=(0,\varepsilon )$ for some $\varepsilon >0,$ by
a differentiable curve $t\mapsto \mathbf{v}_{r}(t)$ from $I$ to $\mathcal{FT}%
_{\mathfrak{r}}(\mathbf{V}_{D}),$ where $\mathfrak{r}\in \mathbb{N}^{T_{D}}$
is such that $\mathbf{v}_{r}(0)=\mathbf{u}(0)=\mathbf{u}_{0}\in \mathcal{FT}%
_{\mathfrak{r}}(\mathbf{V}_{D}).$

\bigskip

Our main goal is to construct a Reduced Order Model of \eqref{BODE1}--\eqref%
{BODE2} over the Banach manifold $\mathcal{FT}_{\mathfrak{r}}(\mathbf{V}%
_{D}).$ Since $\mathbf{F}(t,\mathbf{v}_{r}(t))\in \mathbf{V}_{D_{\Vert
\cdot \Vert _{D}}},$ for each $t\in I,$ and $\mathbf{Z}^{(D)}(\mathbf{v}%
_{r}(t))$ is a closed linear subspace in $\mathbf{V}_{D_{\Vert \cdot \Vert
_{D}}},$ we have the existence of a $\dot{\mathbf{v}}_{r}(t)\in \mathbf{Z}%
^{(D)}(\mathbf{v}_{r}(t))$ such that 
\begin{equation*}
\Vert \dot{\mathbf{v}}_{r}(t)-\mathbf{F}(t,\mathbf{v}_{r}(t))\Vert
_{D}=\min_{\dot{\mathbf{v}}(t)\in \mathbf{Z}^{(D)}(\mathbf{v}_{r}(t))}\Vert 
\dot{\mathbf{v}}(t)-\mathbf{F}(t,\mathbf{v}_{r}(t))\Vert _{D}.
\end{equation*}
It is well known that, if $\mathbf{V}_{D_{\Vert \cdot \Vert _{D}}}$ is a
Hilbert space, then $\dot{\mathbf{v}}_{r}(t)=\mathcal{P}_{\mathbf{v}_{r}(t)}(%
\mathbf{F}(t,\mathbf{v}_{r}(t))),$ where 
\begin{equation*}
\mathcal{P}_{\mathbf{v}_{r}(t)}=\mathcal{P}_{\mathbf{Z}^{(D)}(\mathbf{v}%
_{r}(t))\oplus \mathbf{Z}^{(D)}(\mathbf{v}_{r}(t))^{\bot }}
\end{equation*}%
is called the \emph{metric projection.} It has the following important
property: $\dot{\mathbf{v}}_{r}(t)=\mathcal{P}_{\mathbf{v}_{r}(t)}(\mathbf{F}%
(t,\mathbf{v}_{r}(t)))$ if and only if 
\begin{equation*}
\langle \dot{\mathbf{v}}_{r}(t)-\mathbf{F}(t,\mathbf{v}_{r}(t)),\dot{\mathbf{%
v}}(t)\rangle _{D}=0\text{ for all }\dot{\mathbf{v}}(t)\in \mathbf{Z}^{(D)}(%
\mathbf{v}_{r}(t)).
\end{equation*}

The concept of a metric projection can be extended to the Banach space
setting. To this end we recall the following definitions. A Banach space $X$
with norm $\Vert \cdot \Vert $ is said to be \emph{strictly convex} if $%
\Vert x+y\Vert /2<1$ for all $x,y\in X$ with $\Vert x\Vert =\Vert y\Vert =1$
and $x\neq y.$ It is \emph{uniformly convex} if $\lim_{n\rightarrow \infty
}\Vert x_{n}-y_{n}\Vert =0$ for any two sequences $\{x_{n}\}_{n\in \mathbb{N}%
}$ and $\{y_{n}\}_{n\in \mathbb{N}}$ such that $\Vert x_{n}\Vert =\Vert
y_{n}\Vert =1$ and $\lim_{n\rightarrow \infty }\Vert x_{n}+y_{n}\Vert /2=1.$
It is known that a uniformly convex Banach space is reflexive and strictly
convex. A Banach space $X$ is said to be \emph{smooth} if the limit 
\begin{equation*}
\lim_{t\rightarrow 0}\frac{\Vert x+ty\Vert -\Vert x\Vert }{t}
\end{equation*}%
exists for all $x,y\in U:=\{z\in X:\Vert z\Vert =1\}.$ Finally, a Banach
space $X$ is said to be \emph{uniformly smooth} if its modulus of smoothness 
\begin{equation*}
\rho (\tau )=\sup_{x,y\in U}\left\{ \frac{\Vert x+\tau y\Vert +\Vert x-\tau
y\Vert }{2}-1\right\} ,\,\tau >0,
\end{equation*}%
satisfies the condition $\lim_{\tau \rightarrow 0}\rho (\tau )=0.$ In
uniformly smooth spaces, and only in such spaces, the norm is uniformly Fr%
\'{e}chet differentiable. A uniformly smooth Banach space is smooth. The
converse is true if the Banach space is finite dimensional. It is known that
the space $L^{p}$ $(1<p<\infty )$ is a uniformly convex and uniformly smooth
Banach space.

\bigskip

Let $\langle \cdot ,\cdot \rangle :X\times X^{\ast }\longrightarrow \mathbb{R%
}$ denote the duality pairing, i.e.,%
\begin{equation*}
\langle x,f\rangle :=f(x).
\end{equation*}%
The normalised duality mapping $J:X\longrightarrow 2^{X^{\ast }}$ is defined
by 
\begin{equation*}
J(x):=\{f\in X^{\ast }:\langle x,f\rangle =\Vert x\Vert ^{2}=(\Vert f\Vert
^{\ast })^{2}\},
\end{equation*}%
and it has the following properties (see \cite%
{Alber}):

\begin{enumerate}
\item[(a)] If $X$ is smooth, the map $J$ is single-valued;

\item[(b)] if $X$ is smooth, then $J$ is norm--to--weak$^{\ast}$ continuous;

\item[(c)] if $X$ is uniformly smooth, then $J$ is uniformly norm--to--norm
continuous on each bounded subset of $X.$
\end{enumerate}

\begin{remark}
Notice that, in a Hilbert space and after identifying $X$ with $X^*,$ 
it can be shown (see Proposition~4.8(i) in \cite{Cioranescu}) 
that the normalised duality mapping is the identity operator. 
\end{remark}

Let $\{\mathbf{V}_{\alpha _{\Vert \cdot \Vert _{\alpha }}}\}_{\alpha \in
T_{D}\setminus \{D\}}$ be a representation of a reflexive and strictly convex
tensor Banach space $\mathbf{V}_{D_{\Vert \cdot \Vert _{D}}}=\left. _{\Vert
\cdot \Vert _{D}}\bigotimes_{j\in D}V_{j}\right. ,$ in the topological 
tree-based format and assume that \eqref{tree_injective_norm} holds. For $%
\mathbf{F}(t,\mathbf{v}_{r}(t))$ in $\mathbf{V}_{D_{\Vert \cdot \Vert _{D}}},
$ with a fixed $t\in I,$ it is known that the set 
\begin{equation*}
\left\{ \dot{\mathbf{v}}_{r}(t):\Vert \dot{\mathbf{v}}_{r}(t)-\mathbf{F}(t,%
\mathbf{v}_{r}(t))\Vert _{D}=\min_{\dot{\mathbf{v}}(t)\in \mathbf{Z}^{(D)}(%
\mathbf{v}_{r}(t))}\Vert \dot{\mathbf{v}}(t)-\mathbf{F}(t,\mathbf{v}%
_{r}(t))\Vert _{D}\right\}
\end{equation*}%
is always a singleton. Let $\mathcal{P}_{\mathbf{v}_{r}(t)}$ be the mapping
from $\mathbf{V}_{D_{\Vert \cdot \Vert _{D}}}$ onto $\mathbf{Z}^{(D)}(\mathbf{v%
}_{r}(t))$ defined by $\dot{\mathbf{v}}_{r}(t):=\mathcal{P}_{\mathbf{v}%
_{r}(t)}(\mathbf{F}(t,\mathbf{v}_{r}(t)))$ if and only if 
\begin{equation*}
\Vert \dot{\mathbf{v}}_{r}(t)-\mathbf{F}(t,\mathbf{v}_{r}(t))\Vert
_{D}=\min_{\dot{\mathbf{v}}(t)\in \mathbf{Z}^{(D)}(\mathbf{v}_{r}(t))}\Vert 
\dot{\mathbf{v}}(t)-\mathbf{F}(t,\mathbf{v}_{r}(t))\Vert _{D}.
\end{equation*}%
It is also called \emph{the metric projection.} The classical
characterisation of the metric projection allows us to state the next result.

\begin{theorem}
\label{th_metric_projection} Let $\{\mathbf{V}_{\alpha _{\Vert \cdot \Vert
_{\alpha }}}\}_{\alpha \in T_{D}\setminus \{D\}}$ be a representation of
reflexive and strictly convex tensor Banach space $\mathbf{V}_{D_{\Vert
\cdot \Vert _{D}}}=\left. _{\Vert \cdot \Vert _{D}}\bigotimes_{j\in
D}V_{j}\right. $ in the topological tree-based format and assume that (\ref%
{tree_injective_norm}) holds. Then for each $t\in I$ we have 
\begin{equation*}
\dot{\mathbf{v}}_{r}(t)=\mathcal{P}_{\mathbf{v}_{r}(t)}(\mathbf{F}(t,\mathbf{%
v}_{r}(t)))
\end{equation*}%
if and only if 
\begin{equation*}
\langle \dot{\mathbf{v}}_{r}(t)-\dot{\mathbf{v}}(t),J(\mathbf{F}(t,\mathbf{v}%
_{r}(t))-\dot{\mathbf{v}}_{r}(t))\rangle \geq 0\text{ for all }\dot{\mathbf{v%
}}(t)\in \mathbf{Z}^{(D)}(\mathbf{v}_{r}(t)).
\end{equation*}
\end{theorem}

\bigskip

An alternative approach is the use of the so-called \emph{generalised
projection operator} (see also \cite{Alber}). To formulate this, we will use
the following framework. Let $T_D$ a given tree and assume that for each $%
\alpha \in T_D$ we have a Banach space $\mathbf{V}_{\alpha_{\|\cdot\|_{%
\alpha}}},$ such that \eqref{tree_injective_norm} holds and where $\mathbf{V}%
_{D_{\Vert \cdot \Vert _{D}}}$ is a reflexive, strictly convex and smooth
tensor Banach space. Following \cite{KamiTaka}, we can define a function $%
\phi :\mathbf{V}_{D_{\Vert \cdot \Vert _{D}}}\times \mathbf{V}_{D_{\Vert
\cdot \Vert _{D}}}\longrightarrow \mathbb{R}$ by 
\begin{equation*}
\phi (\mathbf{u},\mathbf{v})=\Vert \mathbf{u}\Vert _{D}^{2}-2\langle \mathbf{%
u},J(\mathbf{v})\rangle +\Vert \mathbf{v}\Vert _{D}^{2},
\end{equation*}%
where $\langle \cdot ,\cdot \rangle $ denotes the duality pairing and $J$ is the
normalised duality mapping. It is known that the set 
\begin{equation*}
\left\{ \dot{\mathbf{v}}_{r}(t):\phi (\dot{\mathbf{v}}_{r}(t),\mathbf{F}(t,%
\mathbf{v}_{r}(t)))=\min_{\dot{\mathbf{v}}(t)\in \mathbf{Z}^{(D)}(\mathbf{v}%
_{r}(t))}\phi (\dot{\mathbf{v}}(t),\mathbf{F}(t,\mathbf{v}_{r}(t)))\right\}
\end{equation*}%
is always a singleton. It allows us to define a map $\Pi _{\mathbf{v}%
_{r}(t)}:\mathbf{V}_{D_{\Vert \cdot \Vert _{D}}}\longrightarrow \mathbf{Z}%
^{(D)}(\mathbf{v}_{r}(t)) $ by $\dot{\mathbf{v}}_{r}(t):=\Pi _{\mathbf{v}%
_{r}(t)}(\mathbf{F}(t,\mathbf{v}_{r}(t)))$ if and only if 
\begin{equation*}
\phi (\dot{\mathbf{v}}_{r}(t),\mathbf{F}(t,\mathbf{v}_{r}(t)))=\min_{\dot{%
\mathbf{v}}(t)\in \mathbf{Z}^{(D)}(\mathbf{v}_{r}(t))}\phi (\dot{\mathbf{v}}%
(t),\mathbf{F}(t,\mathbf{v}_{r}(t))).
\end{equation*}%
The map $\Pi _{\mathbf{v}_{r}(t)}$ is called \emph{the generalised
projection.} It coincides with the metric projection when $\mathbf{V}%
_{D_{\Vert \cdot \Vert _{D}}}$ is a Hilbert space.

\begin{remark}
We point out that, in general, the operators $\mathcal{P}_{\mathbf{v}_r(t)} $
and $\Pi_{\mathbf{v}_r(t)}$ are nonlinear in Banach (not Hilbert) spaces.
\end{remark}

Again, a classical characterisation of the generalised projection gives us
the following theorem.

\begin{theorem}
\label{th_generalized_projection} Let $\{\mathbf{V}_{\alpha _{\Vert \cdot
\Vert _{\alpha }}}\}_{\alpha \in T_{D}\setminus \{D\}}$ be a representation
of reflexive, strictly convex and smooth tensor Banach space $\mathbf{V}%
_{D_{\Vert \cdot \Vert _{D}}}=\left. _{\Vert \cdot \Vert
_{D}}\bigotimes_{j\in D}V_{j}\right. $ in the topological tree-based format
and assume that \eqref{tree_injective_norm} holds. Then for each $t\in I$ we
have 
\begin{equation*}
\dot{\mathbf{v}}_{r}(t)=\Pi _{\mathbf{v}_{r}(t)}(\mathbf{F}(t,\mathbf{v}%
_{r}(t)))
\end{equation*}%
if and only if 
\begin{equation*}
\langle \dot{\mathbf{v}}_{r}(t)-\dot{\mathbf{v}}(t),J(\mathbf{F}(t,\mathbf{v}%
_{r}(t)))-J(\dot{\mathbf{v}}_{r}(t))\rangle \geq 0\text{ for all }\dot{%
\mathbf{v}}(t)\in \mathbf{Z}^{(D)}(\mathbf{v}_{r}(t)).
\end{equation*}
\end{theorem}

\subsection{The time--dependent Hartree method}

Let $\left\langle \cdot,\cdot\right\rangle _{j}$ be a scalar product defined
on $V_{j}$ $\left( 1\leq j\leq d\right) $, i.e., $V_{j}$ is a pre-Hilbert
space. Then $\mathbf{V}=\left. _{a}\bigotimes_{j=1}^{d}V_{j}\right. $ is
again a pre-Hilbert space with a scalar product which is defined for
elementary tensors $\mathbf{v}=\bigotimes_{j=1}^{d}v^{(j)}$ and $\mathbf{w}%
=\bigotimes_{j=1}^{d}w^{(j)}$ by%
\begin{equation}
\left\langle \mathbf{v,w}\right\rangle =\left\langle
\bigotimes_{j=1}^{d}v^{(j)},\bigotimes_{j=1}^{d}w^{(j)}\right\rangle
:=\prod_{j=1}^{d}\left\langle v^{(j)},w^{(j)}\right\rangle _{j}\qquad\text{%
for all }v^{(j)},w^{(j)}\in V_{j}.
\label{(Skalarprodukt fur Elementarprodukte}
\end{equation}
This bilinear form has a unique extension $\left\langle \cdot,\cdot
\right\rangle :\mathbf{V}\times\mathbf{V}\rightarrow\mathbb{R}.$ One
verifies that $\left\langle \cdot,\cdot\right\rangle $ is a scalar product,
called the \emph{induced scalar product}. Let $\mathbf{V}$ be equipped with
the norm $\left\Vert \cdot\right\Vert $ corresponding to the induced scalar
product $\left\langle \cdot,\cdot\right\rangle .$ As usual, the Hilbert
tensor space $\mathbf{V}_{\left\Vert \cdot\right\Vert }=\left. _{\left\Vert
\cdot \right\Vert }\bigotimes_{j=1}^{d}V_{j}\right. $ is the completion of $%
\mathbf{V}$ with respect to $\left\Vert \cdot\right\Vert $. Since the norm $%
\left\Vert \cdot\right\Vert $ is derived via \eqref{(Skalarprodukt fur
Elementarprodukte}, it is easy to see that $\left\Vert \cdot\right\Vert $
is a reasonable and even uniform crossnorm.

Let us consider in $\mathbf{V}_{\|\cdot\|}$ a flow generated by a densely
defined operator $A \in L(\mathbf{V}_{\|\cdot\|},\mathbf{V}_{\|\cdot\|}).$
More precisely, there exists a collection of bijective maps $\boldsymbol{%
\varphi}_t:\mathcal{D}(A) \longrightarrow \mathcal{D}(A),$ here $\mathcal{D}%
(A)$ denotes the domain of $A,$ satisfying

\begin{enumerate}
\item[(i)] $\boldsymbol{\varphi}_0 = \mathbf{id},$

\item[(ii)] $\boldsymbol{\varphi}_{t+s} = \boldsymbol{\varphi}_t \circ 
\boldsymbol{\varphi}_s,$ and

\item[(iii)] for $\mathbf{u}_0 \in \mathcal{D}(A),$ the map $t \mapsto 
\boldsymbol{\varphi}_t$ is differentiable as a curve in $\mathbf{V}%
_{\|\cdot\|},$ and $\mathbf{u}(t):= \boldsymbol{\varphi}_t(\mathbf{u}_0)$
satisfies 
\begin{align*}
\dot{\mathbf{u}} & = A \mathbf{u}, \\
\mathbf{u}(0) & = \mathbf{u}_0.
\end{align*}
\end{enumerate}

In this framework we want to study the approximation of a solution $\mathbf{u%
}(t)=\boldsymbol{\varphi }_{t}(\mathbf{u}_{0})\in \mathbf{V}_{\Vert \cdot
\Vert }$ by a curve $\mathbf{v}_{r}(t):=\lambda (t)\otimes
_{j=1}^{d}v_{j}(t) $ in the Hilbert manifold $\mathcal{M}_{(1,\ldots ,1)}(%
\mathbf{V}),$ also called in \cite{Lubish} the \emph{Hartree manifold.} The
time--dependent Hartree method consists in the use of the Dirac--Frenkel
variational principle on the Hartree manifold. More precisely, we want to
solve the following Reduced Order Model: 
\begin{align*}
\dot{\mathbf{v}}_{r}(t)& =\mathcal{P}_{\mathbf{v}_{r}(t)}(A\mathbf{v}_{r}(t))%
\text{ for }t\in I, \\
\mathbf{v}_{r}(0)& =\mathbf{v}_{0},
\end{align*}%
with $\mathbf{v}_{0}=\lambda _{0}\otimes _{j=1}^{d}v_{0}^{(j)}\in \mathcal{M}%
_{(1,\ldots ,1)}(\mathbf{V})$ being an approximation of $\mathbf{u}_{0}$%
\footnote{%
Indeed, $\mathbf{v}_{0}$ can be chosen as the best approximation of $\mathbf{%
u}_{0}$ in $\mathcal{M}_{(1,\ldots ,1)}(\mathbf{V})$ because $\mathcal{M}%
_{(1,\ldots ,1)}(\mathbf{V})=\mathcal{T}_{(1,\ldots ,1)}(\mathbf{V}%
)\setminus \{\mathbf{0}\}.$}. By using the characterisation of the metric
projection in a Hilbert space, for each $t>0$ we would like to find $\dot{%
\mathbf{v}}_{r}(t)\in \mathrm{T}_{\mathbf{v}_{r}(t)}\mathfrak{i}\left( 
\mathbb{T}_{\mathbf{v}_{r}(t)}(\mathcal{M}_{(1,\ldots ,1)}(\mathbf{V}%
))\right) $ such that 
\begin{align}
\langle \dot{\mathbf{v}}_{r}(t)-A\mathbf{v}_{r}(t),\dot{\mathbf{v}}%
(t)\rangle & =0\text{ for all }\dot{\mathbf{v}}(t)\in \mathrm{T}_{\mathbf{v}%
_{r}(t)}\mathfrak{i}\left( \mathbb{T}_{\mathbf{v}_{r}(t)}(\mathcal{M}%
_{(1,\ldots ,1)}(\mathbf{V}))\right),  \label{eq:11}
\end{align}
\begin{align*}
\mathbf{v}_{r}(0)& =\mathbf{v}_{0}=\lambda _{0}\otimes _{j=1}^{d}v_{0}^{(j)},
\end{align*}%
and where, without loss of generality, we may assume $\Vert v_{0}^{(j)}\Vert
_{j}=1$ for $1\leq j\leq d.$ A first result is the following Lemma.

\begin{lemma}
\label{previous_HF} Let $\mathbf{v}\in \mathcal{C}^{1}(I,\mathcal{U}(\mathbf{%
v}_{0})),$ where $\mathbf{v}(0)=\mathbf{v}_{0}\in \mathcal{M}_{(1,\ldots
,1)}(\mathbf{V})$ and $(\mathcal{U}(\mathbf{v}_{0}),\Theta _{\mathbf{v}%
_{0}}) $ is a local chart for $\mathbf{v}_{0}$ in $\mathcal{M}_{(1,\ldots
,1)}(\mathbf{V}).$ Assume that $\mathbf{v}$ is also a $\mathcal{C}^{1}$%
-morphism between the manifolds $I\subset \mathbb{R}$ and $\mathcal{U}(%
\mathbf{v}_{0})\subset \mathcal{M}_{(1,\ldots ,1)}(\mathbf{V})$ such that $%
\mathbf{v}(t)=\lambda (t)\bigotimes_{j=1}^{d}v_{j}(t)$ for some $\lambda \in 
\mathcal{C}^{1}(I,\mathbb{R})$ and $v_{j}\in \mathcal{C}^{1}(I,V_{j})$ for $%
1\leq j\leq d.$ Then 
\begin{equation}  \label{curve_derivative}
\dot{\mathbf{v}}(t)=\dot{\lambda}(t)\bigotimes_{j=1}^{d}v_{j}(t)+\lambda
(t)\sum_{j=1}^{d}\dot{v}_{j}(t)\otimes \bigotimes_{k\neq j}v_{k}(t)=\mathrm{T%
}_{\mathbf{v}(t)}\mathfrak{i}(\mathrm{T}_{t}\mathbf{v}(1)).
\end{equation}%
Moreover, if $v_{j}(t)\in \mathbb{S}_{V_{j}},$ i.e., $\Vert v_{j}(t)\Vert
_{j}=1,$ for $t\in I$ and $1\leq j\leq d,$ then $\dot{v}_{j}(t)\in \mathbb{T}%
_{v_{j}(t)}(\mathbb{S}_{V_{j}})$ for $t\in I$ and $1\leq j\leq d.$
\end{lemma}

\begin{proof}
First at all, we recall that by the construction of $\mathcal{U}(\mathbf{v}%
_{0})$ it follows that $W_{j}^{\min }(\mathbf{v}_{0})=W_{j}^{\min }(\mathbf{v%
}(t))$ and that $U_{j}^{\min }(\mathbf{v}_{0})=\mathrm{span}\{v_{0}^{(j)}\}$
is linearly isomorphic to $U_{j}^{\min }(\mathbf{v}(t))$ for all $t\in I$
and $1\leq j\leq d.$ Assume $\Theta _{\mathbf{v}_{0}}(\mathbf{v}%
(t))=(\lambda (t),L_{1}(t),\ldots ,L_{d}(t)),$ i.e., 
\begin{equation*}
\mathbf{v}(t):=\lambda (t)\bigotimes_{j=1}^{d}\left( id_{j}+L_{j}(t)\right)
(v_{0}^{(j)}),
\end{equation*}%
where $\lambda \in \mathcal{C}^{1}(I,\mathbb{R}\setminus \{0\}),$ $L_{j}\in 
\mathcal{C}^{1}(I,\mathcal{L}(U_{j}^{\min }(\mathbf{v}_{0}),W_{j}^{\min }(%
\mathbf{v}_{0})))$ and $(id_{j}+L_{j}(t))(v_{0}^{(j)})\in U_{j}^{\min }(%
\mathbf{v}(t))$ for $1\leq j\leq d.$ We point out that the linear map $%
\mathrm{T}_{t}\mathbf{v}:\mathbb{R}\rightarrow \mathbb{T}_{\mathbf{v}(t)}(%
\mathcal{M}_{(1,\ldots ,1)}(\mathbf{V}))$ is characterised by 
\begin{equation}
\mathrm{T}_{t}\mathbf{v}(1)=(\Theta _{\mathbf{v}_{0}}\circ \mathbf{v}%
)^{\prime }(t)=(\dot{\lambda}(t),\dot{L}_{1}(t),\ldots ,\dot{L}_{d}(t)).
\label{tangent_derivative}
\end{equation}%
Since $L_{j}\in \mathcal{C}^{1}(I,\mathcal{L}(U_{j}^{\min }(\mathbf{v}%
_{0}),W_{j}^{\min }(\mathbf{v}_{0})))$ then $\dot{L}_{j}\in \mathcal{C}%
^{0}(I,\mathcal{L}(U_{j}^{\min }(\mathbf{v}_{0}),W_{j}^{\min }(\mathbf{v}%
_{0}))).$ Observe that $U_{j}^{\min }(\mathbf{v}_{0})$ and $U_{j}^{\min }(%
\mathbf{v}(t))$ have $W_{j}^{\min }(\mathbf{v}_{0})$ as a common complement.
From Lemma~\ref{Char_Projections} we know that 
\begin{equation*}
P_{U_{j}^{\min }(\mathbf{v}_{0})\oplus W_{j}^{\min }(\mathbf{v}%
_{0})}|_{U_{j}^{\min }(\mathbf{v}(t))}:U_{j}^{\min }(\mathbf{v}%
(t))\longrightarrow U_{j}^{\min }(\mathbf{v}_{0})
\end{equation*}%
is a linear isomorphism. We can write 
\begin{equation*}
L_{j}(t)=L_{j}(t)P_{U_{j}^{\min }(\mathbf{v}_{0})\oplus W_{j}^{\min }(%
\mathbf{v}_{0})}\text{ and }\dot{L}_{j}(t)=\dot{L}_{j}(t)P_{U_{j}^{\min }(%
\mathbf{v}_{0})\oplus W_{j}^{\min }(\mathbf{v}_{0})},
\end{equation*}%
and then in \eqref{tangent_derivative} we identify $\dot{L}_{j}(t)\in 
\mathcal{L}(U_{j}^{\min }(\mathbf{v}_{0}),W_{j}^{\min }(\mathbf{v}_{0})))$
with 
\begin{equation*}
\dot{L}_{j}(t)P_{U_{j}^{\min }(\mathbf{v}_{0})\oplus W_{j}^{\min }(\mathbf{v}%
_{0})}|_{U_{j}^{\min }(\mathbf{v}(t))}\in \mathcal{L}(U_{j}^{\min }(\mathbf{v%
}(t)),W_{j}^{\min }(\mathbf{v}_{0}))).
\end{equation*}%
Introduce $v_{j}(t):=(id_{j}+L_{j}(t))(v_{0}^{(j)})$ for $1\leq j\leq d.$
Then 
\begin{equation*}
\dot{L}_{j}(t)(v_{j}(t))=\dot{L}_{j}(t)P_{U_{j}^{\min }(\mathbf{v}%
_{0})\oplus W_{j}^{\min }(\mathbf{v}_{0})}|_{U_{j}^{\min }(\mathbf{v}%
(t))}(v_{0}^{(j)}+L_{j}(t)(v_{0}^{(j)}))=\dot{L}_{j}(t)(v_{0}^{(j)})
\end{equation*}%
holds for all $t\in I$ and $1\leq j\leq d.$ Hence 
\begin{equation}
\dot{v}_{j}(t)=\dot{L}_{j}(t)(v_{0}^{(j)})=\dot{L}_{j}(t)(v_{j}(t))
\label{derivative}
\end{equation}%
holds for all $t\in I$ and $1\leq j\leq d.$ From Lemma \ref%
{characterization_tangent_map}(b) and \eqref{tangent_derivative}, we have 
\begin{equation*}
\mathrm{T}_{\mathbf{v}(t)}\mathfrak{i}(\mathrm{T}_{t}\mathbf{v}(1))=\dot{%
\lambda}(t)\bigotimes_{j=1}^{d}v_{j}(t)+\lambda (t)\sum_{j=1}^{d}\dot{L}%
_{j}(t)(v_{j}(t))\otimes \bigotimes_{k\neq j}v_{k}(t),
\end{equation*}%
and, by using \eqref{derivative} for $\mathbf{v}(t)=\lambda
(t)\bigotimes_{j=1}^{d}v_{j}(t),$ we obtain \eqref{curve_derivative}.

To prove the second statement, recall that $U_{j}^{\min }(\mathbf{v}(t))=%
\mathrm{span}\,\{v_{j}(t)\}$ and $V_{j}=U_{j}^{\min }(\mathbf{v}(t))\oplus
W_{j}^{\min }(\mathbf{v}_{0})$ for $1\leq j\leq d.$ Then we consider 
\begin{equation*}
W_{j}^{\min }(\mathbf{v}_{0})=\mathrm{span}\,\{v_{j}(t)\}^{\bot }=\{u_{j}\in
V_{j}:\langle u_{j},v_{j}(t)\rangle _{j}=0\}\text{ for }1\leq j\leq d,
\end{equation*}%
and hence $\langle \dot{v}_{j}(t)),v_{j}(t)\rangle _{j}=0$ holds for $1\leq
j\leq d.$ From Remark~\ref{unit_sphere}, we have $(\dot{v}_{1}(t),\ldots ,%
\dot{v}_{d}(t))\in \mathcal{C}(I,%
\mathop{\mathchoice{\raise-0.22em\hbox{\huge
$\times$}} {\raise-0.05em\hbox{\Large $\times$}}{\hbox{\large
$\times$}}{\times}}_{j=1}^{d}\mathbb{T}_{v_{j}(t)}(\mathbb{S}_{V_{j}})),$
because $W_{j}^{\min }(\mathbf{v}_{0})=\mathbb{T}_{v_{j}(t)}(\mathbb{S}%
_{V_{j}})$ for $1\leq j\leq d.$
\end{proof}

\bigskip

Before stating the next result, we introduce for $\mathbf{v}_{r}(t)=\lambda
(t)\bigotimes_{j=1}^{d}v_{j}(t)$ the following time dependent bilinear forms 
\begin{equation*}
\mathrm{a}_{k}(t;\cdot ,\cdot ):V_{k}\times V_{k}\longrightarrow \mathbb{R},
\end{equation*}%
by 
\begin{equation*}
\mathrm{a}_{k}(t;z_{k},y_{k}):=\left\langle A\left( z_{k}\otimes
\bigotimes_{j\neq k}v_{j}(t)\right) ,\left( y_{k}\otimes \bigotimes_{j\neq
k}v_{j}(t)\right) \right\rangle
\end{equation*}%
for each $1\leq k\leq d.$ Now, we will show the next result (compare with
Theorem~3.1 in \cite{Lubish}).

\begin{theorem}[Time dependent Hartree method]
The solution $\mathbf{v}_{r}(t)=\lambda (t)\bigotimes_{j=1}^{d}v_{j}(t)$ for 
$(v_{1}(t),\ldots ,v_{d}(t))\in 
\mathop{\mathchoice{\raise-0.22em\hbox{\huge
$\times$}} {\raise-0.05em\hbox{\Large $\times$}}{\hbox{\large
$\times$}}{\times}}_{j=1}^{d}\mathbb{S}_{V_{j}}$ of 
\begin{align*}
\dot{\mathbf{v}}_{r}(t)& =\mathcal{P}_{\mathbf{v}_{r}(t)}(A\mathbf{v}_{r}(t))%
\text{ for }t\in I, \\
\mathbf{v}_{r}(0)& =\mathbf{v}_{0},
\end{align*}%
satisfies 
\begin{equation*}
\langle \dot{v}_{j}(t),\dot{w}_{j}(t)\rangle _{j}-\mathrm{a}_{j}(t;v_{j}(t),%
\dot{w}_{j}(t))=0\text{\quad for all }\dot{w}_{j}(t)\in \mathbb{T}%
_{v_{j}(t)}(\mathbb{S}_{V_{j}}),\quad 1\leq j\leq d,
\end{equation*}%
and 
\begin{equation*}
\lambda (t)=\lambda _{0}\exp \left( \int_{0}^{t}\left\langle A\left( \otimes
_{j=1}^{d}v_{j}(s)\right) ,\otimes _{j=1}^{d}v_{j}(s)\right\rangle ds\right)
.
\end{equation*}
\end{theorem}

\begin{proof}
From Lemma~\ref{previous_HF} we have $\mathbb{T}_{\mathbf{v}_{r}(t)}\left( 
\mathcal{M}_{(1,\ldots ,1)}(\mathbf{V})\right) =\mathbb{R}\times 
\mathop{\mathchoice{\raise-0.22em\hbox{\huge
$\times$}} {\raise-0.05em\hbox{\Large $\times$}}{\hbox{\large
$\times$}}{\times}}_{j=1}^{d}\mathbb{T}_{v_{j}(t)}(\mathbb{S}_{V_{j}}),$
Thus, for each $\dot{\mathbf{w}}(t)\in \mathrm{T}_{\mathbf{v}(t)}i\left( 
\mathbb{T}_{\mathbf{v}(t)}\left( \mathcal{M}_{(1,\ldots ,1)}(\mathbf{V}%
)\right) \right) $ there exists $(\dot{\beta}(t),\dot{w}_{1}(t),\ldots ,\dot{%
w}_{d}(t))\in \mathbb{R}\times 
\mathop{\mathchoice{\raise-0.22em\hbox{\huge $\times$}}
{\raise-0.05em\hbox{\Large $\times$}}{\hbox{\large $\times$}}{\times}}%
_{j=1}^{d}\mathbb{T}_{v_{j}(t)}(\mathbb{S}_{V_{j}}),$ such that 
\begin{equation*}
\dot{\mathbf{w}}(t)=\dot{\beta}(t)\bigotimes_{j=1}^{d}v_{j}(t)+\lambda
(t)\sum_{j=1}^{d}\dot{w}_{j}(t)\otimes \bigotimes_{k\neq j}v_{k}(t).
\end{equation*}%
Then \eqref{eq:11} holds if and only if 
\begin{equation*}
\left\langle \dot{\mathbf{v}}_{r}(t)-A\mathbf{v}_{r}(t),\dot{\beta}%
(t)\bigotimes_{j=1}^{d}v_{j}(t)+\lambda (t)\sum_{j=1}^{d}\dot{w}%
_{j}(t)\otimes \bigotimes_{k\neq j}v_{k}(t)\right\rangle =0
\end{equation*}%
for all $(\dot{\beta}(t),\dot{w}_{1}(t),\ldots ,\dot{w}_{d}(t))\in \mathbb{R}%
\times 
\mathop{\mathchoice{\raise-0.22em\hbox{\huge $\times$}}
{\raise-0.05em\hbox{\Large $\times$}}{\hbox{\large $\times$}}{\times}}%
_{j=1}^{d}\mathbb{T}_{v_{j}(t)}(\mathbb{S}_{V_{j}}).$ Then 
\begin{align*}
\dot{\lambda}(t)\dot{\beta}(t)+\lambda (t)^{2}\sum_{j=1}^{d}\left( \langle 
\dot{v}_{j}(t),\dot{w}_{j}(t)\rangle _{j}-\langle
A\bigotimes_{s=1}^{d}v_{s}(t),\dot{w}_{j}(t)\otimes \bigotimes_{k\neq
j}v_{k}(t)\rangle \right)& \\
-\lambda (t)\dot{\beta}(t)\langle
A\bigotimes_{j=1}^{d}v_{j}(t),\bigotimes_{j=1}^{d}v_{j}(t)\rangle &=0,
\end{align*}%
i.e., 
\begin{equation}  \label{HF5}
\begin{array}{l}
\dot{\beta}(t)\left( \dot{\lambda}(t)-\lambda (t)\langle
A\bigotimes_{j=1}^{d}v_{j}(t),\bigotimes_{j=1}^{d}v_{j}(t)\rangle \right) \\ 
+\lambda (t)^{2}\sum_{j=1}^{d}\left( \langle \dot{v}_{j}(t),\dot{w}%
_{j}(t)\rangle _{j}-\langle A\bigotimes_{s=1}^{d}v_{s}(t),\dot{w}%
_{j}(t)\otimes \bigotimes_{k\neq j}v_{k}(t)\rangle \right) =0%
\end{array}%
\end{equation}%
holds for all $\dot{\beta}(t)\in \mathbb{R}$ and $(\dot{w}_{1}(t),\ldots ,%
\dot{w}_{d}(t))\in 
\mathop{\mathchoice{\raise-0.22em\hbox{\huge $\times$}}
{\raise-0.05em\hbox{\Large $\times$}}{\hbox{\large $\times$}}{\times}}%
_{j=1}^{d}\mathbb{T}_{v_{j}(t)}(\mathbb{S}_{V_{j}}).$ If $\lambda (t)$
solves the differential equation 
\begin{align*}
\dot{\lambda}(t)& =\left\langle A\left( \otimes _{j=1}^{d}v_{j}(t)\right)
,\otimes _{j=1}^{d}v_{j}(t)\right\rangle \lambda (t) \\
\lambda (0)& =\lambda _{0},
\end{align*}%
i.e., 
\begin{equation*}
\lambda (t)=\lambda _{0}\exp \left( \int_{0}^{t}\left\langle A\left( \otimes
_{j=1}^{d}v_{j}(s)\right) ,\otimes _{j=1}^{d}v_{j}(s)\right\rangle ds\right)
,
\end{equation*}%
then the first term of \eqref{HF5} is equal to $0.$ Therefore, from (\ref%
{HF5}) we obtain that for all $j\in D$, 
\begin{equation*}
\langle \dot{v}_{j}(t),\dot{w}_{j}(t)\rangle _{j}-\langle
A\bigotimes_{s=1}^{d}v_{s}(t),\dot{w}_{j}(t)\otimes \bigotimes_{k\neq
j}v_{k}(t)\rangle =0,
\end{equation*}%
that is, 
\begin{equation*}
\langle \dot{v}_{j}(t),\dot{w}_{j}(t)\rangle _{j}-\mathrm{a}_{j}(t;v_{j}(t),%
\dot{w}_{j}(t))=0
\end{equation*}%
holds for all $\dot{w}_{j}(t)\in \mathbb{T}_{v_{j}(t)}(\mathbb{S}_{V_{j}}),$
and the theorem follows.
\end{proof}

\end{document}